\documentclass[12pt,letterpaper]{amsart}
\usepackage{euler, epic,eepic,latexsym, amssymb, amscd, amsfonts, xypic, url, color, epsfig, hyperref, comment}
\usepackage{enumitem}
\setlist[enumerate,1]{label={(\arabic*)}}
\setlist[enumerate,2]{label={(\alph*)}}
\input xy
\xyoption{all}
 \newlength{\baseunit}               % the basic unit length
 \newcount{\numlines}                % depth of picture (in number of lines)
\newtheorem{thmintro}{Theorem}

\newtheorem{tm}{Theorem}[section]
\newtheorem{pr}[tm]{Proposition}
\newtheorem{lm}[tm]{Lemma}
\newtheorem{co}[tm]{Corollary}
\newtheorem{constrn}[tm]{Construction}

\theoremstyle{definition}
\newtheorem{df}[tm]{Definition}
\newtheorem{hyp}[thmintro]{Hypothesis}
\newtheorem{dfintro}[thmintro]{Definition}

\theoremstyle{remark}
\newtheorem{rmk}[tm]{Remark}
\newtheorem{ex}[tm]{Example}

\newtheorem{setting}[tm]{Setting}

\newtheorem{exintro}[thmintro]{Example}

\newcommand{\cat}[1]{{\mathbf{#1}}}

\newcommand{\Nodes}{\operatorname{Nodes}}

\newcommand{\Aone}{{\mathbb{A}^{\!1}}}

\newcommand{\GW}{\mathrm{GW}}
\newcommand{\W}{\mathrm{W}}

\newcommand{\Laone}{{\mathrm{L}_{\Aone}}}

\newcommand{\Lnis}{\mathrm{L}_{\textrm{Nis}}}

\newcommand{\sGW}{\mathcal{GW}}

\newcommand{\dpl}{{\calD}}

\newcommand{\PP}{\mathbb{P}}
\newcommand{\Z}{\mathbb{Z}}

\newcommand{\kbar}{\overline{k}}

\newcommand{\R}{\mathbb{R}}

\newcommand{\C}{\mathbb{C}}

\newcommand{\Map}{\operatorname{Map}} % Internal mapping objects
\newcommand{\Hom}{\operatorname{Hom}} % Abelian group of homomorphisms

\newcommand{\Fun}{\operatorname{Fun}}  % Set of functors
\newcommand{\Aut}{\operatorname{Aut}} % Group of automorphisms

\newcommand{\Gal}{\operatorname{Gal}}

\newcommand{\Sym}{\operatorname{Sym}}

\newcommand{\Set}{\cat{Set}}
\newcommand{\Ab}{\cat{Ab}}

\newcommand{\Sch}{\cat{Sch}}

\newcommand{\Sm}{\cat{Sm}}

\newcommand{\bir}{{\operatorname{bir}}}
\newcommand{\odp}{{\operatorname{odp}}}
\newcommand{\unr}{{\operatorname{unr}}}
\newcommand{\Deg}{\text{deg}}
\newcommand{\good}{{\operatorname{good}}}

\newcommand\bbb[1]{\ensuremath{{\mathbb{#1}}}}

\newcommand{\A}{\mathbb{A}}
\newcommand{\Tr}{\operatorname{Tr}}
\newcommand{\ind}{\operatorname{ind}}

\newcommand{\cusp}{\mathrm{cusp}}
\newcommand{\tac}{\mathrm{tac}}

\newcommand{\Spec}{\operatorname{Spec}}

\newcommand{\Bl}{\operatorname{Bl}}

\newcommand{\mass}{\operatorname{mass}}
\newcommand{\Hess}{\operatorname{Hess}}
\newcommand{\Res}{\operatorname{Res}}
\newcommand{\disc}{\operatorname{disc}}

\newcommand{\rank}{\operatorname{Rank}}
\DeclareMathOperator{\Char}{char}

\newcommand{\op}{\operatorname{op}}

\DeclareMathOperator{\codim}{codim}

\newcommand{\M}{\bar{M}}
\DeclareMathOperator{\ev}{ev}
\newcommand{\Tev}{T\!\ev}
\DeclareMathOperator{\evb}{evb}
\DeclareMathOperator{\evi}{evi}
\newcommand{\po}{pt}
\DeclareMathOperator{\Int}{int}

\newcommand{\cO}{\mathcal{O}}

\newcommand{\LiftDeg}{\widetilde{D}}
\newcommand{\Liftpi}{\widetilde{\pi}}
\newcommand{\LiftS}{\widetilde{S}}
\newcommand{\LiftA}{\widetilde{A}}
\newcommand{\Liftev}{\widetilde{\ev}}
\newcommand{\Liftdpl}{{\widetilde{\calD}}}
\newcommand{\Liftsig}{\widetilde{\sigma}}
\newcommand{\LiftV}{\widetilde{V}}
\newcommand{\LiftL}{\widetilde{\calL}}
\newcommand{\Pic}{{Pic}}

\newcommand{\rhol}{\mathfrak{r}}

\newcommand{\calD}{{ \mathcal D}}
\newcommand{\calE}{{ \mathcal E}}
\newcommand{\calF}{{ \mathcal F}}

\newcommand{\calI}{{ \mathcal I}}

\newcommand{\calL}{{ \mathcal L}}

\newcommand{\calN}{{ \mathcal N}}
\newcommand{\calO}{{ \mathcal O}}
\newcommand{\calP}{ {\mathcal P} }

\newcommand{\calU}{{ \mathcal U}}

\newcommand{\calX}{{\mathcal X}}
\newcommand{\calY}{{\mathcal Y}}

\newcommand{\mfS}{\mathfrak{S}}

\newcommand{\hidden}[1]{\footnote{Hidden:  #1}}
\renewcommand{\hidden}[1]{}

\newcommand{\norm}[2]{N_{#1/#2}}
\newcommand{\id}{\mathrm{Id}}

\begin{document}

\pagestyle{plain}

\title{A quadratically enriched count of rational curves}

\author{Jesse Leo Kass}

\address{Current: J.~L.~Kass, Dept.~of Mathematics, University of California, Santa Cruz, 1156 High Street, Santa Cruz, CA 95064, United States of America}
\email{jelkass@ucsc.edu}
\urladdr{https://www.math.ucsc.edu/people/faculty.php?uid=jelkass}

\author{Marc Levine}

\address{Current: M.~Levine, University of Duisburg-Essen, Germany}
\email{marc.levine@uni-due.de}
\urladdr{https://www.esaga.uni-due.de/marc.levine/}

\author{Jake P. Solomon}

\address{Current: J.~P. Solomon, Institute of Mathematics, Hebrew University, Givat Ram Jerusalem, 91904, Israel}
\email{jake@math.huji.ac.il}
\urladdr{http://www.ma.huji.ac.il/~jake/}

\author{Kirsten Wickelgren}

\address{Current: K.~Wickelgren, Department of Mathematics, Duke University, 120 Science Drive
Room 117 Physics, Box 90320, Durham, NC 27708-0320, USA}
\email{kirsten.wickelgren@duke.edu}
\urladdr{https://services.math.duke.edu/~kgw/}


\date{February 2026}

\begin{abstract}
We define a quadratically enriched count of rational curves in a given divisor class passing through a collection of points on a del Pezzo surface $S$ of degree $\geq 3$ over a perfect field $k$ of characteristic $\neq 2,3.$ When $S$ is $\mathbb{A}^1$-connected, the count takes values in the Grothendieck-Witt group $\GW(k)$ of quadratic forms over $k$ and depends only on the divisor class and the fields of definition of the points. More generally, the count is a section of the Grothendieck-Witt sheaf evaluated on $\pi_0^{\mathbb{A}^1}$ of the restriction of scalars of $S$ corresponding to the fields of definition of the points. We also treat del Pezzo surfaces of degree $2$ under certain conditions. The curve count defined in the present work recovers Gromov-Witten invariants when $k = \mathbb{C}$ and Welschinger invariants when $k = \mathbb{R}.$

To obtain an invariant curve count, we define a quadratically enriched degree for an algebraic map $f$ of $n$-dimensional smooth schemes over a field $k$ under appropriate hypotheses. For example, $f$ can be proper, generically finite and oriented over the complement of a subscheme of codimension $2.$ This degree is compatible with F. Morel's $\GW(k)$-valued degree of an $\A^1$-homotopy class of maps between spheres. For $k \subseteq \mathbb{C}$, this produces an enrichment of the topological degree of a map between manifolds of the same dimension.

\end{abstract}
\maketitle
\tableofcontents

{\parskip=12pt % closing bracket is just before the bibliography

\section{Introduction}
\subsection{Background}
A degree $d$ rational plane curve over $\bbb{C}$ is a map $u: \bbb{P}_{\bbb{C}}^1 \to \bbb{P}_{\bbb{C}}^2$ given by $u([s:t]) = [u_0(s,t), u_1(s,t), u_2(s,t)]$ where the $u_i$ are homogeneous polynomials of degree $d$. A dimension count shows that one expects to have finitely many degree $d$ rational plane curves passing through $3d-1$ points. For $d=1$, such curves are the lines through two points. For $d=2$, they are the conics through five. Over $\bbb{C}$, the number of such curves, $N_d$, is independent of the choice of points, provided the points are in general position, and the first values are given by
\[
N_1 = 1, \quad N_2 = 1, \quad N_3 = 12, \quad N_4 = 620, \quad N_5 = 87,304, \quad \ldots
\]
The number $N_4 = 620$ was first computed by Zeuthen~\cite{Zeuthen} in 1873. In the early 1990's, building on ideas of Gromov~\cite{Gromov} and Witten~\cite{Witten88,Witten90}, a general approach to curve counting problems was formulated~\cite{Kontsevich-Manin,McDuff-Salamon,Ruan-Tian94,Ruan-Tian}, which has come to be known as Gromov-Witten theory. An early success of Gromov-Witten theory was a simple recursive formula giving $N_d$ for $d \geq 5$. Another road-mark was the virtual enumeration of rational curves on the quintic threefold in agreement with mirror symmetry~\cite{Kontsevich_Enumeration,Givental,LLY}.

The power of Gromov-Witten theory stems from the topological interpretation of curve counts as intersection numbers. So, even if general position cannot be achieved, one can still make sense of the curve counts. This can be done either through symplectic or algebraic geometry. Here, we focus on the algebraic approach. Let $X$ be a projective algebraic variety over $\C$ of dimension $r$ and let $\bar M_{g,n}(X,\beta)$ be the space of stable maps $u : \Sigma \to X$ where $\Sigma$ is a nodal curve of arithmetic genus $g$ with $n$ marked points $p_1,\ldots,p_n,$ and the degree is $u_*[\Sigma] = \beta \in H_2(X;\Z).$ Let $\ev_i : \bar M_{g,n}(X,\beta) \to X$ be the evaluation map at the $i$th marked point, given by $(u,\Sigma,p) \mapsto u(p_i).$ In general $\bar M_{g,n}(X,\beta)$ is a singular Deligne-Mumford stack and the dimension of irreducible components can vary. However, it admits a virtual fundamental class $[\bar M_{g,n}(X,\beta)]$ of dimension $(1-g)(r-3) + n + \int_\beta c_1(TX).$ The Gromov-Witten invariant counting curves of genus $g$ and degree $\beta$ passing through cycles representing the Poincar\'e duals of $A_i \in H^{l_i}(X)$ is defined by
\[
GW_{g,\beta}(A_1,\ldots,A_n) = \int_{[\bar M_{g,n}(X,\beta)]} \ev_1^*A_1 \cup \cdots \cup \ev_n^*A_n.
\]
So, in the special case that $A_1,\ldots,A_n = \po \in H^{2r}(X)$ are the Poincar\'e dual of the point class, the Gromov-Witten invariant $GW_{g,\beta}(A_1,\ldots,A_n)$ is the virtual degree of the total evaluation map
\[
\ev = \ev_1 \times \cdots \times \ev_n : \bar M_{g,n}(X,\beta) \to X^n.
\]
If we take $X = \PP^2$ and $\ell \in H_2(X;\Z)$ the class of a line, then $N_d = GW_{0,d\ell}(\po,\ldots,\po)=GW_{0,d\ell}(\po^{\otimes n}).$

Over the real numbers $\bbb{R}$, it is no longer true that the number of real degree $d$ rational plane curves passing through $3d-1$ real points is independent of the general choice of points. For example, there can be $8$, $10$, or $12$ real rational cubics passing through $8$ real points~\cite[p. 55]{Degtyarev-Kharlamov}. However, Degtyarev-Kharlamov~\cite[Lem. 4.7.3]{Degtyarev-Kharlamov} showed that a certain difference is always $8$: a real node can have two real branches or the branches can be complex conjugate. The number of real rational cubics whose node has two real branches minus the number whose node has complex conjugate branches is always $8.$
Welschinger showed the invariance of a signed count of rational plane curves over $\R$ of degree $d$ passing through $3d-1-2m$ real points and $2m$ pairs of complex conjugate points. The sign with which a curve contributes to the count is given by the parity of the number of nodes where two complex conjugate branches intersect. More generally, he showed~\cite{Welschinger-invtsReal4mflds,Welschinger3d} the invariance of analogous counts of real $J$-holomorphic spheres on real symplectic manifolds of dimensions $4$ and $6$. In algebraic geometry, this corresponds to counting real rational curves on real surfaces or threefolds.

A topological approach to Welschinger's invariants was developed in the context of open Gromov-Witten theory~\cite{Cho,Solomon-thesis}. The terminology `open' comes from open string theory~\cite{WittenCSST}. Let $X$ be a symplectic manifold of dimension $2r,$ let $L \subset X$ be a Lagrangian submanifold and let $J$ be a tame almost complex structure on $X$. For example, take $L$ to be a component of the real points of a projective algebraic variety $P$ over $\R$, take $X$ to be the complex points of the base change to $\C$ and take $J$ to be the standard complex structure on~$X.$ In this example, we have an anti-symplectic involution $\phi : X \to X$ given by the action of $Gal(\C/\R)$ such that $L \subset Fix(\phi).$ Let $\bar M_{D,s,t}(X/L,\beta)$ denote the space of $J$-holomorphic stable maps $u:(\Sigma,\partial \Sigma) \to (X,L)$ where $\Sigma$ is a nodal disk with $s$ boundary marked points $z_1,\ldots,z_s$ and $t$ interior marked points $w_1,\ldots,w_t$ of degree $u_*[\Sigma,\partial \Sigma] = \beta \in H_2(X,L;\Z).$ Let $\evb_i : \bar M_{D,s,t}(X/L,\beta) \to L$ denote the evaluation map at the $i$th boundary marked point given by $(u,\Sigma,z,w) \mapsto u(z_i).$ Let $\evi_j : \bar M_{D,s,t}(X/L,\beta) \to X$ denote the evaluation map at the $j$th interior marked point given by $(u,\Sigma,z,w) \mapsto u(w_j).$ In nice cases, the space $\bar M_{D,s,t}(X/L,\beta)$ is a manifold with corners of dimension $\mu(\beta)+r-3 + s + 2t$ where $\mu: H_2(X,L;\Z) \to \Z$ is the Maslov index. In general, $\bar M_{D,s,t}(X/L,\beta)$ is singular but nonetheless admits a virtual fundamental class with dimension given by the same formula. Let
\[
\ev_D = \evb_1 \times \cdots \times \evb_s \times \evi_1 \times \cdots \times \evi_t :  \bar M_{D,s,t}(X/L,\beta) \to L^s \times X^t
\]
denote the total evaluation map. Recall that a relative orientation for a map of smooth manifolds $f : M \to N$ is an isomorphism $\det(TM) \overset{\sim}\to f^*\det(TN).$ It was shown in~\cite{Solomon-thesis} that a $Pin$ structure on $L$ and an orientation on $L$ if $L$ is orientable determine a virtual relative orientation for the map $\ev_D$ when the dimensions of the domain and codomain coincide. Since $\bar M_{D,s,t}(X/L,\beta)$ is a manifold with corners, the degree of $\ev_D$ is not a priori defined. However, when $\dim X = 4,6,$ and there is an anti-symplectic involution of $X$ that fixes $L,$ it is possible to glue together certain boundary components of $\bar M_{D,s,t}(X/L,\beta)$ to obtain a new manifold with corners $\widetilde M_{D,s,t}(X/L,\beta)$ with the following two properties.
\begin{enumerate}
\item\label{it:ro}
There is an induced evaluation map $\widetilde{\ev}_D$ that is still relatively oriented.
\item\label{it:codim2}
The image of the boundary $\widetilde{\ev}_D(\partial \widetilde M_{D,s,t}(X/L,\beta)) \subset L^s \times X^t$ has codimension at least $2$.
\end{enumerate}
These two properties allow the degree of $\widetilde{\ev}_D$ to be defined.
There is a natural doubling map $\varpi : H_2(X,L;\Z) \to H_2(X;\Z).$ The degree of $\widetilde{\ev}_D$ coincides up to a factor of $2^{1-t}$ with the Welschinger invariant counting real $J$-holomorphic spheres in $X$ representing the class $\varpi(\beta)$ and passing through $s$ real points and $t$ complex conjugate pairs of points. Open Gromov-Witten theory leads to efficient recursive formulas for Welschinger invariants~\cite{Chen2d,ChenZinger,HorevSolomon,SolomonPreprint,Solomon-Tukachinsky} and the enumeration of disks on the quintic threefold in agreement with mirror symmetry~\cite{PSW}. It also allows the definition of invariants in arbitrary dimension and for $L$ not necessarily fixed by an anti-symplectic involution~\cite{Solomon-Tukachinsky-point-like}.

Analogous results over the real and complex numbers may indicate the presence of a common generalization in the $\bbb{A}^1$-homotopy theory of F. Morel and V. Voevodsky \cite{morelvoevodsky1998} valid over more general fields or base rings. We show this to be the case here. $\bbb{A}^1$-homotopy theory adds homotopy colimits to smooth schemes, allowing one to glue or crush them. For example, one has spheres $\mathbb{P}^n_k/\mathbb{P}^{n-1}_k$, where $k$ is a fixed base field. Morel's $\bbb{A}^1$-Brouwer degree theorem \cite[Theorem 1.23]{morel} identifies the $\mathbb{A}^1$-stable homotopy classes of maps from the sphere $\mathbb{P}^n_k/\mathbb{P}^{n-1}_k$ to itself with the Grothendieck--Witt group $\GW(k)$ of bilinear forms over $k$, recalled below in Section \ref{sssec:gw}. More generally, the theorem computes the $(0,0)$-stable homotopy sheaf of the sphere spectrum in $\bbb{A}^1$-homotopy theory over $k$ to be the sheaf $\sGW$, described more in Section \ref{subsection_with_sGW}. Given polynomial equations for a map $\mathbb{P}^n_k/\mathbb{P}^{n-1}_k \to \mathbb{P}^n_k/\mathbb{P}^{n-1}_k$ the degree is computed as a sum of local degrees in~\cite{KWA1degree}. Morel's $\bbb{A}^1$-Brouwer degree for maps between spheres identifies the target for the $\mathbb{A}^1$-degrees that we develop here and apply to the above evaluation maps. Away from a codimension 1 locus, the degree is the sum of the local degrees over points of the fiber.

\subsection{Statement of results}
The present paper aims to develop certain Gromov--Witten invariants and rational curve counts over perfect fields $k$ of characteristic not $2$ or $3$, by recasting the arguments of~\cite{Solomon-thesis} in $\bbb{A}^1$-homotopy theory. A relative orientation of a morphism $f: M \to N$ of smooth $k$-schemes is an invertible sheaf $L$ on $M$ together with an isomorphism $\rho : Hom(\det TM,\det TN) \to L^{\otimes 2}.$ Let $S$ be a del Pezzo surface over $k,$ in the sense that $S$ is a geometrically connected, smooth, projective $k$-scheme of dimension $2$ with ample anticanonical bundle $-K_S$. Let $d_S = K_S \cdot K_S$ denote the degree of $S$.

Let $\bar M_{0,n}(S,D)$ denote the space of genus zero stable maps with $n$ marked points in the class $D \in Pic(S)$ and consider the total evaluation map $\ev : \bar M_{0,n}(S,D) \to S^n.$ We construct Gromov--Witten invariants associated to points with different residue fields, which we collect in a list called $\sigma$. So, let $\sigma = (L_1,\ldots,L_r)$ be an $r$-tuple of field extensions $k \subset L_i \subset \bar k$ such that $\sum_{i = 1}^r [L_i:k] = n.$ For an $L$-scheme $X$, let $\Res_{L/k} X$ denote the restriction of scalars to $k$. We construct a Galois twist corresponding to $\sigma$ (see Section~\ref{subsectionRCsection:twistsev})
\[
\ev_\sigma: \bar M_{0,n}(S,D)_\sigma \to (S^n)_\sigma = \prod_{i=1}^r \Res_{L_i/k}S.
\]
For the rest of the introduction, we fix $d = -K_S \cdot D$ and $n = d - 1$ and work under the following hypothesis. 

\begin{hyp}\label{hyp:basic}
Let $S$ be a del Pezzo surface over $k$ and $D \in \Pic (S)$. Assume that $D$ is not an $m$-fold multiple of a $-1$-curve for $m>1$. Moreover, assume that $d_S\ge 4$, or $d_S=3$ and $d \neq 6$, or $d_S = 2$ and $d\ge 7$.
\end{hyp}

\subsubsection{Characteristic zero}
Assume first that $k$ has characteristic zero.
Using Hypothesis~\ref{hyp:basic}, we identify a closed subset $A \subset S^n$ such that $V_\sigma : = (\bar M_{0,n}(S,D) \setminus \ev^{-1}(A))_\sigma$ has the following two properties analogous to properties~\ref{it:ro} and~\ref{it:codim2} of $\widetilde M_{D,s,t}(X/L,\beta)$ above. See Theorem~\ref{thm:rel_or_char0}, which builds on \cite[Theorem 4.5]{KLSW-relor}.
\begin{enumerate}[label=(\arabic*$'$)]
\item\label{it:ro'}
The restriction of the total evaluation map $\ev_\sigma: V_\sigma \to (S^n)_\sigma$ is relatively oriented.
\item\label{it:codim2'}
The codimension of $A_\sigma \subset (S^n)_\sigma$ is at least $2$.
\end{enumerate}
In the case $k = \R$, we can make the relation between $\widetilde M_{D,s,t}(X/L,\beta)$ and $V_\sigma$ precise as follows. Take $L \subset X$ the Lagrangian submanifold corresponding to $S,$ take $s$ the number of $i$ such that $L_i = \R$ and $t$ the  number of $i$ such that $L_i = \C.$ There is a commutative diagram
\[
\xymatrix{
\Int \widetilde M_{D,s,t}(X/L,\beta) \ar[r]^(.65){\widetilde{\ev}_D}\ar[d] & L^s \times X^t \ar[d]^\wr \\
V_\sigma(k) \ar[r]^{\ev_\sigma} & (S^n)_\sigma(k)
}
\]
where the right vertical arrow is a bijection and the left vertical arrow is two-to-one in the case $t = 0$ and one-to-one onto a fundamental domain for an action of the group $(\Z/2)^{t-1}$ when $t \geq 1.$

In Definition~\ref{df:pseudo-oriented}, we introduce the notion of a pseudo-orientation of a map from an Artin stack to a smooth $k$-scheme. Building on  F. Morel's $\mathbb{A}^1$-degree \cite{Morel_motivicpi0_sphere,KWA1degree}, Theorem~\ref{tm:A1Degree_pseudo-oriented} associates to a pseudo-oriented map $f: X \to Y$ a unique section of $\deg(f)$ of the Grothendieck-Witt sheaf $\sGW$ over $Y.$ When $Y$ is $\Aone$-connected, such a section is pulled back from a unique element of the Grothendieck-Witt ring $\GW(k).$  
When, in addition, $X$ is smooth and $k = \C,$ the rank of $\deg(f)$ is the usual topological degree. If $k = \R,$ the signature of $\deg(f)$ is the usual topogical degree.    
Properties~\ref{it:ro'} and~\ref{it:codim2'} are used to prove the following theorem, which is one of our main results.
\begin{thmintro}\label{the:intro:deg}
Let $S$ and $D$ satisfy Hypothesis~\ref{hyp:basic}, let $\Char k = 0,$ and assume that $S$ is $\mathbb{A}^1$-connected. Then for each $\sigma$ the evaluation map $\ev_\sigma$ is canonically pseudo-oriented, so there exists an invariant $N_{S,D,\sigma} \in \GW(k)$ given by $\deg(\ev_\sigma)$.
\end{thmintro}
Since the degree of a pseudo-oriented map is compatible with the topological degree as described above, when $k = \C,$ the rank of the invariant $N_{S,D,\sigma}$ coincides with the Gromov-Witten invariant counting rational curves on $S$ of degree $D$ passing through $n = r$ marked points. When $k = \R,$ the signature of the invariant $N_{S,D,\sigma}$ coincides with the Welschinger invariant counting real rational curves on $S$ of degree $D$ passing through $s$ real points and $t$ conjugate pairs of complex points, where $s$ and $t$ are defined as above. This follows from the topological characterization of the Welschinger invariant as the degree of $\widetilde{\ev}_D$ given in~\cite{Cho,Solomon-thesis}. 

The proof of Theorem~\ref{the:intro:deg} is given in Theorem~\ref{thm:rel_or_char0} and Definitions~\ref{df:rel_or_char0} and~\ref{df:qecc}. An explanation of how Hypothesis~\ref{hyp:basic} is used is given in Remark~\ref{rmk:Hyp1needed}.

\subsubsection{Positive characteristic}\label{sssec:pc}
We turn to the case when $k$ has positive characteristic.
Let  $M^\bir_0(S, D) \subset \overline{M}_0(S,D)$ be the open subscheme consisting of maps $u: C \to S$ from irreducible genus $0$ curves such that $C \to u(C)$ is birational. Such $u$ is said to be {\em unramified} if the differential $u^* T^* S \to T^*C$ is surjective. We will use the following hypothesis.
\begin{hyp}\label{hyp:pc}
In addition to Hypothesis~\ref{hyp:basic}, assume $k$ is perfect of characteristic not $2$ or~$3$. If $d_S=2$, assume additionally that for every effective $D' \in Pic(S)$, there is a geometric point $f$ in each irreducible component of $M^\bir_0(S, D')$ with $f$ unramified.
\end{hyp} 

Let $\Lambda$ be a complete discrete valuation ring with residue field $k$ and quotient field $K$ of characteristic $0$. In Section~\ref{sssec:pclift}, building on~\cite[Section 9]{KLSW-relor}, we construct $\LiftS \to \Spec \Lambda$ a smooth del Pezzo surface equipped with an effective $\LiftDeg \in Pic(\LiftS)$ with special fibers $\LiftS_k \cong S$ and $\LiftDeg_k \cong D$. We construct a Galois twist
\[
\Liftev_\sigma: \bar M_{0,n}(\LiftS,\LiftDeg)_\sigma \to (\LiftS^n)_\sigma
\]
that agrees with $\ev_\sigma$ on the special fiber.
Moreover, we identify a closed subset $\LiftA \subset \LiftS^n$ such that $\LiftV_\sigma : = (\bar M_{0,n}(\LiftS,\LiftDeg) \setminus \ev^{-1}(\LiftA))_\sigma$ has the following two properties analogous to properties~\ref{it:ro} and~\ref{it:codim2} of $\widetilde M_{D,s,t}(X/L,\beta)$ above.
\begin{enumerate}[label=(\arabic*$''$)]
\item\label{it:ro''}
The restriction of the total evaluation map $\Liftev_\sigma: \LiftV_\sigma \to (\LiftS^n)_\sigma$ is relatively oriented.
\item\label{it:codim2''}
The codimension of $\LiftA_\sigma \subset (\LiftS^n)_\sigma$ is at least $2$.
\end{enumerate}
Properties~\ref{it:ro''} and~\ref{it:codim2''} are used to prove the following result.
\begin{thmintro}\label{the:intro:degpc}
Let $S,D$ and $k$ satisfy Hypothesis~\ref{hyp:pc} and assume that $S$ is $\mathbb{A}^1$-connected. Then, for each $\sigma$ the evaluation map $\ev_\sigma$ is canonically pseudo-oriented, so there exists an invariant $N_{S,D,\sigma} \in \GW(k)$ given by $\deg(\ev_\sigma)$.
\end{thmintro}
In particular, the pseudo-orientation of $\ev_\sigma$ depends only on the existence of $\LiftS,\LiftDeg,\LiftA,$ but not on the choice of particular ones. The proof of Theorem~\ref{the:intro:degpc} is given in Theorems~\ref{thm:rel_or_charpsf} and~\ref{thm:rel_or_pchar} and Definitions~\ref{df:rel_por} and~\ref{df:qecc}. An explanation of how Hypothesis~\ref{hyp:pc} is used is given in Remark~\ref{rmk:useHyp2}. An explanation of the case $d_S = 2$ in Hypothesis~\ref{hyp:pc} is given in Remark~\ref{rmk:hyp2satisfied}.

\subsubsection{The Grothendieck--Witt ring}\label{sssec:gw}
In order to explain the enumerative meaning of the invariants $N_{S,D,\sigma},$ we recall the definition and basic properties of the Grothendieck-Witt ring $\GW(k).$ The Grothendieck--Witt ring is defined as the group completion of the semi-ring of non-degenerate symmetric bilinear forms over $k$. Since symmetric bilinear forms over a field are stably diagonalizable, an arbitrary element of this group can be expressed as a sum of rank $1$ bilinear forms. Let $\langle a \rangle$ denote the element of $\GW(k)$ corresponding to the rank $1$ bilinear form $k \times k \to k$ given by $(x,y) \mapsto a xy$ for $a$ in $k^*$. Replacing the basis $\{1 \}$ of $k$ by $\{b\}$ for $b$ in $k^*$ gives the equality $\langle a \rangle = \langle a b^2\rangle$, and in particular for fields such that $k^*/(k^*)^2$ is trivial, $\GW(k)$ is isomorphic to $\bbb{Z}$ by the homomorphism taking a bilinear form to its rank, i.e. the dimension of the underlying vector space.  For more general fields, $\GW(k)$ contains more information. For example,
\[
\GW(\bbb{R}) \cong \bbb{Z} \oplus \bbb{Z}, \quad \GW(\bbb{F}_q) \cong \bbb{Z} \times \bbb{F}_q^*/(\bbb{F}_q^*)^2,  \quad  \GW(~\bbb{C}((z))~) \cong \bbb{Z} \times \bbb{C}((z))^*/(\bbb{C}((z))^*)^2 ,
\]
\[
 \GW(\bbb{Q}_q) \cong \frac{\GW(\bbb{F}_q) \oplus \GW(\bbb{F}_q)}{ (\langle 1 \rangle + \langle -1 \rangle , -(\langle 1 \rangle + \langle -1 \rangle))\bbb{Z} }\text{ for }2 \nmid q,
\]
\[
\GW(\bbb{Q}) \cong \bbb{Z} \oplus \bbb{Z} \oplus  \bbb{Z}/2 \bbb{Z} \oplus \bigoplus_{\substack{p \text{ prime }\\ p \neq 2 }} \frac{ \GW(\bbb{F}_p)}{(\langle 1 \rangle + \langle -1 \rangle) \bbb{Z}}.
\]
For finite rank field extensions $L \subseteq E$, there is an additive transfer map $$\Tr_{E/L}: \GW(E) \to \GW(L),$$ which has the following simple description when $L \subseteq E$ is separable: for a symmetric, non-degenerate bilinear form $\beta: V \times V \to E$ over $E$, we can view $V$ as a vector space over $L$ and consider the composition $$V \times V \stackrel{\beta}{\to} E \stackrel{\Tr_{E/L}}{\longrightarrow} L$$ where $\Tr_{E/L}$ is the sum of the Galois conjugates in the algebraic closure of $L$. Since $L \subseteq E$ is separable, $Tr_{E/L} \circ \beta$ is a non-degenerate symmetric bilinear form over $L$. The value of the transfer map on the class $[\beta]$ of the form $\beta$ is given $\Tr_{E/L} [\beta] = [Tr_{E/L} \circ \beta]$.

The Milnor conjecture, proven by Voevodsky and Orlov--Vishik--Voevodsky, defines a sequence of invariants beginning with the rank, discriminant, Hasse--Witt invariant, Arason invariant, which for many fields (including finite fields, number fields, complete discretely valued fields, say in residue characteristic not $2$ etc.) give a terminating algorithm for determining if two elements given by sums of rank $1$ forms $\langle a \rangle$ are equal \cite{Milnor_AlgK-theory_quadratic_forms} \cite{OrlovVishikVoevodsky} \cite{Voevodsky_MCZ2} \cite{Voevodsky-reduced_power_operations}. There are many powerful tools for working with Grothendieck--Witt groups. See for example \cite{lam05} \cite{Lam06} \cite{milnor73}.

\subsubsection{Enumerative meaning}\label{ssec:enummean}
To see the enumerative meaning of the degree $N_{S,D,\sigma}$, we generalize the sign associated to a node with two complex conjugate branches over $\bbb{R}$. Suppose $u: \bbb{P}_{k(u)} \to S$ is a rational curve on $S$ defined over the field extension $k(u)$ of $k$. Let $p$ be a node of $u(\bbb{P}_{k(u)})$. The two tangent directions at $p$ define a field extension $k(p)[\sqrt{D(p)}]$ of $k(p)$, for a unique element $D(p)$ in $k(p)^*/(k(p)^*)^2$. Here the ``$D$" in $D(p)$ could stand for the discriminant of the field extension, or directions, as in the ``tangent directions." By \cite[Expos\'e XV Th\'eor\`eme 1.2.6]{SGA7_10_22}, the extension $k(u) \subseteq k(p)$ is separable. Let $\norm{k(p)}{k(u)}:  k(p)^* \to  k(u)^*$ denote the norm of the field extension $k(u) \subseteq k(p)$ given by the product of the Galois conjugates in the algebraic closure of $k(u)$.

\begin{dfintro}\label{df:massnode:intro}
The {\em mass} of $p$ is defined by
\begin{equation}\label{eon:df:mass:intro}
\mass(p) = \langle \norm{k(p)}{k(u)} D(p) \rangle \quad \text{in}~\GW(k(u)).
\end{equation} 
\end{dfintro} This makes sense because multiplying $D(p) $ by a square in $k(p)$ multiplies the norm by a square in $k(u)$. See Section~\ref{sec:geom_formula_local_deg_ev} for more information. The terminology {\em mass} is as in Welschinger \cite{Welschinger-invtsReal4mflds}.

Let $\sigma = (L_1,\ldots,L_r),$ be a collection of field extensions such that $k \subset L_i \subset \bar k$. We say that a statement holds for $p_i \in S(L_i),\, i = 1,\ldots,r$ in general position if there exists a dense open subset $U\subset \prod_{i=1}^r \Res_{L_i/k} S$ such that the statement holds when $(p_1,\ldots,p_r)$ corresponds to a rational point of $U$ under the canonical bijection $\prod_{i=1}^r S(L_i) \cong (\prod_{i=1}^r \Res_{L_i/k} S) (k)$ coming from the definition of the restriction of scalars. The open subset $U$ may not contain a rational point and thus there may be no $p_1,\ldots,p_r$ in general position. Even for $S=\mathbb{P}^2$, this may happen over a finite field.

The following is valid under the same hypotheses as Theorem~\ref{the:intro:deg} for $k$ of characteristic zero and under the same hypotheses as Theorem~\ref{the:intro:degpc} for $k$ of positive characteristic. It is a special case of Theorem~\ref{thmintro:sGWRes_curve_count}.

\begin{thmintro}\label{the:intro:curve_count} 
For $p_1,p_2,\ldots,p_r \in S$ with $k(p_i) \cong L_i$ in general position, we have the equality
\[
N_{S,D,\sigma} =  \sum_{\substack{u \text{ rational curve} \\  \text{ in class } D\\\text{ through the points} \\ p_1, \ldots, p_r}} \Tr_{k(u)/k}\prod_{p \text{ node}\text{ of }u(\bbb{P}^1)} \operatorname{mass}(p).
\]
in $\GW(k)$.
So the weighted count of degree $D$ rational plane curves through the points $p_1,p_2,\ldots,p_r$ given on the right hand side is independent of the general choice of points. When $k$ is an infinite field and $S$ is rational over $k,$ there exist such $p_1,p_2,\ldots,p_r$.
\end{thmintro}

Consequently, for $k = \bbb{C}$ the rank of $N_{S,D,\sigma}$ coincides with the corresponding Gromov--Witten invariant. For $k=\bbb{R}$, the signature of $N_{S,D,\sigma}$ recovers the signed counts of real rational curves of Degtyarev-Kharlamov and Welschinger. For $k = \bbb{F}_p, \mathbb{Q}_p, \mathbb{Q}$ etc., one obtains a new Gromov--Witten invariant.

Even when points $p_1,\ldots,p_r$ with $k(p_i) \cong L_i$ in general position do not exist, $N_{S,D,\sigma}$ is a meaningful invariant. It is the $\mathbb{A}^1$-degree of an evaluation map given in Theorem~\ref{the:intro:degpc} and an analogue of a Gromov--Witten invariant defined over perfect fields of characteristic not $2$ or $3$, including finite fields. Just as Gromov--Witten invariants make sense of curve counts when general position can not be achieved, these analogues give meaning to curve counts when rational points do not exist in a dense open subset $U$ of $\prod_{i=1}^r \Res_{L_i/k} S$.

This degree also retains concrete enumerative significance: The open subset $U$ will contain many points over finite extensions of $k$ and our constructions behave well under base change. Pick a point of $U$ with field of definition $E$. This point corresponds to an element of $S(E \otimes \prod_{i=1}^r L_i)$ by the definition of the restriction of scalars, which in turn corresponds to a list $p_1', \ldots, p_{j}'$ of field valued points of $S$. Corollary~\ref{cor:intro:countA1DelPezzo} gives the equality in $\GW(E),$
\begin{equation}\label{eq:eneqL}
N_{S,D,\sigma} \otimes E =  \sum_{\substack{u \text{ rational curve} \\  \text{ in class } D\\\text{ through the points} \\ p_1, \ldots, p_{r'}}} \Tr_{k(u)/E}\prod_{p \text{ node}\text{ of }u(\bbb{P}^1)} \operatorname{mass}(p).
\end{equation}
This gives infinitely many concrete enumerative equalities \eqref{eq:eneqL}. Base change to $E$ frequently results in a loss of information, but not always. For example, $\otimes E: \GW(k) \to \GW(E)$ is injective for $[E:k]$ odd. In the case where $k$ is a finite field and $S=\mathbb{P}^2$, there is always an extension $k \subset E$ with $[E:k]$ odd and points $p_1,\ldots,p_{j}$ corresponding to a point of $U(E)$, so base change to $E$ loses no information in the left hand sides of the resulting equalities \eqref{eq:eneqL}. 

\subsubsection{Examples}

\begin{exintro}\label{ex:basic}
$\bbb{A}^1$-connected del Pezzo surfaces include $\bbb{P}^2$, $\bbb{P}^1 \times \bbb{P}^1,$ and $\Bl_B \bbb{P}^1$, where $B$ is a set of closed points $\{ p_1, \ldots, p_r\}$ considered as a subscheme defined over $k$ satisfying $\vert B \vert = \sum_{i=1}^r [k(p_i): k] \leq 7$. In this case, $d_S = 9 - \vert B \vert$. In particular, let $k$ be a perfect field of characteristic not $2$ or $3$. Then, Theorems~\ref{the:intro:deg} and~\ref{the:intro:degpc} give invariants $N_{\bbb{P}_k^2, D, \sigma}$ and $N_{\bbb{P}_k^1 \times \bbb{P}_k^1, D, \sigma}$ in $\GW(k)$ for all Picard classes $D$. Similarly, for $\vert B \vert \leq 6$ and $S = \Bl_B\bbb{P}_k^2$, we have $N_{S, D, \sigma}$ for all $D \in Pic(S)$ that are not $m$-fold multiples of a $-1$-curve.
\end{exintro}

\begin{exintro}\label{ex:smooth_proper_rational_surfaces_are_A1-connected}
Smooth, proper, $k$-rational surfaces are also $\bbb{A}^1$-connected \cite[Corollary 2.3.7]{AsokMorel}. So, Theorems~\ref{the:intro:deg} and~\ref{the:intro:degpc} apply to rational del Pezzo surfaces. A smooth cubic surface over $k$ containing two skew lines over $k$ or two skew lines over a quadratic extension of $k$ which are conjugate is $k$-rational \cite[1.33, 1.34]{Kollar-Rational_and_nearly}.  Cubic surfaces are del Pezzo surfaces with $d_S=3$, so Theorems~\ref{the:intro:deg} and~\ref{the:intro:degpc} give invariants $N_{S, D, \sigma}$ in $\GW(k)$ for any $D$ with $d = -K_S \cdot D \neq 6$ and $D$ not an $m$-fold multiples of a $-1$-curve. For example, let $S_0\subset \bbb{P}^3$ be the smooth cubic surface given by the zero locus of $x^2 y + y^2 z + z^2 w + w^2 x$. Then $S_0$ is rational \cite[1.4]{Kollar-Rational_and_nearly} giving invariants $N_{S_0, D, \sigma}$ in $\GW(k)$.
\end{exintro}

Examples of del Pezzo surfaces which are not $\bbb{A}^1$-connected are available in Example~\ref{ex:SnotA1connected}.

\begin{exintro}
We compute $N_{S,- K_{S}, \sigma} =  \langle -1 \rangle\chi^{\A^1}(S)  + \langle 1 \rangle +  \Tr_{k(\sigma)/k} \langle 1 \rangle$, where $\chi^{\bbb{A}^1}(S)$ denotes the $\mathbb{A}^1$-Euler characteristic. See Example~\ref{ex:NS-KSsigma}. For real and complex computations giving analogous Welschinger and Gromov--Witten invariants, see for example \cite[Proposition 4.7.3]{Degtyarev-Kharlamov} \cite[3.3]{Sottile-EnumerativeReal} \cite[Chapter 7 Keynote Question(a), Proposition 7.4]{EisenbudHarris3264}. For $S_0$ as in Example~\ref{ex:smooth_proper_rational_surfaces_are_A1-connected}, $\chi^{\bbb{A}^1}(S_0) = \langle -5\rangle + 4(\langle 1\rangle + \langle -1 \rangle)$ \cite{LLV-Eulerchar}, which gives $N_{S_0,- K_{S_0}, \sigma} =  \langle 5\rangle +  \langle 1\rangle +4(\langle 1\rangle + \langle -1 \rangle) +  \Tr_{k(\sigma)/k} \langle 1 \rangle$. 
\end{exintro}

Considerable progress has been made computing the invariants of Theorems \ref{the:intro:deg} and \ref{the:intro:degpc} since this work became available. Andr\'es Jaramillo Puentes and Sabrina Pauli computed the enriched count of rational curves through rational points on a toric surface via a tropical correspondence theorem \cite{PuentesPauli-Correspondence}, building on their previous work~\cite{PuentesPauli-Bezout}. Hannah Markwig, Jaramillo Puentes, Pauli, and Felix R\"ohrle further develop these tropical techniques to allow points with residue fields which are degree $2$ extensions of the base field \cite{PuentesMarkwigPauliRorhle-quad_extensions}. Using this work and a quadratically enriched Abramovich--Bertram formula \cite{BrugalleWickelgren-AB}, Erwan Brugall\'e, Johannes Rau, and the fourth named author compute $N_{S,D,\sigma}$ for all $\sigma$ and $(S,D)$ as above with $d_S \geq 6$ \cite{BRW-WWI}. Conjecturally, the formula of \cite{BRW-WWI} for $N_{S,D,\sigma}$ holds for all rational surfaces. Markwig, Jaramillo Puentes, Pauli, and R\"ohrle also have the analogue of a Caparaso--Harris recursion formula when the interpolated points are rational \cite{JPMPR-tropPlaneCurves}.

\subsubsection{Without the connectedness hypothesis} \label{subsection:notA1-connected}\label{subsection_with_sGW}
Theorems~\ref{the:intro:deg},~\ref{the:intro:degpc} and \ref{the:intro:curve_count} above are special cases of more general results that do not require that $S$ be $\mathbb{A}^1$-connected.

The Grothendieck--Witt groups $\GW(E)$ discussed above for $E$ a finite type field extension of $k$, together with certain boundary maps, determine a sheaf of abelian groups on smooth $k$-schemes
\begin{align*}
\sGW: \Sm^{\op} \to \Ab,
\end{align*}
which is unramified in the sense of e.g. \cite[Definition 2.1]{morel}. For $X$ a smooth $k$-scheme, $\sGW(X) \subset \GW(k(X))$ is the subset of the Grothendieck--Witt group of its field of rational functions which is in the kernel of boundary maps indexed by the codimension $1$ points of $X$. See \cite[Definition 2.1, Lemma 3.10, Section 3.2]{morel}. For a presheaf $X: \Sm^{\op} \to \Set$, define $\sGW(X):=\Map_{\Fun(\Sm^{\op},\Set)}(X,\sGW)$. This will be discussed further in Section \ref{subsection:Globaldegreemapsmpro}.

To formulate our general result, we use the sheaf of $\bbb{A}^1$-connected components, $\pi_0^{\bbb{A}^1}$. This sheaf arises naturally when considering the degree of a morphism to a scheme that is not $\bbb{A}^1$-connected, reflecting the classical phenomenon that a map to a disconnected manifold may have a different degree over different connected components.
Unlike in classical topology, it is not possible to decompose a smooth scheme into $\bbb{A}^1$-connected pieces; the sheaf of connected components is often a very complicated object. For a smooth scheme $X$, define $\pi_0^{\bbb{A}^1}(X)$ to be the Nisnevich sheaf associated to the presheaf taking a smooth $k$-scheme $U$ to $[U,X]_{\bbb{A}^1}$, where $[U,X]_{\bbb{A}^1}$ denotes the (unstable) $\bbb{A}^1$-homotopy classes of maps from $U$ to $X$. A smooth scheme $X$ is said to be $\bbb{A}^1$-connected when $\pi_0^{\bbb{A}^1}(X)$ is trivial. This is discussed further in Section~\ref{subsection:GWk-valued_global_degree}. As in topology, there is a natural map $X \to \pi_0^{\bbb{A}^1}(X)$. Since $\sGW$ is $\Aone$-homotopy invariant~\cite{morel}, every map $X \to \sGW$ factors uniquely through the map $X \to \pi_0^{\Aone}(X).$ See Proposition~\ref{pr:htpyinvtsheaf_iso_Xtopi0} for details. Thus, we obtain the following generalization of Theorem~\ref{the:intro:deg} when the $\Aone$-connectedness hypothesis is removed.
\begin{thmintro}\label{the:intro:degnc}
Let $S$ and $D$ satisfy Hypothesis~\ref{hyp:basic}, and let $\Char k = 0.$ Then for each $\sigma$ the evaluation map $\ev_\sigma$ is canonically pseudo-oriented, so there exists an invariant 
\[
\underline N_{S,D,\sigma} \in \sGW(\pi_0^{\bbb{A}^1}(\prod_{i=1}^r \Res_{L_i/k} S))
\]
given by the degree of $\ev_\sigma$.
\end{thmintro}
The proof of Theorem~\ref{the:intro:degnc} is given in Theorem~\ref{thm:rel_or_char0} and Definitions~\ref{df:rel_or_char0} and~\ref{df:qecc}.
Similarly, Theorem~\ref{the:intro:degpc} generalizes as follows when the $\Aone$-connectedness hypothesis is removed.
\begin{thmintro}\label{the:intro:degpcnc}
Let $S,D$ and $k$ satisfy Hypothesis~\ref{hyp:pc}. Then, for each $\sigma$ the evaluation map $\ev_\sigma$ is canonically pseudo-oriented, so there exists an invariant $\underline N_{S,D,\sigma} \in \sGW(\pi_0^{\bbb{A}^1}(\prod_{i=1}^r \Res_{L_i/k} S))$ given by the degree of $\ev_\sigma$.
\end{thmintro}

The proof of Theorem~\ref{the:intro:degpcnc} is given in Theorems~\ref{thm:rel_or_charpsf} and~\ref{thm:rel_or_pchar} and Definitions~\ref{df:rel_por} and~\ref{df:qecc}.
For a $k$-point $x$ of $X$ and a section $\underline{N}$  in $\sGW(\pi_0^{\bbb{A}^1}(X))$, let $\underline{N}(x) \in \GW(k)$ denote the pullback of $\underline{N}$ to $x$ along the composition $\Spec k \stackrel{x}{\to} X \to \pi^{\bbb{A}^1}_0(X)$.
The following generalization of Theorem~\ref{the:intro:curve_count} is valid under the same hypotheses as Theorem~\ref{the:intro:degnc} for $k$ of characteristic zero and under the same hypotheses as Theorem~\ref{the:intro:degpcnc} for $k$ of positive characteristic. It is a special case of Theorem~\ref{thmintro:sGWRes_curve_count} below, where the proof is given.

\begin{thmintro}\label{the:intro:curve_countnc}
For $p_1,p_2,\ldots,p_r \in S$ with $k(p_i) \cong L_i$ in general position corresponding to $p_* \in \prod_{i=1}^r \Res_{L_i/k} S,$ we have the equality in $\GW(k),$
\[
\underline{N}_{S,D, \sigma}(p_*) =  \sum_{\substack{u \text{ rational curve} \\  \text{ in class } D\\\text{ through the points} \\ p_1, \ldots, p_r}}  \Tr_{k(u)/k} \prod_{p \text{ node}\text{ of }u(\bbb{P}^1)}\operatorname{mass}(p).
\]
When $k$ is an infinite field and $S$ is rational over $k$, such a general choice of points exists.
\end{thmintro}

We may alternatively package the invariants $\underline{N}_{S,D, \sigma}$ into a single invariant as follows. Let $\Sym^n_0 S \subset \Sym^n S$ be the complement of the union of pairwise diagonals in the $n$-fold symmetric product. Let $\Delta_\sigma \subset (S^n)_\sigma = \prod_{i=1}^r \Res_{L_i/k} S$ denote the union of pairwise diagonals. 
\begin{thmintro}\label{thmintro:GWSymn0}
If $\Char k = 0,$ let $S$ and $D$ satisfy Hypothesis~\ref{hyp:basic} and if $\Char k >0$ let $S,D$ and~$k$ satisfy Hypothesis~\ref{hyp:pc}.
There exists $\underline N_{S,D}^{\mfS} \in \sGW(\pi_0^{\bbb{A}^1}(\Sym^n_0 S))$ that pulls back to the restriction of $\underline N_{S,D,\sigma}$ for each $\sigma$ under the natural map $\prod_{i=1}^r \Res_{L_i/k} S \setminus \Delta_\sigma \to \Sym^n_0 S$.
\end{thmintro}
The proof of Theorem~\ref{thmintro:GWSymn0} is given in Theorems~\ref{thm:rel_or_char0_sym},~\ref{thm:rel_or_charpsf_sym} and~\ref{thm:rel_or_pchar_sym}, Definitions~\ref{df:rel_or_char0_sym},~\ref{df:rel_por_sym} and~\ref{df:qecc_sym}, and Proposition~\ref{pr:pullback_deg_sym_gives_deg_twist}.

\begin{exintro}
Building on Example~\ref{ex:basic}, let $S$ be a twist of $\Bl_B\bbb{P}_k^2$ and let $k$ be a perfect field of characteristic not $2$ or $3$. Then, Theorems~\ref{the:intro:degnc} and~\ref{the:intro:degpcnc} give us invariants $\underline N_{S, D, \sigma}$ in $\sGW(\pi_0^{\bbb{A}^1}(\prod_{i=1}^r \Res_{L_i/k} S))$ for all $D \in Pic(S)$ that are not $m$-fold multiples of a $-1$-curve.
\end{exintro}

It is a problem of fundamental importance to find quadratically enriched analogs of the rich algebraic structures present in Gromov-Witten theory and open-Gromov-Witten-Welschinger theory. For example, it would be desirable to unify the WDVV equations~\cite{Kontsevich-Manin,McDuff-Salamon,Ruan-Tian,Ruan-Tian94,Witten90} and the open WDVV equations~\cite{Chen2d,ChenZinger,HorevSolomon,SolomonPreprint,Solomon-Tukachinsky} in the quadratically enriched context, thus extending them to an arbitrary field. The WDVV and open WDVV equations are systems of non-linear partial differential equations satisfied by generating series for the Gromov-Witten and open Gromov-Witten invariants respectively. They yield powerful recursion relations that in many cases determine invariants of arbitrary degree from a small number of a low degree invariants.
A unification of these equations underlies the associativity of the ring structure of relative quantum cohomology~\cite{Solomon-Tukachinsky}. Nonetheless, the geometric proofs of the two systems of equations are different and it is not straightforward to unify them in the quadratically enriched setting. We plan to address this problem in future work.

\subsection{Outline}
The paper is organized as follows. The main focus of Section~\ref{Section:degree} is the definition a pseudo-oriented morphism and the quadratically enriched degree thereof. More straightforward notions of orientation and degree are discussed along the way. There is also a discussion of the behavior of degree under base change. $\Aone$-homotopy theory forces the degree to be an element of $\GW(k)$ under connectivity hypotheses. These connectivity hypotheses are discussed at the end of the section, and then used in Section~\ref{section:examples}.
Section~\ref{sec:local degree} gives a formula for the degree of a map, which is a priori a global invariant, as a sum of local contributions from each point in a fiber. The contribution from a given point is called the local degree at that point. 
Section~\ref{section:CountsRationalCurves} summarizes necessary background concerning moduli spaces of stable maps and recalls results from~\cite{KLSW-relor}. These results are used to define pseudo-orientations of evaluation maps.
Section~\ref{sec:degev} defines the quadratically enriched count of rational curves on $S$ of degree $D$ asociated with the extension fields $\sigma = (L_1,\ldots,L_r)$ as the degree of a pseudo-oriented twisted evaluation map. This degree is compared with that of a symmetrized evaluation map. Section~\ref{sec:geom_formula_local_deg_ev} gives a geometric interpretation of the local degree of the twisted evaluation map at a general rational curve in terms of the field of definition and the nodes of the image curve.
Section~\ref{sec:enumerative interpretation} uses the geometric interpretation of the local degree and results on the existence of appropriate loci in the moduli space of curves to give a purely enumerative computation of the degree of the twisted evaluation map under certain hypotheses.
Section~\ref{section:examples} gives example computations of quadratically enriched counts of rational curves.  

\subsection{Acknowledgements} We thank Jean Fasel, Dan Freed, Fabien Morel, and Rahul Pandharipande for useful discussions. ML is supported by the ERC Grant QUADAG: this paper is part of a project that has received funding from the European Research Council (ERC) under the European Union's Horizon 2020 research and innovation programme (grant agreement No. 832833).\\
\includegraphics[scale=0.08]{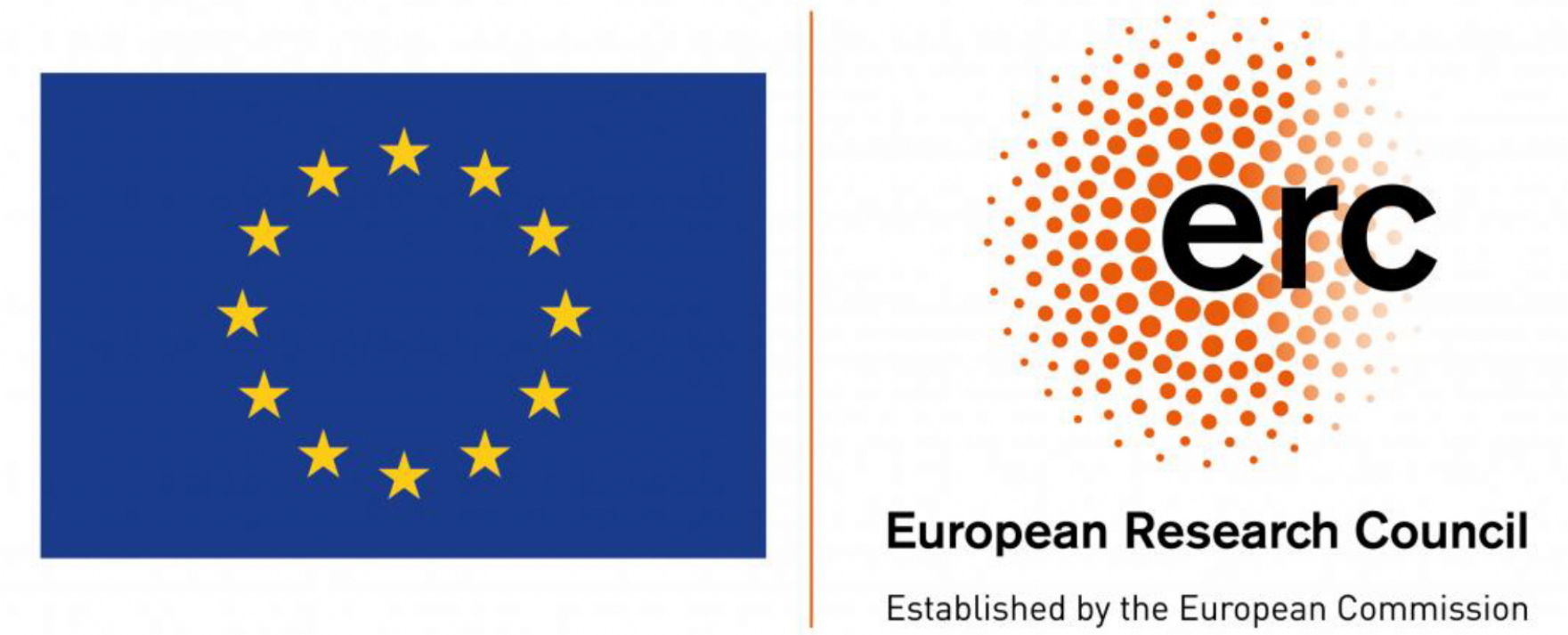}. \\
JS was partially supported by ERC Starting Grant 337560 as well as ISF Grants~569/18 and~1127/22 and the Miriam and Julius Vinik Chair in Mathematics. KW was partially supported by National Science Foundation Awards DMS-1552730, DMS-2001890, DMS-2103838, and DMS-2405191. She also thanks the Newton Institute and the organizers and participants of the special program {\em Homotopy harnessing higher structures} for hospitality while working on this paper, and the {\em A Room of One's Own} initiative for focused writing time.

\section{Degree}\label{Section:degree}
\subsection{Orientations}\label{section:orientations}

Define a local complete intersection morphism $f: X \to Y$ as in \cite[Tag 068E]{stacks-project}. For example, let $i$ be a closed immersion locally determined by a regular sequence and let $\pi$ be a smooth map. The composition $f= \pi \circ i$ is then a local complete intersection morphism. A finite type map between regular schemes is also a local complete intersection morphism \cite[Tag 0E9K]{stacks-project}. For $f: X \to Y$ a local complete intersection morphism, the cotangent complex $L_f$ is perfect \cite[08SL]{stacks-project} and we may form its determinant, which is a line bundle on $X$. (We could view a shift of this line bundle by some integer as the determinant viewed as an element of the derived category, but we don't do this.) Define $\omega_f$ by $$\omega_f := \det L_f.$$

\begin{ex}
For $f= \pi \circ i$, there is a canonical isomorphism \begin{equation}\label{omega_of_composite}\omega_f \cong i^* \omega_{\pi} \otimes \omega_i\end{equation}\cite[Tag 08QX]{stacks-project}. When we additionally assume that $i$ is a closed immersion determined by a regular sequence and $\pi$ is smooth as above, we have canonical isomorphisms $$\omega_i \cong \det (\calI/\calI^2)^* ,$$ \begin{equation}\label{omega_smooth} \omega_{\pi} \cong \det \Omega_\pi,\end{equation} where $\calI$ denotes the ideal sheaf associated to the closed immersion $i$, and $\Omega_\pi$ denotes the sheaf of K\"ahler differentials \cite[08SJ]{stacks-project} \cite[Tag 08R4]{stacks-project}.\footnote{The references treat affine schemes, but the isomorphisms globalize.} 
\end{ex}

\begin{ex}
For $f: X \to Y$ with $X$ and $Y$ smooth $k$-schemes, $\omega_f = \Hom(\det TX, f^* \det TY)$. \end{ex}

\begin{df}\label{df:rofunction}
An orientation for a complete local intersection (lci) morphism $f$ is the choice of an invertible sheaf $L$ on $X$ and an isomorphism $\rho:\omega_f\to L^{\otimes 2}$.

\end{df}

We do not assume $X$ or $Y$ smooth in Definition~\ref{df:rofunction}.

\subsection{Degree of an oriented map}\label{subsection:degree_finite_map}\label{subsection:degree_finite_flat_oriented_lci}

Suppose $f:X \to Y$ is a finite, flat, local complete intersection morphism with relative orientation $\rho:\omega_f\stackrel{\cong}{\to} L^{\otimes 2}$. For simplicity, we assume that $X$ and $Y$ are locally Noetherian. We define the $\mathbb{A}^1$-degree of $f$. For a finite extension $A \to B$ of commutative rings, $\Hom_A(B,A)$ is a $B$-module via the multiplication of $B$ on $B$: $(bf)(b') := f(bb')$. Globalizing, let $f^{\natural}\Hom_{\calO_Y} (f_* \calO_X, \calO_Y)$ denote the corresponding sheaf of $\calO_X$-modules on $X$. Grothendieck--Serre duality in this context can be described as a canonical isomorphism 
\begin{equation}\label{eq:Serre_duality_finite_flat_lci}
\omega_f \cong f^{\natural}\Hom_{\calO_Y} (f_* \calO_X, \calO_Y)
\end{equation} together with the trace map $\Tr_f: f_* \omega_f \to \calO_Y$ defined by the map $f_* f^{\natural}\Hom_{\calO_Y} (f_* \calO_X, \calO_Y) \to \calO_Y$ given by evaluation at $1$ in $\calO_X$ and the isomorphism \eqref{eq:Serre_duality_finite_flat_lci}. Grothendieck--Serre duality says that this trace map satisfies the condition that for all locally free coherent sheaves $M$ on $X$, the composition
\begin{equation}\label{eq:Serre_duality_morphism_version_finite_flat_lci}
f_* \Hom_{\calO_X}(M, \omega_f) \to \Hom_{\calO_Y} (f_* M, f_* \omega_f) \stackrel{\Tr_f}{\to} \Hom_{\calO_Y}(f_* M, \calO_Y)
\end{equation} is an isomorphism. In Equation~\ref{eq:Serre_duality_morphism_version_finite_flat_lci}, the first map is the canonical homomorphism and the second is induced by  $\Tr: f_* \omega_f \to \calO_Y$.  We include a description of the canonical isomorphism \eqref{eq:Serre_duality_finite_flat_lci} satisfying \eqref{eq:Serre_duality_morphism_version_finite_flat_lci} following Scheja and Storch \cite{SchejaResiduen} \cite{scheja}. See also, e.g., \cite[III Section 6]{HartshorneRD}.

Locally, $Y$ can be covered by affines $\Spec A$ such that $f^{-1}(\Spec A) \cong \Spec B$ and $f$ corresponds to the ring map $ A \to A[x_1,\ldots,x_n]/I \cong B $ where $I = (f_1,\ldots,f_n)$. Choose elements $a_{ij}$ in $A[x_1,\ldots,x_n] \otimes_A A[x_1,\ldots,x_n]$ such that
\[
f_i \otimes 1 - 1 \otimes f_i = \sum_j a_{ij} (x_i \otimes 1 - 1 \otimes x_i),
\] and let $\Delta$ in $B \otimes_A B$ denote the image of $\det (a_{ij})$, which is independent of the choice of $a_{ij}$ \cite[Satz 3.1]{scheja} . 

Sending the element $\overline{f_1} \wedge \overline{f_2} \wedge \ldots \wedge \overline{f_n}$ in $\wedge^n (I/I^2)$ to $\Delta$ defines a $B$-module isomorphism $\wedge^n (I/I^2) \to (\Delta)$. By \cite[Satz 1.1]{SchejaResiduen}, this isomorphism is independent of the chosen generators for $I$. On the other hand, the map $B \otimes_A B \to \Hom_A (\Hom_A(B,A),B)$ sending $b_1 \otimes b_2$ to $\alpha \mapsto \alpha(b_1)b_2$ restricts to an isomorphism of $B$-modules from the ideal $(\Delta)$ to $ \Hom_B (\Hom_A(B,A),B)$. See \cite[Satz 3.1, 3.2]{scheja}. Composing, we obtain an isomorphism of $B$-modules $\wedge^n (I/I^2) \to  \Hom_B (\Hom_A(B,A),B)$. Since the relevant $B$ modules are projective of finite rank, we can dualize: the canonical isomorphism $\wedge^n (I/I^2) \cong  \Hom_B (\omega_{B/A},B)$ defines an isomorphism $\varphi: \omega_{B/A} \to \Hom_A(B,A)$ which is independent of the generators of $I$. It follows from \cite[Satz 1.4]{SchejaResiduen} that $\varphi$ is independent of the choice of presentation, allowing us to globalize, producing the canonical isomorphism \eqref{eq:Serre_duality_finite_flat_lci}. To see that \eqref{eq:Serre_duality_morphism_version_finite_flat_lci} holds for all locally free coherent sheaves $M$ on $X$, it suffices to treat the case $M = \calO_X$, which is in \cite[Satz 1.4]{SchejaResiduen}.

Returning to the oriented map $f,$ we have that $f_* L$ is locally free because $f$ is flat. Since $f^*$ is symmetric monoidal, there is a canonical lax monoidal structure on the right adjoint $f_*$ giving a map $f_* L \otimes f_* L \to f_* ( L^{\otimes 2} )$.

\begin{df} \label{GSDdegreedef}
Suppose $f:X \to Y$ is a finite, flat, local complete intersection morphism between locally Noetherian schemes with relative orientation $\rho:\omega_f\stackrel{\cong}{\to} L^{\otimes 2}$. The {\em degree} $\deg f$ of $f$ is the bilinear form $ f_* L \otimes f_* L \to \calO_Y$ given by the composition \begin{equation*}
f_* L \otimes f_* L \to f_* ( L^{\otimes 2} )\stackrel{f_*(\rho^{-1})}{\longrightarrow} f_* \omega_f \stackrel{\Tr_f}{\longrightarrow} \calO_Y
\end{equation*} When we wish to make the orientation explicit, we also write $\deg (f,\rho)$.
\end{df}

\begin{pr}\label{pr:GSDdegreedef-well-defined}
$\deg f$ is symmetric and non-degenerate.
\end{pr}

\begin{proof}

The swap map $L^{\otimes 2} \to L^{\otimes 2}$ defined by taking $\ell \otimes \ell'$ to $\ell' \otimes \ell$ is equal to the identity map, and it follows that $\deg f$ is symmetric.

The adjoint map $f_* L \to \Hom_{\calO_Y} (f_* L , \calO_Y)$ to $\deg f$ is the composition of the pushforward of $\mu_{\rho}: L \to \Hom_{\calO_X}(L, L^{\otimes 2}) \stackrel{\rho}{\cong} \Hom_{\calO_X}(L, \omega_f)$ with the canonical map $f_* \Hom_{\calO_X}(L,  \omega_f) \to \Hom_{\calO_Y}(f_* L, f_* \omega_f)$, followed by the map $\Hom_{\calO_Y}(f_* L, f_* \omega_f) \to \Hom_{\calO_Y}(f_* L,\calO_Y)$ induced by $\Tr_f $. Since $\rho$ is an isomorphism, so is  $\mu_{\rho}$ and $f_*  \mu_{\rho}$. The composition of the following two maps is the isomorphism \eqref{eq:Serre_duality_morphism_version_finite_flat_lci} from Grothendiek--Serre duality .

\end{proof}

\begin{ex}\label{A1deg_finite_etale_map}
A finite \'etale map $f: X \to Y$ admits a canonical relative orientation $\omega_f \cong \calO_X^{\otimes 2}$ and the resulting $\bbb{A}^1$-degree is simply the classical trace form.
\end{ex}

\begin{ex}
Let $f: X \to Y$ be a finite map between smooth $n$-dimensional $k$-schemes with relative orientation $\rho:\omega_f\stackrel{\cong}{\to} L^{\otimes 2}$. Then $f$ is flat by \cite[Theorem 23.1 p.179]{Matsumura_CRT} and lci \cite[Tag 0E9K]{stacks-project}. We thus have defined $\deg f$.
\end{ex}

This degree commutes with base change. Let \begin{equation}\label{basechangeff'}\xymatrix{ X' \ar[r]^{g'} \ar[d]^{f'}& X \ar[d]^f\\ Y'\ar[r]^g & Y }\end{equation} be a pullback diagram with $f$ a finite, flat, local complete intersection morphism oriented by $\rho$. If $g$ is flat, then $f'$ is automatically a local complete intersection morphism by \cite[Tag 069I]{stacks-project}. However, in our discussion of base change, we will not assume $g$ to be flat, and instead assume that $f'$ is a local complete intersection morphism. Since $f$ is flat, the square \eqref{basechangeff'} and \cite[\href{https://stacks.math.columbia.edu/tag/08QQ}{08QQ}]{stacks-project} define a canonical isomorphism \begin{equation}\label{basechangeomegaf'ispullbackf} \omega_{f'} \cong (g')^*\omega_{f}.\end{equation} Therefore $(g')^*\rho$ determines an orientation of $f'$.

\begin{pr}\label{pr:degree_finite_commutes_basechange}
Let \eqref{basechangeff'} be a pullback square such that $f$ is a finite, flat, local complete intersection morphism oriented by $\rho$. Suppose that  $f'$ is a local complete intersection morphism, and that $X$,$Y$,$X'$, and $Y'$ are locally Noetherian. Then we have the equality in $\sGW(Y')$ $$\deg(f', (g')^*\rho) = g^* \deg(f, \rho).$$
\end{pr}

\begin{proof}
Let $L$ denote the line bundle on $X$ associated to the orientation $\rho$, i.e., the orientation of $f$ is the isomorphism $\rho: L^{\otimes 2} \to \omega_f$. The natural map from cohomology and base change and the isomorphism \eqref{basechangeomegaf'ispullbackf} determine the commutative diagram $$\xymatrix{ g^* f_* L^{\otimes 2} \ar[r]^{\cong} \ar[d]_{g^* f_* \rho} &f'_* (g')^* L^{\otimes 2} \ar[d]^{f'_* (g')^*\rho} \ar[r]^{\cong}& \ar[d]^{f'_* (g')^* \rho} f'_* ((g')^* L)^{\otimes 2} \\ g^* f_* \omega_f \ar[r]^{\cong} & f'_* (g')^*\omega_f \ar[r]^{\cong}& f'_* \omega_{f'}},$$ where the horizontal morphisms are isomorphisms. The claim follows by the commutativity \cite[0B6J]{stacks-project} of $$\xymatrix{ g^* f_* \omega_f \ar[r] \ar[d]_{g^*\Tr_f} &  f'_* \omega_{f'}\ar[d]^{\Tr_{f'}}\\ g^* \calO_X \ar[r]^{\cong} & \calO_{X'}}$$
\end{proof}

\subsection{Degree of a map oriented away from codimension \texorpdfstring{$2$}{two}}\label{subsection:Globaldegreemapsmpro}
We use a theorem of Ojanguren and Panin to relax the orientation hypothesis required for the definition of the degree in Section \ref{subsection:degree_finite_map}, obtaining a degree  for a map oriented away from codimension $2$ (see Definition \ref{df:or_away_codim_c}), which is a section of the Grothendieck--Witt sheaf.

Let $k$ be a field. The Grothendieck--Witt sheaf $\sGW$ is a sheaf on smooth $k$-schemes with the Nisnevich topology which can be defined as the sheafification of the functor sending $Y$ to the group completion of the semi-ring of isomorphism classes of locally free sheaves $V$ on $Y$ equipped with a non-degenerate symmetric bilinear form $V \times V \to \calO_Y$. It has a construction given in \cite[Chapter 3]{morel}. The sheaf $\sGW$ is unramified by the theorem of Ojanguren and Panin \cite{PaninOjanguren}. In particular, suppose $U \subseteq Y$ is an open subset of $Y$ with complement of codimension at least $2$. Then a locally free sheaf on $U$ equipped with a symmetric, non-degenerate bilinear form determines an element of $\sGW(Y)$. (A complete definition of an unramified sheaf is in \cite[Chapter 2, Def 2.1, Remark 2.4]{morel}.) 

For this reason, we may throw away codimension $2$ closed subsets. This allows us to both replace the assumption on the existence of an orientation of $f: X \to Y$ with the assumption that an appropriate restriction has an orientation, as well as obtain restrictions of $f$ which are finite and flat, under assumptions of generic finiteness or generic \'etaleness. (We will then apply Definition~\ref{GSDdegreedef} to obtain our degree.)

A map $f: X \to Y$ between integral schemes is said to be {\em generically finite} if $f$ is dominant (meaning its set-theoretic image is dense) and the associated extension of function fields is finite. For example, if the differential $df$ of a map $f$ between connected, smooth $n$-dimensional $k$-schemes is injective at one point, then $f$ is generically finite. If $X$ and $Y$ have more than one connected components, which are all integral, say that $f$ is generically finite if $f$ is dominant, only finitely many components of $X$ map to each component of $Y$, and for each component of $X$, its function field is a finite extension of the function field of the component of $Y$ containing its image. We include in the definition of $f$ being generically finite that the connected components of $X$ and $Y$ are integral.  If $f: X \to Y$ is generically finite, there is a dense open $U \subset Y$ such that $f^{-1}(U) \to U$ is finite by \cite[Tag 02NX]{stacks-project}. We include a proof of the following well-known proposition for completeness. 

\begin{pr}\label{nonfinite_locus_is_codim2}
Let $Y$ be a smooth scheme over a field $k$ or discrete valuation ring $\Lambda$. Let $f: X \to Y$ be a proper, generically finite map. Then there exists a codimension $2$ subset $Z$ of $Y$ such that $f$ is finite and flat when pulled back to the complement of $Z$.
\end{pr}

\begin{proof}
For any point $x$ of $X$, the map $\calO_{Y,f(x)} \to \calO_{X,x}$ is injective because $f$ is dominant. Let $x$ be such that $y=f(x)$ is codimension $1$ in $Y$. Since $y$ is codimension $1$ and $Y$ is smooth, the ring $\calO_{Y,y}$ is a discrete valuation ring. Since the connected components of $X$ are integral (this is part of the definition of $f$ being generically finite for us), $\calO_{X,x}$ is an integral domain. In particular, $\calO_{X,x}$ is torsion free as a $\calO_{Y,f(x)}$-module, and it follows that $f$ is flat at $x$ because $\calO_{Y,y}$ is a principle ideal domain.

Let $U$ be the subset of points $y$ of $Y$ such that $f$ is flat at all the points $x$ in $f^{-1}(y)$. We claim that $U$ is open. Suppose $y_0$ specializes to $y_1$ in $Y$ and that $y_1$ is in $U$. Let $x_0$ be a point of $f^{-1}(y_0)$. Let $\overline{x_0}$ denote the closure of $x_0$. Since $f$ is proper, $f(\overline{x_0})$ is closed. $f(\overline{x_0})$ contains $y_0$ and therefore $y_1$ by construction. Thus we can choose $x_1$ such that $x_0$ specializes to $x_1$ and $f(x_1) = y_1$. We thus have a flat extension $\calO_{Y,y_1} \subseteq \calO_{X,x_1}$ and ideals $p_{x_0}$ and $p_{y_0}$ in $\calO_{X,x_1}$ and $\calO_{Y,y_1}$ respectively such that $p_{x_0} \cap\calO_{Y,y_1} = p_{y_0} $. Since localization is flat, it follows that $\calO_{Y,y_0} \subseteq \calO_{X,x_0}$ is flat and $U$ is open as claimed.

By the above, $U$ contains all the points of codimension $1$, whence its complement $Z$ is closed of codimension at least $2$. Let $f^0$ denote the pullback of $f$ to $f^{-1}(U)$. Since proper maps are stable under base change, $f^0$ is proper. $f^0$ is flat by construction. Thus the fibers are equidimensional \cite[Theorem 15.1]{Matsumura_CRT} and \cite[Tag 02JB]{stacks-project}. Since $f$ is generically finite, the fibers of $f^0$ are of dimension $0$ and therefore finite. Thus $f^0$ is a proper map with finite fibers and therefore finite \cite[Tag 02OG]{stacks-project}.
\end{proof}

\begin{df}\label{df:or_away_codim_c}
Let $B$ be the spectrum of a field $k$ or discrete valuation ring $\Lambda$. Let $f\colon X \to Y$ be a $B$-map from an Artin stack $X$ to a smooth $B$-scheme $Y$. $f$ is said to be {\em orientable away from codimension $c$} if \begin{itemize}
\item there exists a dense open subset $U \subseteq Y$  such that $Y -U$ has codimension $\geq c$,
\item $f^{-1}(U)$ is a scheme with integral connected components,
\item the restriction $f\vert_{f^{-1}(U)}\colon f^{-1}(U) \to U$ of $f$ is a generically finite, proper, local complete intersection morphism,
\item and there is a line bundle $L$ on $f^{-1}(U)$ together with an isomorphism $$\rho\colon \omega_{f \vert_{f^{-1}(U)}}\stackrel{\cong}{\to}  L^{\otimes 2},$$ orienting $f\vert_{f^{-1}(U)}\colon f^{-1}(U) \to U$.
\end{itemize} An {\em orientation of $f$ away from codimension $c$} is a choice of $(U,L,\rho)$ as above.  
\end{df}

\begin{rmk}\label{fassumption}\label{Rmk:degree_for_f_fassumption}
If  $f\vert_{f^{-1}(U)}\colon f^{-1}(U) \to U$ is generically \'etale, then $f\vert_{f^{-1}(U)}$ is generically finite by \cite[Th\'eor\`eme 16.6.1(c')]{egaIV_4}. When we give an orientation away from codimension $c$ as in Definition~\ref{df:or_away_codim_c}, we will often show that $f\vert_{f^{-1}(U)}$ is generically \'etale to show $f\vert_{f^{-1}(U)}$ is generically finite.

\end{rmk}

\begin{pr}\label{pr:deg(f)f_non_proper_in_sGW(Y)}\label{fassumption_nonproper}

Suppose $k$ is a field, and $f: X \to Y$ is a $k$-morphism from an Artin stack $X$ to a smooth $k$-scheme $Y$, oriented away from codimension $2$ by $(U,L,\rho)$. Then there is a unique 
\[
\deg f \in \sGW(Y)
\] compatible with Definition~\ref{GSDdegreedef} in the following sense: For any open subset $U' \subset U$ with complement of codimension greater than or equal to $2$ such that $f \vert_{f^{-1}(U')}$ is a finite, flat, local complete intersection morphism, the image of $\deg f$ under the restriction map $\sGW(Y) \to \sGW(U')$ is $\deg(f \vert_{f^{-1}(U')}, \rho|_{f^{-1}(U')})$ as defined in Definition~\ref{GSDdegreedef}. Moreover, there exists such an open set~$U'.$
\end{pr}

\begin{proof}
By Proposition~\ref{nonfinite_locus_is_codim2}, there exists an open subset $U' \subset U$ with complement of codimension greater than or equal to $2$ such that $f \vert_{f^{-1}(U')}$ is finite and flat. Note that $f \vert_{f^{-1}(U')}$ is also a proper, local complete intersection morphism, oriented by the restriction of $\rho$. We define $\deg(f \vert_{f^{-1}(U')}, \rho)$ in $\GW(U')$ with Definition~\ref{GSDdegreedef}. Since $\sGW$ is unramified and the complement of $U' \subseteq Y$ is codimension greater than or equal to $2$, there is a unique $\deg(f,\rho, U')$ in $\sGW(Y)$ mapping to $\deg(f \vert_{f^{-1}(U')}, \rho)$ under the restriction. Given a second choice $U''$ for the open subset $U'$, we have that $U'' \cap U' \subseteq U'$ and $\deg(f,\rho, U') \mapsto \deg(f,\rho, U'' \cap U')$ by construction. Since the restriction map $\sGW(Y) \to \sGW(U' \cap U'')$ is an isomorphism, we have that $\deg(f,\rho, U') = \deg(f,\rho, U'')$, so we may define $\deg(f, \rho) : = \deg(f,\rho, U')$ in $\sGW(Y)$ and this definition is independent of the choice of $U'$.
\end{proof}

\begin{ex}
Let $f: X \to Y$ be a generically finite map where $Y$ is a proper, smooth $n$-dimensional $k$-scheme, and $X$ is a regular, proper $k$-scheme of dimension $n$. Let $U$ be an open subset of $Y$ such that $Y-U$ has codimension at least two. $f$ is proper because $X$ and $Y$ are, whence the base change $f\vert_{f^{-1}(U)}$ is as well, so in particular, $f\vert_{f^{-1}(U)}$ is finite type. The map $f\vert_{f^{-1}(U)}: f^{-1}(U) \to U$ is a local complete intersection morphism because it is a finite type map between regular schemes  \cite[Tag 0E9K]{stacks-project}, so it makes sense to speak of a relative orientation of $f\vert_{f^{-1}(U)}$. A relative orientation of $f\vert_{f^{-1}(U)}$ then defines the degree $\deg f$ in $\sGW(Y)$.

\end{ex}

\begin{ex}
Let $X$ be a geometrically normal, proper scheme over $k$ of dimension $n$. Let $f: X \to Y$ be a generically finite map where $Y$ is a proper, smooth $n$-dimensional $k$-scheme. The assumption that $X$ is geometrically normal implies that $X/k$ is smooth at codimension $1$ points  \cite[Tags 031S, 038X]{stacks-project}. Since the set of points of $X$ where $X \to \Spec k$ is smooth is open, contains the points of codimension $1$, and $X$ is regular at any point where $X/k$ is smooth, there is necessarily an open subset $U\subset Y$ such that $Y-U$ has codimension greater than or equal to $2$ and $f^{-1}(U)$ is regular. An orientation of the pullback of $f$ to $f^{-1}(U)$ thus defines $\deg f$ in $\sGW(Y)$.
\end{ex}

\begin{df}\label{df:deg(f)f_non_proper_in_sGW(Y)}\label{df:deg(f)_in_sGW(Y)}
For $f: X \to Y$ oriented away from codimension $2$, define $\deg f$ in $\sGW(Y)$ by Proposition \ref{fassumption_nonproper}. If we wish to make the choice of orientation away from codimension $2$ explicit, we write $\deg(f,\rho)$ or $\deg(f,U,L,\rho)$ for $\deg f$.
\end{df}

\begin{df}\label{df:degf(y)}
The degree $\deg (f,\rho)(y)$ of $f$ at a point $y$ of $Y$ is defined to be the pullback of $\deg f$ along $y: \Spec k(y) \to Y$, so $\deg (f,\rho)(y)$ in $\GW(k(y))$. When the relative orientation is clear from context, we also write $\deg f(y)$.
\end{df}

\subsection{Degree of a pseudo-oriented map}
\label{subsection:deg_chap_p_lifting_data}

We define a degree for a map oriented away from codimension $1$ in the presence of additional data, which we define to be a pseudo-orientation. See Definition \ref{df:pseudo-oriented}.

\begin{df}\label{df:equiv_or_away_codim1}
Two orientations $(U_1,L_1,\rho_1)$ and $(U_2,L_2,\rho_2)$ of $f$ away from codimension 1 are said to be {\em equivalent} if there is a dense open subset $U$ of $U_1 \cap U_2$ and an isomorphism $\psi: L_1\vert_{f^{-1}(U_1)} \to L_2\vert_{f^{-1}(U_2)}$ such that $\psi^{\otimes 2} \rho_1 = \rho_2$.
\end{df}

\begin{df}\label{df:lifting_data}
Let $k$ be a field and let $f\colon X \to Y$ be a $k$-map from an Artin stack $X$ to a smooth $k$-scheme $Y$. 
Let $(U,L,\rho)$ be an orientation of $f$ away from codimension 1. {\em Lifting data} for $(f, U,L,\rho)$  is as follows:\begin{itemize} 
\item a discrete valuation ring $\Lambda$ with residue field $k$, 
\item a lifting of $Y \to \Spec k$ to a smooth, finite type, morphism $\calY\to \Spec \Lambda$, 
\item and an open subset $\calU\subset \calY$ such that $\calU \cap Y$ is dense in $U$ and the intersection of the complement $\calY-\calU$ with the generic fiber is codimension $\geq 2$. 
\end{itemize}
Lifting data furthermore includes:
\begin{itemize}
\item a lifting of the map $f \vert_{f^{-1}(\calU \cap U)}: f^{-1}(\calU \cap U) \to \calU \cap U$ to a proper, generically finite, local complete intersection morphism $\mathfrak{f}\colon \calX \to \calU$ of $\Lambda$-schemes. \item a lift of $L\vert_{f^{-1}(\calU \cap U)}$ to a line bundle $\calL$ on $\calX$ 
\item and a lift of $\rho\vert_{f^{-1}(\calU \cap U)}$ to an isomorphism $\rhol: \omega_{\mathfrak{f}}\stackrel{\cong}{\to} \calL^{\otimes 2}$.
\end{itemize}
\end{df}

\begin{df}\label{df:pseudo-oriented}
Let $k$ be a field and let $f\colon X \to Y$ be a $k$-map from an Artin stack $X$ to a smooth $k$-scheme $Y$. A {\em pseudo-orientation} of $f$ is an equivalence class of orientations away from codimension 1 such that there exists a lifting datum for a representative.
\end{df}

\begin{ex}\label{ex:or_away_codim_2_is_pseudo-orientation}
    Let $f\colon X \to Y$ be a map over a field $k$ from an Artin stack $X$ to a smooth $k$-scheme $Y$. An orientation $(U,L,\rho)$ of $f$ away from codimension $2$ as in Definition~\ref{df:or_away_codim_c} determines a canonical pseudo-orientation: let $\Lambda = k[[t]]$, and let $(\calY, \calU, \mathfrak{f}, \calL, \rhol)$ be the pullback of $(Y,U,f,L,\rho)$.  
\end{ex}

Note that a pseudo-orientation for a map $f$ does not include the lifting data; such lifting data only has to exist.

\begin{setting}\label{fassumption_charpnonproper} 
Let $k$ be a field. Let $f\colon X \to Y$ be a pseudo-oriented map from an Artin stack $X$ to a smooth $k$-scheme $Y$.
\end{setting}

In this section, we define $\deg f \in \sGW(Y)$ in Setting~\ref{fassumption_charpnonproper}. To do this, we use further purity results.

 For $A$ a commutative ring, we let $W(A) = \GW(A)/M$ denote the Witt group, defined to be the group completion of the isomorphism classes of non-degenerate symmetric bilinear forms over $A$, modulo the ideal of metabolic forms. See for example \cite{milnor73}. In place of constructing an unramified sheaf $\sGW$ on smooth schemes over a discrete valuation ring, we use the following results on the Witt and Grothendieck--Witt groups to extend certain sections of $\sGW(U)$ to a (necessarily unique) section of $\sGW(Y)$ for $U \subset Y$ dense.

\begin{df} Let $\calO$ be a regular local ring with quotient field $K$. Let $\Spec\calO^{(1)}$ be the set of height one prime ideals of $\calO$. For $P\in \Spec\calO^{(1)}$, let $\calO_P\subset K$ denote the localization of $\calO$ at $P$. For $A\subset K$ a subring, let $\overline{W}(A)$ be the image of $W(A)$ in $W(K)$.  We say that {\em purity holds for $W(\calO)$} if
\[
\overline{W}(\calO)=\cap_{P\in \Spec\calO^{(1)}}\overline{W}(\calO_P)
\] Similarly, let $\overline{\GW}$ denote the image of $\GW(A)$ in $\GW(K)$, and say that {\em purity holds for $\GW(\calO)$} if
\[
\overline{\GW}(\calO)=\cap_{P\in \Spec\calO^{(1)}}\overline{\GW}(\calO_P).
\] 
\end{df}

\begin{tm}[Colliot-Th\'el\`ene and Sansuc \hbox{\cite[Corollaire 2.5]{CT-S}}] \label{thm:CTS} Let $\calO$ be a regular ring of dimension $\le 2$ containing $1/2$. Then purity holds for $W(\calO)$.
\end{tm}

\begin{lm}\label{lm:GWpurity}
Let $\calO$ be a regular ring of dimension $\le 2$ containing $1/2$. Then purity holds for $\GW(\calO)$. 
\end{lm}

\begin{proof}
Let $K$ denote the quotient field of $\calO$. Consider $q$ in $\cap_{P\in \Spec\calO^{(1)}}\overline{\GW}(\calO_P)$. By Theorem~\ref{thm:CTS}, the image $\overline{q}$ of $q$ in $W(K)$ lies in $\overline{W}(\calO)$. Since $\GW(\calO) \to W(\calO)$ is surjective, there is $\tilde{q}$ in $\GW(\calO)$ with image $\overline{q}$. 

For a local ring containing $1/2$, the metabolic forms in the Grothendieck--Witt group are the same as the ideal generated by the hyperbolic form $\langle 1 \rangle + \langle -1 \rangle $ \cite[Chapter 1, 6.3]{milnor73}. It follows that the difference between $q$ and the image of $\tilde{q}$ is $n(\langle 1 \rangle + \langle -1 \rangle)\in \GW(K)$ for some $n \in \mathbb{Z}$. Thus $\tilde{q} + n(\langle 1 \rangle + \langle -1 \rangle)\in \GW(\calO)$ has image $q$, showing
\[
\overline{\GW}(\calO)\supset \cap_{P\in \Spec\calO^{(1)}}\overline{\GW}(\calO_P).
\] which implies purity. 
\end{proof}

\begin{tm}[Knebusch \hbox{\cite[Satz 11.1.1]{Knebusch}}]\label{thm:Knebusch}  
Let $\calO$ be a Dedekind domain with $1/2\in \calO$. Let $K$ be the quotient field of $\calO$. Then the map $W(\calO)\to W(K)$ is injective.
\end{tm}

\begin{lm}\label{lm:GWDVRinjectK}
Let $\calO$ be a local, Dedekind domain with $1/2\in \calO$. Let $K$ be the quotient field of $\calO$. Then the map $\GW(\calO)\to \GW(K)$ is injective.
\end{lm}

\begin{proof}
Suppose $q$ in $\GW(\calO)$ lies in the kernel of $\GW(\calO)\to \GW(K)$. Then $\rank q = 0$. By Theorem~\ref{thm:Knebusch}, we also have that $q$ lies in the kernel of $\GW(\calO)\to W(\calO)$. Thus $q$ is rank $0$ and $q$ lies in the ideal generated by metabolic forms. By \cite[Chapter 1, 6.3]{milnor73}, the metabolic forms in the Grothendieck--Witt ring of a local ring containing $1/2$ are the same as the ideal generated by the hyperbolic form $\langle 1 \rangle + \langle -1 \rangle $. By \cite[Theorem 1.3]{RogersSchlichtingGW}, rank $1$ forms generate $\GW(\calO)$ and  $$\langle 1 \rangle + \langle -1 \rangle =( \langle 1 \rangle + \langle -1 \rangle) \langle r \rangle $$ for any $r$ in $\calO$. This uses that $\calO$ is local. It follows that the ideal in $\GW(\calO)$ generated by the hyperbolic form is the same as the subgroup generated by the hyperbolic form. Thus an element of this ideal of rank $0$ is $0$.
\end{proof}

\begin{rmk}
Further results on purity for Witt groups can be found in for example \cite{PaninOjanguren} \cite{BalmerGillePaninWalter-Gersten} \cite{BalmerWalter-Gersten}.
\end{rmk}

As above, let $\sGW$ denote the unramified sheaf on smooth $k$-schemes with $k$ a field.

\begin{pr}\label{pr:GWExtension}
Let $\Lambda$ be a discrete valuation ring with residue field $k$ of characteristic $\neq 2$. Let $\pi:\calY\to \Spec \Lambda$ be a smooth morphism of finite type, and let $\calU\subset \calY$ be an open subscheme satisfying the properties that \begin{itemize}
\item the intersection $\calU_k$ with the closed fiber $\calY_k$ is dense in $\calY_k$,
\item and the intersection of the complement $\calY-\calU$ with the general fiber is codimension $\geq 2$.
\end{itemize} Let $q: \calE \times \calE \to \calO_{\calU}$ be a symmetric nondegenerate bilinear form over $\calU$. Then the restriction of  $q$ to $\sGW(\calU_k)$ extends uniquely to a section in $\sGW(\calY_k)$.
\end{pr}

\begin{proof} Let $x$ be a codimension one point of $\calY_k$, which we consider as a codimension two point of $\calY$,  and let $\calO=\calO_{\calY, x}$. Let $\bar{\eta}$ be a generic point of $\calU_k$ in the connected component containing $x$. Let $L$ be the field of rational functions on $\calY$, and let $K$ be the ring  of rational functions on $\calY_k$. Since $\calU_k$ is dense in $\calY_k$, $K$ is also the ring of rational functions on $\calU_k$.  Since  $\calO$ is a regular ring of dimension two, and $\calU$ contains all points of $\calY$ of codimension~$1$, it follows from Lemma~\ref{lm:GWpurity} that $q$ is in the image of   $\GW(\calO)$ in $\GW(L)$.  Moreover, by Lemma~\ref{lm:GWDVRinjectK}, the map $\GW(\calO_{\calY, \bar{\eta}})\to \GW(L)$ is injective, so the restriction of $q$ to $\GW(\calO_{\calY, \bar{\eta}})$ is in the image of $\GW(\calO)\to \GW(\calO_{\calY, \bar{\eta}})$. Restricting to the fiber over $\Spec k$,  this implies that the image $\bar{q}$ of $q$ in $\GW(K)$ is in the image of $\GW(\calO_{\calY_k, x})$. Since $x$ was an arbitrary codimension one point of $\calY_k$, it follows that $\bar{q}$ extends uniquely to a section of $\sGW$ over $\calY_k$ because $\sGW$ is unramified.
\end{proof}

\begin{tm}\label{tm:A1Degree_pseudo-oriented}
Let $k$ be a field. Let $f\colon X \to Y$ be a map from an Artin stack $X$ to a smooth $k$-scheme $Y$ equipped with a pseudo-orientation. Then there is a unique section 
\[
\deg f \in \sGW(Y)
\] compatible with the pseudo-orientation in the following sense. The pseudo-orientation includes an equivalence class of orientations of $f$ away from codimension $1$. For any representative $(U,L,\rho)$ of the this equivalence class, the restriction of $\deg f$ to $U$ is the $\mathbb{A}^1$-degree of $f\vert_{f^{-1}(U)}: f^{-1}(U) \to U$ in the sense of Definition~\ref{df:deg(f)_in_sGW(Y)}.
\end{tm}

\begin{proof}
Because $f$ is pseudo-oriented, we may choose a representative $(U,L,\rho)$ of the equivalence class of orientation of $f$ away from codimension $1$ for which there exists lifting data 
\[
(\Lambda, \calU \subset \calY, \mathfrak{f}\colon \calX \to \calU, \calL, \tilde{\rho}: \omega_{\mathfrak{f}}\stackrel{\cong}{\to} \calL^{\otimes 2}).
\] Since $\mathfrak{f}$ is a generically finite, proper morphism, we can find an open dense subset $\calU'$ of $\calU$ such that the restriction of $\mathfrak{f}$ to $\mathfrak{f}^{-1}(\calU')$ is finite and flat by Proposition~\ref{nonfinite_locus_is_codim2}. Then $\mathfrak{f}\vert_{\mathfrak{f}^{-1}(\calU')}$ is a finite, flat, lci morphism oriented by $\tilde{\rho}$. Let $\deg(\mathfrak{f}\vert_{\mathfrak{f}^{-1}(\calU')}, \tilde{\rho})$ be the bilinear form on $\calU'$ of Definition~\ref{GSDdegreedef}, which is symmetric and non-degenerate by Proposition~\ref{pr:GSDdegreedef-well-defined}. By Proposition~\ref{pr:GWExtension}, the restriction of $\deg(\mathfrak{f}\vert_{\mathfrak{f}^{-1}(\calU')}, \tilde{\rho})$ to $\sGW(\calU'_k)$ extends to a unique section of $\sGW(Y)$, which we call $\deg f$. This establishes the uniqueness of $\deg f$. 

It remains to show that for any other representative $(U_2,L_2, \rho_2)$ of the orientation away from codimension $1$, we have that the restriction of $\deg f$ to $U_2$ is $\deg(f: f^{-1}(U_2) \to U_2 ,\rho_2)$. Because $(U_2,L_2, \rho_2)$ and $(U,L,\rho)$ are equivalent orientations away from codimension $1$, there is a dense open subset $W \subset U \cap U_2$ and an isomorphism $\psi: L\vert_{f^{-1}(W)} \to L_2\vert_{f^{-1}(W)}$ such that $\psi^{\otimes 2} \rho = \rho_2$. The isomorphism $\psi$ determines an isomorphism between the bilinear forms $\deg (f: f^{-1}(W) \to W ,\rho)$  and $\deg (f: f^{-1}(W) \to W ,\rho_2)$ of Definition~\ref{GSDdegreedef}. Since the restriction of $\deg f$ to $W$ is $\deg (f: f^{-1}(W) \to W ,\rho)$ and $\sGW(U_2) \to \sGW(W)$ is injective, the claim follows.

\end{proof}

\begin{df}\label{df:A1_deg_charp}
For $f: X \to Y$ as in Setting~\ref{fassumption_charpnonproper}, define the $\Aone$-degree, $\deg(f)$ in $\sGW(Y)$, by Theorem~\ref{tm:A1Degree_pseudo-oriented}.
\end{df}

\begin{pr}\label{pr:deg-pseudo-or-comp-pullback}
    Let $k$ be a field and let 
    \[
    \xymatrix{ X' \ar[r]^{g'} \ar[d]_{f'}& X \ar[d]^f\\ Y'\ar[r]^g & Y }
    \] be a pullback square where $f$ and $f'$ are pseudo-oriented maps over $k$ from Artin stacks to smooth $k$-schemes such that the pullback of the orientation of $f$ away from codimension $1$ is equivalent to the orientation of $f'$ away from codimension $1$. Then we have the equality in $\sGW(Y')$ $$\deg(f') = g^* \deg(f).$$
\end{pr}

\begin{proof}
    Since $\sGW$ is unramified, it suffices to show the claim after restriction to a dense open of $Y'$. This follows from Propositions~\ref{pr:degree_finite_commutes_basechange} and \ref{nonfinite_locus_is_codim2}. 
\end{proof}

\subsection{\texorpdfstring{$\GW(\pi^{\mathbb{A}^1}_0(Y))$}{Section}-valued global degree}\label{subsection:GWk-valued_global_degree}

We discuss several notions of connectivity on a smooth $k$-scheme $Y$ ($\bbb{A}^1$-connectivity, $\bbb{A}^1$-chain connectivity, $\bbb{A}^1$-odd chain connectivity) allowing one to associate an element of $\GW(k)$ to a section of $\sGW(Y)$. The notions of $\bbb{A}^1$-connectivity and $\bbb{A}^1$-chain connectivity are well-studied within the $\bbb{A}^1$-homotopy theory communities and the resulting Proposition~\ref{pr:htpyinvtsheaf_iso_Xtopi0} and Lemma~\ref{lm:phi_iso_A1_invt_Nis_sheaf_sets} are well-known. We include proofs for completeness. The notion of $\bbb{A}^1$-odd chain connectivity (Definition~\ref{df:A1oddextended_chain_connected}) is possible to check concretely in examples, so we include Lemma~\ref{lm:A1oddextendedchainconnected_inv_property} as well.

Our interest in connectivity arises as in algebraic topology. The degree of a map $M \to N$ between smooth, oriented, compact $n$-dimensional manifolds is an integer when $N$ is connected. Without assuming that $N$ is connected, the degree can be viewed as an integer valued function on the connected components $\pi_0(N)$ of $N$.  The purpose of this section is to provide the analogous results in our algebraic setting. For example, the $\GW(Y)$-valued degree of Definition \ref{df:deg(f)f_non_proper_in_sGW(Y)} (or that of Definition \ref{GSDdegreedef} or \ref{df:deg(f)_in_sGW(Y)}) is pulled back from a unique element of the Grothendieck--Witt group $\GW(k)$ of $k$ when $Y$ is appropriately connected in an algebraic sense, for example, when $Y$ is $\mathbb{A}^1$-connected. More generally, this degree is a section of $\GW(\pi^{\mathbb{A}^1}_0(Y))$, where $\pi^{\mathbb{A}^1}_0(Y)$ denotes the sheaf of $\bbb{A}^1$-connected components. We give the needed definitions and notations.

For smooth $k$-schemes, or more generally, for simplicial presheaves on smooth $k$-schemes, $X$ and $Y$, let $[X,Y]_{\bbb{A}^1}$ denote the set of $\bbb{A}^1$-homotopy classes of maps from $X$ to $Y$ \cite{morelvoevodsky1998}. Let $\pi_0^{\mathbb{A}^1}(X)$ denote the Nisnevich-sheafification of the presheaf taking a smooth $k$-scheme $U$ to $[U,X]_{\bbb{A}^1}$. There is a canonical map
\[
\phi_X: X \to \pi_0^{\mathbb{A}^1}(X).
\] For example,  $\phi_k : \Spec k \to \pi_0^{\mathbb{A}^1}(\Spec k) $ is an isomorphism \cite[Definition 2.1.4]{AsokMorel}.

\begin{df}
$X$ is {\em $\bbb{A}^1$-connected} if the canonical map $\pi_0^{\mathbb{A}^1}(X) \to \pi_0^{\mathbb{A}^1}(\Spec k) \stackrel{\phi_k}{\cong}\Spec k$ is an isomorphism.
\end{df}

\begin{ex}\label{ex:A1-connected} $\bbb{A}^1$-connected smooth schemes include the following:
\begin{enumerate}
\item \label{ex:covered_by_affines_implies_A1_connected}
\cite[Lemma 2.2.11, Lemma 2.2.5]{AsokMorel} If $X$ is a smooth $k$-variety that is covered by finitely many affine spaces $\bbb{A}^n_k$, whose pairwise intersections all contain a $k$-point, then $X$ is $\bbb{A}^1$-connected.
\item  Asok and Morel \cite[Corollary 2.3.7]{AsokMorel} show that a smooth proper surface over $k$ which is rational over $k$ is $\A^1$-connected.
\item 
\cite[Example 1.33, 1.35]{Kollar-Rational_and_nearly} A smooth cubic surface is rational over $k$ if it contains two skew lines over $k$ or two conjugate skew lines defined over $k(\sqrt{a})$ for some degree $2$ extension $k \subset k(\sqrt{a})$. It then follows \cite[Exercise 1.34, Example 1.35]{Kollar-Rational_and_nearly} that \[ x^2y+y^2z+z^2w + w^2 x = 0\] and \[x^3+y^3+z^3+w^3=0\] determine $\A^1$-connected smooth cubic surfaces over fields of characteristic not $2$ or $3$.
\end{enumerate}
\end{ex}

\begin{rmk}
\cite[Example 2.1.6]{AsokMorel} If $X$ is a smooth $\bbb{A}^1$-connected $k$-scheme, then $X$ has a rational point. For example, $\Spec L$ is not an $\bbb{A}^1$-connected $k$-scheme for $k \subset L$ a finite separable extension.
\end{rmk}

A Nisnevich sheaf $\calF$ of sets is said to be {\em $\mathbb{A}^1$-homotopy invariant} if the projection $U \times \bbb{A}^1 \to U$ induces bijection $ \calF(U) \to \calF(U \times \bbb{A}^1)$ for all smooth schemes $U$. For a smooth $k$-scheme $X$, the map $\phi_X$ induces the map $$\phi_X^*: \Hom(\pi_0^{\mathbb{A}^1}(X), \calF) \to \Hom(X, \calF),$$ where, in this expression, $\Hom$ denotes the set of maps of presheaves on smooth $k$-schemes. The following proposition is known to experts, but we include a proof for completeness.

\begin{pr}\label{pr:htpyinvtsheaf_iso_Xtopi0}\label{pr:A1connected_implies_GW=GW(k)}
Let $\calF$ be an $\mathbb{A}^1$-homotopy invariant Nisnevich sheaf of sets and let $X$ be a smooth $k$-scheme. Then \begin{enumerate}
\item\label{lm:htpyinvtsheaf_iso_Xtopi0:gen} $\phi_X^*$ is a bijection.
\item \label{lm:htpyinvtsheaf_iso_Xtopi0:XA1conn} If in addition $X$ is $\mathbb{A}^1$-connected, then the canonical map $\calF(k) \to \calF(X)$ is an isomorphism.
\end{enumerate}

The claims \eqref{lm:htpyinvtsheaf_iso_Xtopi0:gen} and \eqref{lm:htpyinvtsheaf_iso_Xtopi0:XA1conn} hold for $\calF = \sGW$.
\end{pr}

\begin{proof}
The sheaf $\sGW$ is $\mathbb{A}^1$-homotopy invariant by \cite[Ch 2 and 3, e.g. Theorem 3.37]{morel} so it suffices to prove \eqref{lm:htpyinvtsheaf_iso_Xtopi0:gen} and \eqref{lm:htpyinvtsheaf_iso_Xtopi0:XA1conn} for $\calF$ an $\mathbb{A}^1$-homotopy invariant Nisnevich sheaf of sets.

The canonical commutative triangle
\[
\xymatrix{ X \ar[rr]^{\phi_X} \ar[rd] && \ar[ld] \pi_0^{\mathbb{A}^1}(X)\\
& \Spec k &}
\] gives rise to the commutative triangle
\[
\xymatrix{\calF(X) && \ar[ll]^{\phi_X^*} \Hom( \pi_0^{\mathbb{A}^1}(X),\calF) \\
& \calF(\Spec k) \ar[ur] \ar[ul] &}
\]
Since $X$ is $\mathbb{A}^1$-connected, $\pi_0^{\mathbb{A}^1}(X) \to \Spec k$ is an isomorphism (of sheaves of sets). Thus we have that \eqref{lm:htpyinvtsheaf_iso_Xtopi0:gen} implies \eqref{lm:htpyinvtsheaf_iso_Xtopi0:XA1conn}.

For \eqref{lm:htpyinvtsheaf_iso_Xtopi0:gen}, it follows from \cite[Corollary 3.22]{morelvoevodsky1998} that the map $\phi_X$ is an epimorphism of Nisnevich sheaves of sets. (To use \cite[Corollary 3.22]{morelvoevodsky1998}, let  $I = \mathbb{A}^1$, $\calX = X$, and $\calX'$ be the $\mathbb{A}^1$-localization of $X$.) Thus $\phi_X^*$ is injective.

For $\alpha: X \to \calF$, the natural transformations $\phi$ and $\pi_0^{\mathbb{A}^1}$ give the commutative diagram
\[
\xymatrix{ X \ar[d]_{\phi_X}\ar[r]^{\alpha} & \calF \ar[d]^{\phi_{\calF}}\\
\pi_0^{\mathbb{A}^1}(X) \ar[r]_{\pi_0^{\mathbb{A}^1}(\alpha)} & \pi_0^{\mathbb{A}^1}(\calF)}
\]
Since $\calF$ is an $\mathbb{A}^1$-homotopy invariant sheaf of sets, $\phi_{\calF}:\calF \to \pi_0^{\mathbb{A}^1}(\calF)$ is an isomorphism of Nisnevich sheaves of sets (Lemma \ref{lm:phi_iso_A1_invt_Nis_sheaf_sets}).Then $\alpha = \phi_X^*(\phi_{\calF}^{-1} \circ \pi_0^{\mathbb{A}^1}(\alpha)) $. Thus $\phi_X^*$ is surjective.
\end{proof}

\begin{lm}\label{lm:phi_iso_A1_invt_Nis_sheaf_sets}
 Let $\calF$ be an $\mathbb{A}^1$-homotopy invariant Nisnevich sheaf of sets. Then $\phi_{\calF}:\calF \to \pi_0^{\mathbb{A}^1}(\calF)$ is an isomorphism of Nisnevich sheaves of sets.
\end{lm}

Viewing a set as a discrete simplicial set, we have that $\calF(U) \cong \pi_0(\vert \calF(U) \vert)$ for any sheaf of sets $\calF$ and smooth $k$-scheme $U$. This allows us to replace $\calF$ with its connected components. The purpose of Lemma~\ref{lm:phi_iso_A1_invt_Nis_sheaf_sets} is to furthermore replace $\calF$ with its $\A^1$-connected components when $\calF$ $\mathbb{A}^1$-homotopy invariant. How one proves Lemma~\ref{lm:phi_iso_A1_invt_Nis_sheaf_sets} depends on how one sets up $\A^1$-homotopy theory, but we provide a proof with the necessary setup for completeness. 

\begin{proof}
The category of simplicial presheaves on smooth $k$-schemes can be given the structure of a simplicial model category with the global injective model structure, the injective local model structure with the Nisnevich topology, and the $\bbb{A}^1$-model structure. Let $\Lnis$ and $\Laone$ denote fibrant replacement functors for the injective local model structure with the Nisnevich topology and the $\bbb{A}^1$-model structure, respectively. (See for example, \cite[2.2]{AsokWickelgrenWilliams} \cite{DuggerHollanderIsaksen}.)

All sheaves, thought of as discrete simplicial sheaves, are globally fibrant \cite[pg 10 3)]{Jardine_Fields_simplicial_presheaves}. Thus $\calF$ is globally fibrant. Since a local weak equivalence of globally fibrant sheaves is a global weak equivalence, we have that the map \begin{equation}\label{eq:LnisFswwe}\calF(U) \to \Lnis \calF(U)\end{equation}  is a weak equivalence of simplicial sets for all smooth $k$-schemes $U$ \cite[pg 10 4)]{Jardine_Fields_simplicial_presheaves}. By \cite[Proposition 3.19]{morelvoevodsky1998}, it follows that $\Lnis \calF$ is $\Aone$-local (and thus fibrant in the $\Aone$-model structure). By \cite[Proposition 2.2.1]{AsokWickelgrenWilliams}, the map $\calF \to \Lnis \calF$ factors $\calF \to  \Laone \calF \to \Lnis \calF$. The map $\Laone \calF \to \Lnis \calF$ is an $\Aone$-weak equivalence by 2-out-of-3 and $\Laone \calF$ and $\Lnis \calF$ are both fibrant in the injective local model structure. Thus the map $\Laone \calF \to \Lnis \calF$ is a local whence global weak equivalence.The sheaf $\pi_0^{\mathbb{A}^1}(\calF)$ is the Nisnevich sheaf associated to the presheaf $U \mapsto \pi_0 \vert \Laone \calF (U) \vert$. Since $\Laone \calF (U) \simeq \Lnis \calF (U) \simeq \calF(U)$ and $\calF(U)$ is a set (i.e., discrete topological space), we have that the natural map $\calF (U) \to \pi_0 \vert \Laone \calF (U) \vert$ is a bijection. Since $\calF$ is a sheaf, it follows that the natural map $\phi_{\calF}: \calF \to \pi_0^{\mathbb{A}^1}(\calF)$ is an isomorphism.

\end{proof}

By Proposition~\ref{pr:htpyinvtsheaf_iso_Xtopi0}, we can refine our definition of the degree of a map with codomain $Y$ to lie in $\sGW(\pi_0^{\mathbb{A}^1}(Y))$ when $Y$ is a smooth $k$-scheme. In particular, when $Y$ is $\bbb{A}^1$-connected, the degree lies in $\GW(k)$.

There are alternate connectivity conditions in $\bbb{A}^1$-homotopy theory which also give rise to a $\GW(k)$-valued degree. The notion of $\bbb{A}^1$-chain connected varieties was defined in \cite[Definition 2.2.2]{AsokMorel} as follows. Let $L$ be a finitely generated separable extension of $k$, which is defined to mean that there exists a subextension $k \subseteq E \subseteq L$ such that $E$ is purely transcendental over $k$ and $L$ is separable and algebraic over $E$. Let $y$ and $y'$ be $L$-points of $Y$.

\begin{df} \cite[Definition 2.2.2]{AsokMorel}
An {\em elementary $\bbb{A}^1$-equivalence between $y$ and $y'$} is a map $f: \bbb{A}^1_L \to Y_L$ such that $f(t) = y$ and $f(t')=y'$ for some $t,t'$ in $\bbb{A}^1(L)$.
\end{df}

Elementary $\bbb{A}^1$-equivalence generates an equivalence relation $\sim$ on $Y(L)$. Denote the quotient $Y(L)/\sim$ by $\pi_0^{\bbb{A}^1, ch}(Y)(L)$.

\begin{df}\label{df:A1chain_connected}
$Y$ is {\em $\bbb{A}^1$-chain connected} if for every finitely generated separable field extension $L/k$, the set of equivalence classes $\pi_0^{\bbb{A}^1, ch}(Y)(L) = Y(L)/\sim$ consists of exactly one element.
\end{df}

For $Y$ a smooth, proper variety over a field $k$, it is a Theorem of A. Asok and F. Morel that $\bbb{A}^1$-chain connectedness and $\bbb{A}^1$-connectedness are equivalent \cite[Theorem 2]{AsokMorel}. More generally, there is an evident map $\psi_{Y,L}:\pi_0^{\bbb{A}^1, ch}(Y)(L)\to \pi_0^{\bbb{A}^1}(Y)(L)$ sending the class of $y\in Y(L)$ in $\pi_0^{\bbb{A}^1, ch}(Y)(L)$ to the class $[y]\in  \pi_0^{\bbb{A}^1}(Y)(L)$.  Asok and Morel show that if $Y$ is finite type and proper over $k$, then $\psi_{Y,L}$ is an isomorphism for all finitely generated separable extensions $L$ of $k$ \cite[Theorem 2.4.3]{AsokMorel}.

Since the pullback along field extensions of finite odd degree induces an injection on Grothendieck--Witt groups and we are interested in a $\GW(k)$-valued degree, we weaken the notion of $\bbb{A}^1$-chain connectedness as follows.

\begin{df}\label{df:A1oddextended_chain_connected}
$Y$ is {\em $\bbb{A}^1$-odd chain connected} if for every finitely generated separable field extension $L/k$, and every pair $y,y'$ in $Y(L)$, there exists a finite extension $L \subseteq L'$ of odd degree such that $y \sim y'$ in $Y(L')$.
\end{df}

We remark that an $n$-dimensional smooth $k$-scheme $Y$ has many closed points with separable residue field: for each $y$ in $Y$ there is an open neighborhood $U$ and an \'etale map $\phi: U \to \bbb{A}^n_k$ \cite[D\'efinition II.1.1]{sga1}. The points of $\bbb{A}^n_k$ with separable residue field are dense. The image $\phi(U)$ is open \cite[Tag 01U2]{stacks-project}, and therefore contains points with separable residue field. For all $u$ of $U$ such that $\phi(u)$ has separable residue field, $k \subseteq k(u)$ is separable.

\begin{lm}\label{lm:A1oddextendedchainconnected_inv_property}
Let $Y$ be a smooth proper $k$-scheme which is $\bbb{A}^1$-odd chain connected and assume that $Y(k) \neq \emptyset$. Then for any section $\beta$ of $\sGW(Y)$, there is a unique $b$ in $\GW(k)$, such that for every point $y: \Spec k(y) \to Y$ as in the commutative diagram $$\xymatrix{ \Spec k(y) \ar[dr]^p \ar[rr]^{y} && \ar[ld] Y \\ & \Spec k &}$$ and such that $k \subseteq k(y)$ is separable, we have $$y^* \beta = p^* b .$$

\end{lm}

\begin{proof}
Choose $y_0$ in $Y(k)$. We must let $b= y_0^* \beta$, showing uniqueness. Let $y: \Spec k(y) \to Y$ be a closed point. By Proposition \ref{pr:A1connected_implies_GW=GW(k)} and Example \ref{ex:A1-connected} \eqref{ex:covered_by_affines_implies_A1_connected}, if $i:k(y) \subseteq L'$ is an extension of fields, $y'$ in $Y(L')$ and $y' \sim (y \circ i)$, then $$i^* y^*\beta = (y')^* \beta.$$ We may view $y_0 \circ p$ as an element of $Y(k(y))$. Since $Y$ is $\bbb{A}^1$-odd chain connected, there is an extension $i:k(y) \subseteq L'$ of odd degree such that $y \circ i \sim y_0 \circ p \circ i$ in $Y(L')$. Thus $$i^* y^* \beta = i^* p^* b$$ in $\GW(L')$. An odd degree field extension induces an injection on $\GW$ \cite[Chapter~VII, Corollary~2.6]{lam05}\footnote{The reference shows the claim on Witt groups, from which the injection on $\GW$ follows from the isomorphism $\GW \cong \W \times_{\Z/2} \Z$.}, therefore $y^* \beta = p^* b$ as claimed.
\end{proof}

\begin{df}\label{df:GWk-degree-cases}
Let $f: X \to Y$ and $\rho$ be as in Remark \ref{fassumption}, Proposition~\ref{fassumption_nonproper}, or respectively Setting~\ref{fassumption_charpnonproper}. If $Y$ is additionally $\bbb{A}^1$-odd chain connected and $Y(k) \neq \emptyset$, then define the degree of $f$, denoted $\deg (f,\rho)$ or $\deg f$, to be the $b$ associated by Lemma \ref{lm:A1oddextendedchainconnected_inv_property} to the section of $\sGW(Y)$ given by the degree of $f$ defined in Definition \ref{df:deg(f)_in_sGW(Y)} or respectively Definition~\ref{df:A1_deg_charp}.
\end{df}

\subsection{Connectivity and restriction of scalars}

We will have use of the degrees of maps whose targets are restrictions of scalars. Since the $\bbb{A}^1$-degree lands in $\sGW(\pi_0^{\bbb{A}^1}(-))$ applied to the target, we give a result on the connectivity of a restriction of scalars in this section.

We extend the notation of $\bbb{A}^1$-chain connected components $\pi_0^{\bbb{A}^1, ch}(Y)(L)$ of a finite type $k$-scheme $Y$ over a field $L$ (see Definition~\ref{df:A1chain_connected}) to products of fields by defining
 \[
\pi_0^{\bbb{A}^1, ch}(Y)(\prod_{i=1}^rL_i):=\prod_{i=1}^r \pi_0^{\bbb{A}^1, ch}(Y)(L_i).
 \] Define $\psi_{Y,\prod_i L_i}: \pi_0^{\bbb{A}^1, ch}(Y)(\prod_i L_i) \to \pi_0^{\bbb{A}^1}(Y)(\prod_i L_i)$ by sending the class of an $r$-tuple of points $\prod_{i=1}^r y_i\in Y(\prod_{i=1}^r L)$ to the class $[\prod_{i=1}^r y_i]$ in $ \pi_0^{\bbb{A}^1}(Y)(\prod_i L_i)$ as above. Clearly $(Y,L)\mapsto \pi_0^{\bbb{A}^1, ch}(Y)(L)$ is (covariantly) functorial in $Y$ and $L$ and  $\psi_{Y,L}$ is natural in $(Y,L)$.

For a field $F$, let $\Sch_F$ denote the category of finite type, separated $F$-schemes.

Let $k\subset L$ be a finite  separable field extension. We have the Weil restriction functor $\Res_{L/k}:\Sch_L\to \Sch_k$, which is right adjoint to the extension of scalars functor $X\mapsto X_L:=X\times_{\Spec k}\Spec L$. Passing to the limit over suitable open subschemes induces the natural isomorphism $X(L\otimes_k F)\cong  \Res_{L/k}(X)(F)$ for all finitely generated extensions $F$ of $k$.  The points $0$ and $1$ define sections $i_0, i_1:\Spec (-)\to \bbb{A}^1_{(-)}$ which are stable under pullback. The isomorphisms $X(L\otimes_k F)\cong  \Res_{L/k}(X)(F)$ and $X(\bbb{A}^1_{L\otimes_k F})\cong  \Res_{L/k}(X)(\bbb{A}^1_F)$ induce an isomorphism, natural in $F$ and $X$
 \begin{equation}\label{Pi0ChRestriction}
 \rho_{L,X,F}:\pi_0^{\bbb{A}^1, ch}(X)(L\otimes_k F)\stackrel{\cong}{\to} \pi_0^{\bbb{A}^1, ch}(\Res_{L/k}(X))(F).
 \end{equation}

\begin{lm}\label{lm:A1Res(field)}
Let $k\subset L$ be a finite separable field extension. 
For $X$ a finite type, proper $L$-scheme and $F$ a finitely generated separable extension of $k$, the isomorphism \eqref{Pi0ChRestriction} induces an isomorphism  
\[
\pi^{\bbb{A}^1}_0(\psi_X)(L\otimes_kF):\pi_0^{\bbb{A}^1}(X)(L\otimes_kF)\stackrel{\cong}{\to} \pi_0^{\bbb{A}^1}(\Res_{L/k}(X))(F),
\]
natural in $F$ and $X$.

\end{lm}

\begin{proof}
 If $F$ is a finitely generated separable extension of $k$, then $L\otimes_kF$ is a finite product of  finitely generated separable extensions of $k$; if $X$ is a proper $L$-scheme, then $\Res_{L/k}(X)$ is a proper $k$-scheme. We apply \cite[Theorem 2.4.3]{AsokMorel}, which together with the isomorphism \eqref{Pi0ChRestriction} gives us the sequence of isomorphisms
\begin{align*}
\pi_0^{\bbb{A}^1}(X)(L\otimes_kF)\xrightarrow{\psi^{-1}_{X,L\otimes_kF}}
\pi_0^{\bbb{A}^1, ch}(X)(L\otimes_kF)
\\\xrightarrow{\rho_{k,L,X,F}}
\pi_0^{\bbb{A}^1, ch}(\Res_{L/k}(X))(F)\xrightarrow{\psi_{\Res_{L/k}(X),F}}
\pi_0^{\bbb{A}^1}(\Res_{L/k}(X))(F),
 \end{align*}
 natural in $F$ and $X$.
\end{proof}

\begin{pr}\label{A1-connectivity_res_scalars}
Let $k\subset L$ be a finite separable field extension. Let $X$ be a smooth proper $L$-scheme.  If $X$ is $\bbb{A}^1$-connected, then so is $\Res_{L/k}(X)$.
\end{pr}

\begin{proof}
By Lemma~\ref{lm:A1Res(field)} and \eqref{Pi0ChRestriction}, $\Res_{L/k}(X)$ is $\bbb{A}^1$-chain connected. (See Definition~\ref{df:A1chain_connected}). Since $X$ is a smooth, proper $L$-scheme, $\Res_{L/k}(X)$ is a smooth proper $k$-scheme. By \cite[Theorem 2]{AsokMorel}, it follows that $\Res_{L/k}(X)$ is $\bbb{A}^1$-connected. 
\end{proof}

We record the following well-known fact.

\begin{pr}
Let $X$ and $Y$ be smooth $k$-schemes which are $\bbb{A}^1$-connected. Then $X \times_k Y$ is $\bbb{A}^1$-connected.
\end{pr}

\begin{proof}
$\pi_0^{\bbb{A}^1}(X \times_k Y) $ is the sheaf associated to the presheaf sending a smooth $k$-scheme $U$ to $\pi_0 (\Laone (X \times_k Y) (U))$. The functor $\Laone$ commutes with finite products (see e.g. \cite[2.2.1(iii)]{AsokWickelgrenWilliams}). It follows that 
\begin{align*}
\pi_0 (\Laone (X \times_k Y) (U)) \cong \pi_0 ((\Laone (X) \times \Laone (Y)) (U)) \cong\\ 
 \pi_0 (\Laone (X)(U) \times \Laone (Y)(U)) \cong \pi_0 \Laone X (U) \times  \pi_0 \Laone X (U). 
\end{align*} Since sheafification preserves finite limits, it follows that $\pi_0^{\bbb{A}^1}(X \times_k Y) \cong \pi_0^{\bbb{A}^1}(X) \times \pi_0^{\bbb{A}^1}(Y)$.
\end{proof}

\section{Local degree}\label{sec:local degree}

Now suppose that $f: X \to Y$ is an oriented map of smooth $n$-dimensional schemes over $k$ or, more generally, a map from an Artin stack to a smooth $k$-scheme which is oriented away from codimension $2$ (as in Definition~\ref{df:or_away_codim_c}) or pseudo-oriented (as in Definition~\ref{df:pseudo-oriented} and Setting \ref{fassumption_charpnonproper}). Let $x$ be a point of $X$ with image $y=f(x)$. Suppose there are Zariski open neighborhoods $W$ and $U$ of $x$ and $y$, respectively, such that $f(W) \subset U$ and the restriction $f\vert_W : W \to U$ is a finite map of schemes oriented by $\rho: L^{\otimes 2} \stackrel{\cong}{\to} \omega_{f\vert_{W}}$. By Proposition \ref{nonfinite_locus_is_codim2}, it is possible to find many $x$ and $y$ which admit such a $U$,$W$ and $\rho$ (because it is possible to find many such $x$ and $y$ for a generically finite, oriented, proper map $f: X \to Y$ to a smooth $k$-scheme $Y$, and the more general setup supposes a restriction satisfying these hypotheses). 

In this section, we define the local degree $\deg_x (f,\rho)$ in $\GW(k(y))$ under these circumstances . We also use the notation $\deg_x f$ for $\deg_x (f,\rho)$ when there is no danger of confusion. The degree defined in Section \ref{Section:degree} will be shown to be a sum of local degrees in Proposition \ref{pr:deg=sumlocal}, and we will give a formula to compute $\deg_x f$ with the Jacobian determinant in Proposition~\ref{degxf=Tr<J>}.

\subsection{Definition and properties}

To simplify notation, we let $f$ also denote the restriction $f=f\vert_W : W \to U$, where $f$,$W$, and $U$ are as in the beginning of the section. Let $f\times y: U \times_y \Spec k(y) \to \Spec k(y)$ denote the pullback of $f$ along $y: \Spec k(y) \to U$ as in the pullback diagram $$\xymatrix{  U \times_y \Spec k(y) \ar[r] \ar[d]^{f\times y} & W\ar[d]^{f}  \\ \Spec k(y) \ar[r]^y & U.}$$ The fiber $f^{-1}(y) \cong U \times_y \Spec k(y) $ is the coproduct $$U \times_y \Spec k(y) \cong \coprod_{z \in W: f(z) = y} \calO_{f^{-1}(y),z}.$$ For every $z$ in $W$ mapping to $y$, let $f_z$ denote the composition of the inclusion $\Spec \calO_{f^{-1}(y),z} \to f^{-1}(y)$ with $f \times y$. $$\xymatrix{\Spec \calO_{f^{-1}(y),z}\ar[d]_{f_z} \ar[r] & W \ar[d]^{f}\\
 \Spec k(y) \ar[r]^y & U. }$$

In our notation, $L$ is a locally free sheaf of rank $1$ on $W$ and $\rho: L^{\otimes 2} \to \omega_f $ is the isomorphism giving the orientation. Denote the fiber of the sheaf $f_*L$ at $y$ by $f_* L(y)$. Since $f$ is finite, the canonical map $f_* L(y) \to (f \times y)_* L$ is an isomorphism. By a slight abuse of notation, we also let $L$ denote its pullback to  $\Spec \calO_{f^{-1}(y),z} $. We have a canonical isomorphism \begin{equation}\label{pushforardL_as_perp_direct_sum}f_* L(y) \cong \oplus_{z \in W: f(z) = y} (f_z)_* L.\end{equation} Recall that the pullback of the pairing $\deg (f,\rho)$ of Definition \ref{GSDdegreedef} to $\Spec k(y)$ is denoted $\deg(f,\rho)(y)$, as in Definition \ref{df:degf(y)}. Since this pairing factors through the multiplication $f_* L \otimes f_* L \to f_* (L^{\otimes 2})$, it follows that \eqref{pushforardL_as_perp_direct_sum} is an orthogonal decomposition.

\begin{df}
Let $\deg_x (f,\rho)$ in $\GW(k(y))$ be the restriction of $\deg (f,\rho)(y)$ to $(f_z)_* L$. When the relative orientation is clear from context, we also write $\deg_x f$.
\end{df}

Let $k$ be a field. Suppose $f: X \to Y$ is a map from an Artin stack to a smooth $k$-scheme which is which is oriented away from codimension $2$ or pseudo-oriented. By hypothesis there is a dense open subset $U \subseteq Y$ such that the restriction $f \vert_{f^{-1}(U)}: f^{-1}(U) \to U$ is a generically finite, proper local completed intersection morphism. Moreover, there is a line bundle $L$ on $f^{-1}(U)$ together with an isomorphism $\rho: L^{\otimes 2} \to \omega_f$ giving the orientation. By the orthogonal decomposition \eqref{pushforardL_as_perp_direct_sum}, we have shown that the global degree is the sum of local degrees. 

\begin{pr}\label{pr:deg=sumlocal}
Let $k$ be a field. Suppose $f: X \to Y$ is a map from an Artin stack to a smooth $k$-scheme which is oriented away from codimension $2$ or pseudo-oriented. For all $y$ in $U$, there is an equality $\deg f (y) = \sum_{x \in f^{-1}(y)} \deg_x f$ in $\GW(k(y))$.
\end{pr}

Again take $x$ in $f^{-1}(U)$ mapping to $y$. Let $k(y) \subseteq E$ be a field extension, and let $f_E\vert_{U_E}: W_E \to U_E$ denote $f \vert_W \otimes_k E$. We have the commutative diagram \begin{equation*}
\xymatrix{W  \ar[d]_{f \vert_W}& \ar[l]_{\pi'} W_E \ar[d]^{f_E\vert_{W_E}} \\ U & \ar[l]^{\pi} U_E}
\end{equation*} There is a canonical point $\tilde{y}$ of $U_E$ mapping to $y$ determined by the map $\Spec E \to \Spec k(y) \to U$ and the identity map $\Spec E \to \Spec E$. The pullback of $\rho$ determines an orientation $(\pi')^* \rho$ of $f_E\vert_{W_E}$ because there is a canonical isomorphism $(\pi')^* \omega_{f \vert_W} \cong \omega_{f_E\vert_{W_E}}$\cite[Tag 01V0]{stacks-project}. 

\begin{pr}\label{local_deg_stable_base_change}
For all $\tilde{x}$ in $W_E$ mapping to $\tilde{y}$, we have the equality $\deg_{\tilde{x}} (f_E, (\pi')^* \rho) = \deg_x (f, \rho) \otimes E$ in $\GW(E)$.
\end{pr}

\begin{proof}
Let $L$ denote the line bundle on $W$ of the orientation $\rho$. Since $f_E\vert_{W_E}$ is the pullback of a finite, flat map, $f_E\vert_{W_E}$ is finite and flat. Because $\pi$ is flat and $f\vert_W$ is lci, the map $f_E\vert_{W_E}$ is lci by \cite[Tag 069I]{stacks-project}. We are in the situation of Proposition \ref{pr:degree_finite_commutes_basechange} and its proof. The natural map from cohomology and base change is an isomorphism $ (f_E\vert_{W_E})_* (\pi')^* L \cong \pi^*( f \vert_W)_* L$  ( \cite[Tag 02KG]{stacks-project}). Under this isomorphism the subspaces $((f_E)_{\tilde{x}})_*  (\pi')^* L$ and $\pi^*(f_x)_* L$ are identified. Combining cohomology and base change with the functoriality of the trace map \cite[Tag 0B6J]{stacks-project} (as in Proposition~\ref{pr:degree_finite_commutes_basechange}), we obtain a commutative diagram
\begin{equation*}
\small
\xymatrix@C=2.2em{\pi^*(f_x)_* L \otimes \pi^*(f_x)_* L \ar[r] \ar[d]^{\cong} &\pi^*( f \vert_W)_* L \otimes \pi^*( f \vert_W)_* L \ar[d]_{\cong} \ar[r]& \pi^*( f \vert_W)_* L^{\otimes 2} \stackrel{\rho}{\cong}  \pi^*( f \vert_W)_* \omega_{f} \ar[d]^{\cong} \ar[r] & \pi^* \cO_U \ar[d]^{\cong} \\ ((f_E)_{\tilde{x}})_*  (\pi')^* L \otimes ((f_E)_{\tilde{x}})_*  (\pi')^* L \ar[r] & (f_E)_* (\pi')^* L \otimes (f_E)_* (\pi')^* L \ar[r] & (f_E)_* (\pi')^* L^{\otimes 2} \stackrel{(\pi')^* \rho}{\cong} (f_E)_* \omega_{f_E} \ar[r] & \cO_{U_E}}
\end{equation*} where the vertical maps are isomorphisms. This commutative diagram gives an isomorphism of the forms defining $\deg_x (f, \rho) \otimes E$ and $\deg_{\tilde{x}}(f_E, (\pi')^* \rho)$.
\end{proof}

By Proposition \ref{local_deg_stable_base_change}, the computation of $\deg_x(f, \rho)$ reduces to the case where $y=f(x)$ is a rational point, because it may be computed after base change to $k(y)$, and in particular we may assume that $y$ is a closed point. We now give such a method of computation for $\deg_x f$ for $x$ a closed point with $k \subseteq k(x)$ separable.

As above, let $W$ and $U$ be smooth $k$-schemes of dimension $n$, let $x$ be a point of $X$, and let $f\vert_W : W \to U$ be finite and oriented by $\rho$. Suppose that $y=f(x)$ is a closed point with $k \subseteq k(x)$ separable. Then $f\vert_W^{-1}(y) \hookrightarrow W$ is a closed immersion. Consider the pullback diagram \begin{equation}\label{f'vertWsquare}\xymatrix{ f\vert_W^{-1}(y) \ar[d]_{f'\vert_W} \ar[r]^{y'} & W \ar[d]^{f\vert_W}\\
 \Spec k(y) \ar[r]^y & U } \end{equation}

\begin{lm}
The map $f'\vert_W$ is a finite, flat, local complete intersection morphism.
\end{lm}

\begin{proof}
$f'\vert_W$ is finite and flat because these properties are stable under pullback. Since $\Spec k(y)$ and $U$ are regular schemes, the finite type map $y: \Spec k(y) \to U$ is a local complete intersection morphism \cite[0E9K]{stacks-project}. Since $f\vert_W$ is finite and $W$ and $U$ are smooth and dimension $n$, it follows that $f\vert_W$ is flat  \cite[Theorem 23.1 p.179]{Matsumura_CRT}. Thus the pullback $y'$ is a local complete intersection morphism \cite[Tag 069I]{stacks-project}. Since $W$ and $U$ are smooth $k$-schemes, $f\vert_W$ is a local complete intersection morphism. Thus the composition $ f\vert_W \circ y'  = y \circ f'\vert_W$ is lci as well. Since $U \to \Spec k$ is smooth, it follows that the structure map for $f\vert_W^{-1}(y)$ over $\Spec k$ is lci. Since $\Spec k(y) \to \Spec k$ is smooth, it follows that $f'\vert_W$ is lci as claimed \cite[Tag 069M]{stacks-project}
\end{proof}

Since $f'\vert_W$ is flat, the square \eqref{f'vertWsquare} and \cite[Tag 08QQ]{stacks-project} define a canonical isomorphism \begin{equation}\label{omegaf'ispullback} \omega_{f'\vert_W} \cong (y')^*\omega_{f\vert_W}.\end{equation} Since $f\vert_W$ is finite, $x$ determines a closed and open subscheme $\Spec \calO_{f^{-1}(y),x} \hookrightarrow f\vert_W^{-1}(y)$ of $f\vert_W^{-1}(y)$. By \cite[Tags 063L, 063M]{stacks-project}, the inclusion of a closed and open component in a locally Noetherian scheme is a local complete intersection morphism, because Koszul-regular immersions are lci. Define $f_x: \Spec \calO_{f^{-1}(y),x} \to \Spec k(y)$ to be the composition of $\Spec \calO_{f^{-1}(y),x} \hookrightarrow f\vert_W^{-1}(y)$ and $f'\vert_W$, so in particular $f_x$ is finite, flat, lci and fits into the commutative diagram \begin{equation}\label{eq:f-1yx}\xymatrix{\Spec \calO_{f^{-1}(y),x}\ar[d]_{f_x} \ar[rr]^{i_{f,x}} && W \ar[d]^{f\vert_W}\\
 \Spec k(y) \ar[rr]^y && U ~~. }\end{equation} The isomorphism \eqref{omegaf'ispullback} defines an isomorphism $ \omega_{f_x} \cong i_{f,x}^* \omega_{f\vert_W},$ whence $i_{f,x}^*\rho $ defines a relative orientation of $f_x$.

\begin{pr}\label{degxfrho=degfxix*rho}
Let $W$ and $U$ be smooth $k$-schemes of dimension $n$, let $x$ be a point of $X$, and let $f : W \to U$ be finite and oriented by $\rho$. Suppose that $y=f(x)$ is a closed point with $k(y) \subseteq k(x)$ separable. Then $$\deg_x(f, \rho) = \deg (f_x, i_{f,x}^*\rho).$$
\end{pr}

\begin{proof}
$f$ is flat by \cite[Theorem 23.1 p.179]{Matsumura_CRT}. Both $\deg_x(f, \rho)$ and $\deg (f_x, i_{f,x}^*\rho)$ are bilinear forms on the $k(y)$-vector space $(f_x)_* i_{f,x}^* L$. Note that $\deg (f_x, i_{f,x}^*\rho)$ is the composition $$((f_x)_* i_{f,x}^* L)^{\otimes 2} \to (f_x)_* i_{f,x}^* (L^{\otimes 2}) \xleftarrow[\cong]{i_{f,x}^*\rho} (f_x)_* \omega_{f_x}   \to k(y).$$ By definition, $\deg_x(f, \rho)$ is the composition $$((f_x)_* i_{f,x}^* L)^{\otimes 2} \to ((f \vert_W)_* L(y))^{\otimes 2}  \to (f \vert_W)_* (L^{\otimes 2})(y) \xleftarrow[\cong]{\rho} (f \vert_W)_* \omega_{f \vert_W} (y)\to k(y).$$ The diagram \eqref{eq:f-1yx} induces a map 
\[
\Spec \calO_{f^{-1}(y), x} \to f^{-1}(y):=\Spec k(y) \times_U  W
\] from $\Spec \calO_{f^{-1}(y),x}$ to the fiber product of $y$ and $f\vert_W$. Composing with cohomology and base change of the pullback, we have the commutative diagram \begin{equation}\label{eq:cohbcf}\xymatrix{ (y^*(f \vert_W)_* L)^{\otimes 2}  \ar[r]& y^*(f \vert_W)_* (L^{\otimes 2}) &&  \ar[ll]_{\cong}^{y^*(f \vert_W)_*\rho}y^*(f \vert_W)_* \omega_{f \vert_W} \\
((f_x)_* i_{f,x}^* L)^{\otimes 2} \ar[r] \ar[u]&(f_x)_* i_{f,x}^* (L^{\otimes 2}) \ar[u] && \ar[ll]_{\cong}^{(f_x)_*i_{f,x}^*\rho} (f_x)_* \omega_{f_x}  . \ar[u] }\end{equation}

Applying  \cite[Lemma 3.4.3 TRA1]{Conrad_Grothendieck_duality_and_base_change} to the composition $\Spec \calO_{f^{-1}(y), x} \to \Spec k(y) \times_U  W \to \Spec k(y)$, we have the commutative diagram
\[
\xymatrix{(f_y)_* \omega_{f_y}  \ar[r]  &k(y) \\ (f_x)_* \omega_{f_x}  \ar[r] \ar[u] &k(y) \ar[u]^{\text{id}}}
\] where the horizontal arrows are the trace maps.

The trace map for finite flat maps commutes with base change: To see this,  let $A \subset B$ be ring map corresponding to a finite flat map $f: \Spec B \to \Spec A$ and let $y$ be a point of $\Spec A$. View $\Hom_A(B,A)$ as a coherent sheaf on $\Spec A$. Then there is a canonical isomorphism $f_* \omega_f \cong \Hom_A(B,A)$ and the trace map $\Tr_{B/A}:\Hom_A(B,A) \to A$ is evaluation at $1$  \cite[(3.4.7) p147]{Conrad_Grothendieck_duality_and_base_change}. The pullback of $\Tr_{B/A}$ by $\Spec k(y) \to \Spec A$ is the evaluation at $1$ map $\Hom_{k(y)}(B_y, k(y))$ as claimed. Thus we have another commutative diagram
\[
\xymatrix{ y^*(f \vert_W)_* \omega_{f \vert_W} (y)\ar[r] & k(y)\\ (f_y)_* \omega_{f_y}  \ar[r] \ar[u] &k(y) \ar[u]^{\text{id}} }
\]
where the horizontal arrows are the trace maps.

Putting the previous two commutative diagrams together, we have a commutative diagram
$$\xymatrix{ y^*(f \vert_W)_* \omega_{f \vert_W} (y)\ar^{\Tr}[rr] && k(y) \\ (f_x)_* \omega_{f_x}  \ar[rr]^{\Tr} \ar[u] &&k(y) \ar[u]^{\text{id}}} $$
which combines with the diagram \eqref{eq:cohbcf} and the fact mentioned above that both $\deg_x(f, \rho)$ and $\deg (f_x, i_{f,x}^*\rho)$ are bilinear forms on the $k(y)$-vector space $(f_x)_* i_{f,x}^* L$ to prove the proposition.

 \end{proof}

\subsection{Computation with the Jacobian}

Let $f: X \to Y$ be a oriented map between smooth, connected schemes of dimension $n$. So we have a line bundle $L$ on $X$, and an isomorphism $$\rho: \Hom (\det TX, f^* \det TY) \to  L^{\otimes 2} .$$ Then $f$ induces a map $Tf$ on tangent bundles and a global section $$\det Tf \in \Hom (\det TX, f^* \det TY).$$ Taking the image under the $\rho$ gives a global section $\rho(\det Tf)$ of $L^{\otimes 2}$.

\begin{constrn}\label{sectionofsquareasfn}
{\em A section of the {\em square} of a line bundle determines a canonical element of the Zariski sheaf $\calO_{X}/(\calO_{X}^*)^2$. Namely, suppose $\sigma$ is a section of  $L^{\otimes 2}$. Around any point $x$, choose a local trivialization of $L$, identifying $\sigma$ with an element of $\calO_{X,x}$. Any two choices of local trivialization will change this element by the square of a unit in $\calO_{X,x}$.}
\end{constrn}

\begin{df}\label{df:Jacobian}
The {\em Jacobian} $Jf$ of $f$ is the section of $\calO/(\calO^*)^2$ corresponding to $\rho(\det Tf)$ by Construction \ref{sectionofsquareasfn}: $$Jf = \rho(\det Tf) \in \calO/(\calO^*)^2.$$ Let $Jf_x$ (respectively $Jf(x)$) be the image of $Jf$ in $\calO_{X,x}/ (\calO_{X,x}^*)^2$ (respectively $k(x)/(k(x)^*)^2$).
\end{df}

\begin{pr}\label{degxf=Tr<J>}
If $f$ is \'etale at $x$, then $\deg_x f = \Tr_{k(x)/k(y)} \langle Jf(x) \rangle$.
\end{pr}

\begin{proof}

By Proposition~\ref{degxfrho=degfxix*rho}, we have $\deg_x(f, \rho) = \deg (f_x, i_{f,x}^*\rho).$ Since $f$ is \'etale at $x$, the canonical map $\calO_{f^{-1}(y),x} \to k(x)$ is an isomorphism and $k(y) \subseteq k(x)$ is a finite, separable extension. Choosing a trivialization of $L$ near $x$ defines an isomorphism $i_{f,x}^*L \cong \calO_{f^{-1}(y),x}$, thus defining an isomorphism  $i_{f,x}^*L \cong k(x)$, which we now fix for the rest of the proof. By Definition~\ref{GSDdegreedef}, it follows that $ \deg (f_x, i_{f,x}^*\rho)$ in $\GW(k(y))$ is represented by
\begin{equation}\label{eq:JacPropformeq}
k(x) \otimes k(x) \to k(x)\cong i_{f,x}^*L^{\otimes 2} \xrightarrow{(f_x)_*i_{f,x}^*\rho^{-1}} (f_x)_* \omega_{f_x} (y) \xrightarrow{\Tr_{f_x}} k(y).
\end{equation} Here, the first map comes from the lax monoidal structure $f_* L \otimes f_* L \to f_* (L^{\otimes 2})$, and is thus identified with the multiplication map. 

From  \eqref{eq:Serre_duality_finite_flat_lci} applied to $f_x$, we have an isomorphism $\omega_{f_x} \cong f_x^{\natural}\Hom_{k(y)}(k(x),k(y))$. Under this isomorphism, the section $\det Tf$ of  $\omega_{f_x}$ is sent to $\Tr: k(x) \to k(y)$ by \cite[(4.2)Satz]{scheja}, where here $\Tr: k(x) \to k(y)$ denotes the trace associated to the finite \'etale algebra $k(y)\subseteq k(x) $. The map $\Tr_{f_x}$ in \eqref{eq:JacPropformeq} under this isomorphism $\omega_{f_x} \cong f_x^{\natural}\Hom_{k(y)}(k(x),k(y))$ is evaluation at $1 \in k(x)$ as in Section~\ref{subsection:degree_finite_flat_oriented_lci}. Thus $\Tr_{f_x}(c\det Tf) = \Tr(c)$ for any $c \in k(x)$.

Unwinding Definition~\ref{df:Jacobian}, it follows that the image of $(a,b)$ in $k(x) \otimes k(x)$ under the pairing \eqref{eq:JacPropformeq} is
\[
\Tr\big(\frac{ab}{ \rho ( \det Tf (x))}\big).
\] Thus $ \deg (f_x, i_{f,x}^*\rho) = \Tr_{k(x)/k(y)}\langle \frac{1}{Jf(x)} \rangle = \Tr_{k(x)/k(y)}\langle Jf(x) \rangle$ as claimed.

\end{proof}

Taking $x$ to be the generic point of $X$, this shows the Jacobian can be used to compute $\deg f$.

\begin{co}\label{deg=trjacgeneric}
Let $f: X \to Y$ be a separable map between smooth, proper connected $k$-schemes of dimension $n$. Suppose there exists a closed subset $Z$ of $Y$ of codimension at least $2$ such that the restriction of $f$ to $f^{-1}(Y-Z)$ is oriented. Let $\eta$ denote the generic point of $X$. Then $ \Tr_{k(X)/k(Y)} \langle Jf(\eta) \rangle$ is in the image of $\sGW(Y) \subseteq \GW(k(Y))$ and the degree of $f$ is given by $$\deg f = \Tr_{k(X)/k(Y)} \langle Jf(\eta) \rangle.$$ Moreover, if $Y$ is either $\bbb{A}^1$-connected or $Y$ has a $k$-point and is $\bbb{A}^1$-odd chain connected, then $\Tr_{k(X)/k(Y)} \langle Jf(\eta) \rangle$ is in the image of the pullback $\GW(k) \to \GW(k(Y)).$

\end{co}

\begin{proof}
$f$ is \'etale at $\eta$ because $f$ is a separable morphism, so we may apply Proposition \ref{degxf=Tr<J>} to $x=\eta$. The first statement then follows from Proposition \ref{local_deg_stable_base_change} and the fact that $\sGW$ is an unramified sheaf. The second assertion follows from Corollary \ref{pr:A1connected_implies_GW=GW(k)} in the first case and Lemma \ref{lm:A1oddextendedchainconnected_inv_property} in the second case.
\end{proof}

Combining Proposition \ref{degxf=Tr<J>} with Proposition \ref{pr:deg=sumlocal} implies:

\begin{co}
Let $y$ be a regular value of $f$. Then $\deg f = \sum_{x \in f^{-1} y} \Tr_{k(x)/k(y)} \langle J(f) \rangle$ in $\GW(k(y))$.
\end{co}

\section{Stable maps and orientations}\label{section:CountsRationalCurves} In the remaining sections, we give quadratically enriched counts of rational curves on del Pezzo surfaces passing through an appropriate set of points by taking the $\bbb{A}^1$-degree of maps from moduli spaces of such curves.

\subsection{Stable maps to a surface}\label{subsection:modulispace_rational_curves_del_Pezzo}
\subsubsection{Definitions}
We set up notation for the needed moduli spaces and maps, consistent with that of \cite{KLSW-relor}, to which we refer the reader for further information and references. The following is based on \cite{Abramovich--Oort-mixed_char}.

\begin{df}
Let $T$ be a scheme and let $X$ be a smooth projective $T$-scheme of dimension~$2$. Let $D \in Pic(X).$ For any $T$ scheme $U,$ an $n$-pointed stable map to $X$ over $U$ of genus $g$ and degree $D$ is a diagram
\[
\xymatrix{
C  \ar[d]_{\varpi}\ar[r]^u & X \\
U \ar@/_/[u]_{p_i}
}
\]
where $i = 1,\ldots,n,$ satisfying the following conditions.
\begin{enumerate}
\item
The morphism $\varpi : C \to U$ is a projective flat family of curves.
\item
The geometric fibers of $\varpi$ are reduced with at most nodes as singularities.
\item
The sheaf $\varpi_* \omega_{C/U}$ is locally free of rank $g.$
\item
The morphisms $p_i$ are sections of $\varpi$ which are disjoint and land in the smooth locus of $\varpi.$
\item
For all $L \in Pic(X),$ the degree of $u^*L$ on geometric fibers of $\varpi: C \to U$ is $D \cdot L.$
\item
Let $P_i$ denote the image of $p_i.$ The line bundle $\omega_{C/U}(\sum P_i) \otimes u^*\calO(3)$ is relatively ample.
\end{enumerate}
We say that the stable map $(u: C \to X,p_1,\ldots,p_n)$ is \emph{irreducible} if all geometric fibers of $\varpi : C \to U$ are irreducible.
\end{df}
\begin{df}
Let $U,U'$ be $T$-schemes. A \emph{morphism} of stable maps
\[
\alpha: (C \to U, u: C \to X, p_1,\ldots,p_n)  \to (C' \to U', u': C' \to X, p_1',\ldots,p_n')
\]
is a commutative diagram
\[
\xymatrix{
C \ar[r]^{\alpha_C}\ar[d] & C' \ar[d] \\
U \ar[r]^{\alpha_U} & U'
}
\]
inducing an isomorphism $C \to C' \times_U U',$ such that $\alpha_C \circ p_i = p_i' \circ \alpha_U$ and $u = u' \circ \alpha_C.$
\end{df}
When it does not cause confusion, we may drop the structure morphism $C \to U$ from our notation for stable maps.
We denote by $\M_{g,n}(X,D)$ the category of $n$-pointed stable maps to $X$ of genus $g$ and degree~$D,$ and by $M_{g,n}(X,D) \subset \M_{g,n}(X,D)$ the substack of irreducible stable maps. It follows from 
\cite[Theorem 2.8, p. 90]{Abramovich--Oort-mixed_char} that $\M_{g,n}(X,D)$ is a proper Artin algebraic stack with finite stabilizers admitting a projective coarse moduli scheme. 
See also~\cite[Section 4]{deJongHeStarr2011} and \cite{KLSW-relor} for more information.
Our counts of rational curves will be the $\bbb{A}^1$-degrees of appropriate modifications of the following evaluation map.
\begin{df}
The {\em evaluation map} $\ev:\M_{g,n}(X,D)\to X^n$ is given by
\[
(u: C\to X, p_1,\ldots, p_n) \mapsto (u(p_1),\ldots, u(p_n))
\]
\end{df}

\subsubsection{The double point locus}\label{sssec:dpl}
Following~\cite[Section 5]{KLSW-relor}, we define the double point locus using ideas from \cite[Example 9.2.8]{fulton98}. 
\begin{df}\label{definition:double_point_locus}
	Given a composition of closed immersions of schemes $Z \subset W \subset X$, we define the {\em subscheme of $W$ residual to $Z$} to be the closed subscheme defined by the ideal sheaf $(I_{W} : I_{Z})$.  Recall that this is the ideal sheaf of all local sections $s$ of $\calO_{X}$ such that $s t$ lies in $I_{W}$ for all local sections $t$ of $I_{Z}$.
\end{df}
Let $V \subset \bar{M}_{g,n}(X,D)$ be a substack that is a scheme. Let
\[
\xymatrix{
C_V  \ar[d]_{\varpi}\ar[r]^{u_V} & X \\
V
}
\]
be the universal stable map over $V.$ So, $\tilde u_V = u_V \times \varpi : C_V \to X \times V$ is a map over $V.$ Consider the product
\[
\tilde u_V \times \tilde u_V : C_V \times_V C_V \to (X \times V) \times_V (X \times V)\cong X \times X \times V .
\]
The preimage $(\tilde u_V \times \tilde u_V)^{-1}(\Delta_{X \times V})$ of the diagonal $\Delta_{X \times V} \subset (X \times V) \times_V (X \times V)$ contains the diagonal $\Delta_{C_V} \subset C_V \times_V C_V$ as a closed subscheme.
\begin{df}\label{df:dpl}
The \emph{double point locus} over $V$ is the subscheme 
\[
\dpl_V \subset (\tilde u_V \times \tilde u_V)^{-1}(\Delta_{X \times V})
\]
residual to $\Delta_{C_V}.$ It comes with a canonical projection $\pi :\dpl_V  \to V$ induced by $\varpi.$
\end{df}

\subsubsection{Symmetrization}\label{sssec:sym}
Let $\mfS_n$ denote the symmetric group on the set $\{1,\ldots,n\}.$ 
The group $\mfS_n$ acts on $X^n$ by permuting the factors. Let $\mfS_n$ act on $\M_{g,n}(X,D)$ by permuting the marked points and acting trivially on the underlying curve and the morphism to $X$. That is, $\tau \in \mfS_n$ acts on $\M_{g,n}(X,D)$ by
\[
\tau(u: C \to X, p_1, \ldots, p_n) = (u: C \to X, p_{\tau^{-1}(1)}, \ldots, p_{\tau^{-1}(n)}).
\]
Observe that the evaluation map $\ev:\M_{g,n}(X,D)\to X^n$ is $\mfS_n$-equivariant with respect to the forgoing actions.

Let $X^n_0 \subset X^n$ denote the complement of the pairwise diagonals and let 
\[
\M_{g,n}(X,D)^0 = \ev^{-1}(X^n_0).
\] 
So, the restriction of the action of $\mfS_n$ on $X^n$ to $X^n_0$ is free.
Let $V \subset \M_{g,n}(X,D)^0$ be an open substack preserved by the action of $\mfS_n.$ By the equivariance of $\ev,$ the action of $\mfS_n$ on $V$ is also free. If $V$ is a scheme so that the double point locus $\dpl_V$ is defined, the action of $\mfS_n$ on $V$ lifts canonically to $\dpl_V$ in such a way that $\pi_V$ is $\mfS_n$ equivariant.

Since $X$ is projective, $X^n_0$ is quasi-projective. Assume that $V$ is a scheme, so it is quasi-projective since $\M_{g,n}(X,D)$ admits a projective coarse moduli scheme. Assume also that $\dpl_V$ is quasi-projective.  Let $V_{\mfS}, (\dpl_V)_{\mfS}$ and $\Sym^n_0 X,$ denote the quotients by the action of $\mfS_n$ of $V,\dpl_V,$ and $X^n_0$ respectively. These quotients are also quasi-projective. In the general case when $V \subset \M_{g,n}(X,D)^0$ is an open substack, we also write $V_{\mfS}$ for the quotient by the action of $\mfS_n,$ which is an algebraic stack. Since $\ev:V\to X^n_0$ and $\pi_V : \dpl_V \to V$ are $\mfS_n$-equivariant, they descend to morphisms
\[
 \ev_{\mfS}: V_{\mfS} \to \Sym^n_0 X, \qquad \qquad 
 \pi_{\mfS}:(\dpl_V)_{\mfS} \to V_{\mfS}.
\]

To see why symmetrization is relevant to the present work, specialize for a moment to the case $T = \Spec k$ for $k$ a field.
A $k$ point $p$ of $\Sym^n_0X$ corresponds to a tuple of distinct $L_i$ points $p_i$ of $X$ for $i = 1,\ldots,r,$ where $k \subset L_i$ is a separable extension and $\sum_{i = 1}^r [L_i:k] = n.$ So, $\ev_\mfS^{-1}(p)$ corresponds to rational curves on $X$ passing through $p_1,\ldots,p_r.$ Thus, symmetrization allows us to count curves passing through non-rational points. The field extensions $L_i$ are not fixed. In order to fix the field extensions, which is necessary to obtain invariant counts, we introduce twists below.

\subsubsection{Twists}\label{sssec:twists}\label{subsectionRCsection:twistsev}
Assume now that $T$ is an affine scheme. For $T' \to T$ a finite \'etale map, let $[T':T]$ denote the degree. Let $T_i \to T$ for $i = 1,\ldots,r$ be finite \'etale maps with $\sum_{i = 1}^r [T_i : T] = n$ and write 
\[
\sigma = (T_1,\ldots,T_r).
\]

The list $\sigma$ is used to define twists $\ev_\sigma$ of the evaluation map in the following manner.
Let $\widetilde T \to T$ denote the universal cover, that is, the inverse limit of all finite \'etale $T$-schemes. Let $\overline{\calP}(\sigma)$ denote the $\widetilde T$-points of $\coprod_{i=1}^r T_i$ and let $\mfS_{\overline{\calP}(\sigma)} \cong \mfS_n$ denote the symmetric group. The automorphism group $\Aut(\widetilde T/T)$ acts on $\widetilde T$ points of $T$-schemes.   
Thus $\sigma$ gives rise to a canonical homomorphism 
\[
\Gal(\sigma) :\Aut(\widetilde T /T) \to \mfS_{\overline{\calP}(\sigma)},
\]
For convenience, we fix an identification $\overline{\calP}(\sigma) \cong \{1,2,\ldots, n\}$ and thus an isomorphism $\mfS_{\overline{\calP}(\sigma)} \cong \mfS_n$.

Recall the action of $\mfS_n$ on $X^n$ and $\M_{g,n}(X,D)$ from Section~\ref{sssec:sym}. Let $V \subset \M_{g,n}(X,D)$ be an open substack preserved by the action of $\mfS_n.$ In the case that $V$ is a scheme, recall also the action of $\mfS_n$ on the double point locus $\dpl_V.$ For $Y=X^n$, or $Y =V$, or if $V$ is a scheme and $Y = \dpl_V$, we define a $1$-cocycle 
\[
z(\sigma) : \Aut(\widetilde T /T) \to \Aut (Y_{\widetilde T}), \qquad Y_{\widetilde T} = Y \times_T \widetilde T
\]
by
\begin{equation}\label{twistcocyclenew}
g \mapsto \left[\Gal(\sigma)(g) \times g : Y \times_T \widetilde T \to Y \times_T \widetilde T \right].
\end{equation}  
The $1$-cocycle $z(\sigma)$ determines a twist $Y_{\sigma}$. If $Y$ is a (quasi-)projective $T$-scheme, so is $Y_\sigma$ by Theorem~3 of~\cite{Gro60}. If $Y$ is an Artin stack, so is $Y_\sigma$ by Theorem 1.5 and Lemma 3.5 of~\cite{fortman}. Observe that if $V$ is a scheme, it is quasi-projective since $\M_{g,n}(X,D)$ admits a projective coarse moduli scheme. Since $\ev_{\widetilde T}: V_{\widetilde T} \to X_{\widetilde T}^n$ and $\pi_{\widetilde T}: \dpl_{V_{\widetilde T}} \to V_{\widetilde T}$ are equivariant for the twisted action, they descend to morphisms
\[
 \ev_{\sigma}: V_\sigma \to (X^n)_\sigma, \qquad \qquad 
 \pi_{\sigma}:(\dpl_V)_{\sigma} \to V_\sigma.
\]
We abbreviate $\dpl_\sigma = (\dpl_V)_\sigma.$

\subsubsection{Types of stable maps}\label{sssec:types}
Let $(u: C \to X, p_1, \ldots, p_n)$ be a point of $M_{0,n}(X, D)$. Even though $C$ is smooth, the image curve $u(C)$ is not smooth unless equality holds in the adjunction inequality
\[
D \cdot D + 2  \geq -K_X \cdot D. 
\] 
See~\cite{Hartshorne} Chapter IV Exercise 1.8 and Chapter V Proposition 1.5. In \cite{KLSW-relor}, we study the geometry of $M_{0,n}(X, D)$ using the singularities of the image curve. We say that $u$ is {\em unramified} if $u: C \to f(C)$ is an unramified map of relative curves. Since $C$ is smooth, this is equivalent to the induced map on cotangent bundles $T^*u: u^* T^*X \to T^* C$ being surjective.
Consider now the case that $(u: C \to X, p_1, \ldots, p_n)$ is a geometric point, so $C \simeq \bbb{P}^1_F$ where $F$ is an algebraically closed field. We say that $u$ is \emph{birational} if the map $u : C \to u(C)$ is birational. We say that $u$ has {\em only ordinary double points} if $u$ is unramified, and for each point $q \in X$, the preimage $u^{-1}(q) \subset C$ consists of at most two points, and when $u^{-1}(q)$ consists of two points $c_1,c_2,$ the subspaces $T_{c_i}u(T_{c_i}C) \subset T_qX$ are distinct for $i = 1,2.$ We say that $u$ has a {\em cusp or worse} if there is a point $c \in C$ such that $T_c u \colon T_{c}C \to T_{u(c)}X$ is the zero map. We say $u$ has a {\em tacnode or worse} if there are $c_1 \neq c_2 \in \bbb{C}$ such that $u(c_1) = u(c_2)$ and $T_{c_1}u( T_{c_1}C) = T_{c_2}u(T_{c_2}C)$. 

Following~\cite[Def.~2.4]{KLSW-relor}, let
\[
M^\bir_{0,n}(X,D) \subset M_{0,n}(X,D)
\]
be the open substack the geometric points of which are birational stable maps. Since birational stable maps have no automorphisms, $M^\bir_{0,n}(X,D)$ is a scheme.
Following~\cite[Def.~2.6]{KLSW-relor}, let 
\[
M^\unr_{0,n}(X,D) \subset M_{0,n}(X,D)
\]
be the substack representing unramified stable maps.
Let
\[
M_{0,n}^\odp(X,D) \subset M_{0,n}(X,D)
\]
represent unramified stable maps $(C \to U, u: C \to X, p_1, \ldots, p_n)$ such that for every geometric point $t \in U,$ the pullback map $u_t : C_t \to X$ has only ordinary double points. The following is a consequence of Lemmas~2.8 and~2.14 of~\cite{KLSW-relor}.
\begin{lm}\label{lm:openinc}
We have open inclusions
\[
M_{0,n}^\odp(X,D)\subset M_{0,n}^\unr(X,D) \subset M^\bir_{0,n}(X,D) \subset M_{0,n}(X,D).
\]
\end{lm}
In particular, $M_{0,n}^\odp(X,D)$ and $M_{0,n}^\unr(X,D)$ are schemes.
The scheme $M_{0,n}^\odp(X,D)$ is called the \emph{ordinary double point locus.}
Let $d:=\Deg(-D\cdot K_S).$ It is well-known that $M^\unr_{0,n}(X,D)$ and hence also $M^\odp_{0,n}(X, D)$ is either empty or a smooth scheme of dimension $n+d-1$. See, for example~\cite[Lemma~2.17]{KLSW-relor}.
Following~\cite[Def.~2.13]{KLSW-relor}, let 
\[
Z_\cusp \subset M_{0,n}^\bir(X,D), \qquad Z_\tac \subset M_{0,n}^\unr(X,D),
\]
denote the sets of stable maps with a cusp or worse and with a tacnode or worse, respectively. These are closed subsets by~\cite[Lem. 2.12]{KLSW-relor}. Following~\cite[Def. 2.42]{KLSW-relor} we make the following definition.

\begin{df}
Let $D_\cusp$ be the closure in $\bar M_{0,n}(X,D)$ of the union of codimension one integral components of $Z_\cusp.$ Let $D_\tac$ be the closure in $\bar M_{0,n}(X,D)$ of the union of codimension one integral components of $Z_\tac.$
\end{df}

Let $V \subset \M_{0,n}(X,D)$ be an open substack preserved by the action of $\mfS_n$ such that $\ev(V) \subset X_0^n,$ and let $V_\mfS$ be the quotient by $\mfS_n$ as in Section~\ref{sssec:sym}. We define $D^V_{\cusp,\mfS}$ and $D^V_{\tac,\mfS}$ to be the reduced images of $D_\cusp\cap V$ and $D_\tac\cap V$ in $V_\mfS$ respectively. We may abbreviate $D_{\cusp,\mfS}= D^V_{\cusp,\mfS}$ and $D_{\tac,\mfS} = D^V_{\tac,\mfS}.$

Since the definition of $D_\cusp$ and $D_\tac$ does not make reference to marked points, they are preserved under pullback by forgetful maps $\bar M_{0,n}(X,D) \to \bar M_{0,m}(X,D)$ for $m < n.$ Let $V \subset \bar M_{0,n}(X,D)$ be an open substack preserved by the action of $\mfS_n$, and let $V_\sigma$ be the associated twist as in Section~\ref{sssec:twists}. There is a forgetful map $\eta_\sigma: V_\sigma \to \bar M_{0,0}(X,D)$ to the untwisted moduli space of stable maps without marked points since $\mfS_n$ acts trivially on the domain curve and map. Following~\cite[Section 8]{KLSW-relor}, we define
\[
D^V_{\cusp,\sigma} = \eta_\sigma^{-1}(D_\cusp)\subset V_\sigma, \qquad \qquad D^V_{\tac,\sigma} = \eta_\sigma^{-1}(D_\tac) \subset V_\sigma.
\]
We may abbreviate $D_{\cusp,\sigma} = D^V_{\cusp,\sigma}$ and $D_{\tac,\sigma} = D^V_{\tac,\sigma}.$

\subsection{Orientations}
\subsubsection{The discriminant}\label{sssec:disc}
We will need the following terminology and notation to discuss the canonical orientation of the evaluation map of moduli space of stable maps to a Del Pezzo surface. Let $f:Y\to Z$ be a finite, flat morphism of smooth $k$-schemes. Then  $f_*\calO_Y$ is a finite locally free $\calO_Z$-module. The multiplication map on $\calO_Y$ gives the morphism of $\calO_Z$-modules $m:f_*\calO_Y\otimes_{\calO_Z}f_*\calO_Y\to f_*\calO_Y$. Since $f_*\calO_Y$ is a finite locally free $\calO_Z$-module, we have the trace map $\Tr_f:f_*\calO_Y\to \calO_Z$ defined by sending $s\in f_*\calO_Y(U)$ to the trace of the multiplication map $\times s: f_*\calO_Y(U)\to f_*\calO_Y(U)$. The trace form $$f_* \calO_Y \otimes f_* \calO_Y \to \calO_Z$$ $$ (b, b') \mapsto \Tr(bb') $$ defines a map $\tau: f_* \calO_Y \to f_* \calO_Y^{-1}$. The determinant $$\det \tau: \det  f_* \calO_Y \to \det f_* \calO_Y^{-1}$$ determines a section $\disc_f$ of $\det f_* \calO_Y^{-2}$ and a canonical isomorphism 
\begin{equation}
\label{O(disc(f))=f_*OY^2}\calO(D(\disc_f)) \cong (\det f_* \calO_Y)^{\otimes 2}.
\end{equation}
Since the trace form is non-degenerate if $f$ is \'etale, we see that the divisor of $\disc_f$ is supported on the branch locus of $f$. The associated element of $\calO_Z/(\calO_Z^*)^2$ (see Construction~\ref{sectionofsquareasfn}) has the property that at every closed point $z$ of $Z$, 
$$\disc_{f,z} = \det(\Tr(b_i b_j)_{i,j})$$ where $b_i$ runs over a basis of  $(f_* \calO_Y)_z$ as an $\calO_{Z,z}$-module.

\subsubsection{Del Pezzo surfaces}\label{sssec:delpezzo}
Let $S$ be a del Pezzo surface over a perfect field $k,$ in the sense that $S$ is a geometrically connected, smooth, projective $k$-scheme of dimension $2$ with ample anticanonical bundle $-K_S$. Let $d_S = K_S \cdot K_S$ denote the degree of $S$. Let $D \in Pic(S)$ and set 
\[
d= -K_S \cdot D, \qquad n = d-1.
\]
Let $\sigma = (L_1,\ldots,L_r)$ be an $r$-tuple of field extensions $k \subset L_i \subset \bar k$ such that $\sum_{i = 1}^r [L_i:k] = n.$ Given an open substack $V \subset \bar M_{0,n}(S,D)$ preserved by the action of $\mfS_n$, recall the twisted evaluation map $\ev_\sigma : V_\sigma \to (S^n)_\sigma$ and, when $V$ is a scheme, the double point locus $\pi_\sigma : \dpl_\sigma \to V_\sigma$ defined in Section~\ref{sssec:twists} as well as the divisors $D_{\cusp,\sigma}, D_{\tac,\sigma} \subset V_\sigma$ defined in Section~\ref{sssec:types}. 

\subsubsection{Characteristic zero}

\begin{tm}\label{thm:rel_or_char0}
Assume $k$ is a field of characteristic zero and let $(S,D)$ satisfy Hypothesis~\ref{hyp:basic}. Then, there is a closed subset $A \subset S^n$ with $\codim A \geq 2$ preserved by the action of $\mfS_n$ such that 
\[
V = \bar M_{0,n}(S,D) \setminus \ev^{-1}(A)
\]
satisfies the following.
\begin{enumerate}
\item\label{it:rel_or_char0:scheme}
$V$ is a smooth $k$-scheme.
\item \label{item:rel_or_char0:smooth} $\ev_{\sigma}: V_{\sigma} \to (S^n)_{\sigma}$ is a finite flat morphism of smooth $k$-schemes.
\item\label{item:rel_or_char0:Dsmooth}
$\pi_\sigma : \dpl_\sigma \to V_\sigma$ is a finite flat morphism of smooth $k$-schemes.
\item \label{item:rel_or_char0:detdev}
The canonical section $\det \Tev_{\sigma}:\calO_{V_{\sigma}}\to \omega_{\ev_{\sigma}}$ induces an isomorphism
\[
\det \Tev_{\sigma}:\calO_{V_{\sigma}}(D_{\cusp, \sigma})\to \omega_{\ev_{\sigma}}.
\]
\item \label{item:rel_or_char0:disc_pi}
The canonical section $\disc_{\pi_{\sigma}}:\calO_{V_{\sigma}}\to [\det(\pi_\sigma)_*\calO_{\dpl_{\sigma}}]^{\otimes -2}$ induces an isomorphism
\[
\disc_{\pi_{\sigma}}:\calO_{V_{\sigma}}(D_{\cusp,\sigma})\to [\det (\pi_{\sigma})_*\calO_{\dpl_{\sigma}}(-D_{\tac,\sigma})]^{\otimes -2}.
\]
\item \label{item:rel_or_char0:rel_or}
Letting $\calL_{\sigma}:=[\det(\pi_\sigma)_*\calO_{\dpl_{\sigma}}(-D_{\tac,\sigma})]^{-1}$, the composition
\[
\disc_{\pi_{\sigma}}\circ(\det \Tev_{\sigma})^{-1}:\omega_{\ev_{\sigma}}\to \calL_{\sigma}^{\otimes 2}
\]
is an isomorphism.
\end{enumerate}
\end{tm}
\begin{proof}
Take $V$ to be the subscheme $\bar M_{0,n}(S,D)^\good \subset \bar M_{0,n}(S,D)$ the existence of which is asserted by Theorem~4.5 of~\cite{KLSW-relor}. Then~\ref{it:rel_or_char0:scheme} is Theorem~4.5(1) of~\cite{KLSW-relor}. The projection $\pi : \dpl_V \to V$ is a finite flat morphism of smooth $k$ schemes by Corollary~5.14 of~\cite{KLSW-relor}. Claim~\ref{item:rel_or_char0:Dsmooth} follows since finiteness, flatness and smoothness are preserved under Galois twists. The remaining claims are Theorem 8.1 of~\cite{KLSW-relor}.
\end{proof}

Recall Definition~\ref{df:or_away_codim_c} concerning orientations away from codimension $c.$
\begin{df}\label{df:rel_or_char0}
Assume $k$ is a field of characteristic zero and let $(S,D)$ satisfy Hypothesis~\ref{hyp:basic}. The \textbf{double point orientation} of the evaluation map 
\[
\ev_\sigma : \bar M_{0,n}(S,D)_\sigma \to (S^n)_\sigma
\]
is the orientation away from codimension $2$ given by the open set $U = (S^n \setminus A)_\sigma \subset S^n_\sigma,$ the line bundle  $\calL_\sigma \to V_\sigma = \ev_\sigma^{-1}(U)$ and the isomorphism $\disc_{\pi_{\sigma}}\circ\det \Tev_{\sigma}^{-1} : \omega_{\ev_\sigma} \to \calL_\sigma^{\otimes 2}$ as in Theorem~\ref{thm:rel_or_char0}.
\end{df}

\begin{rmk}\label{rmk:Hyp1needed}
    The assumptions in Hypothesis~\ref{hyp:basic} are used to construct $V$ in Theorem~\ref{thm:rel_or_char0}. For example, assuming that $D$ is not an $m$-fold multiple of a $-1$-curve for $m>1$ ensures that $d=D \cdot (-K_S)$ is not very negative. Negativity of $D \cdot (-K_S)$ can produce a difference between the expected dimension and dimension of spaces of such maps. Moreover, the assumption that $D$ is not an $m$-fold multiple of a $-1$-curve also ensures that the degrees of certain associated maps do not get too small either. See \cite[Lemma 3.10]{KLSW-relor}. This is useful in ensuring that stable maps with domain curves having at least 3 components are codimension $2$. See \cite[Lemma 3.12]{KLSW-relor}. The restrictions on $d$ for low degree del Pezzo surfaces come from excluding undesirable behaviors of image curves from codimension $0$ and $1$ loci of stable maps. For example, these restrictions are used to show that the locus of image curves with a tacnode of order $\geq 2$ is codimension $2$. See \cite[Proposition 2.38]{KLSW-relor}.
\end{rmk}

\subsubsection{Positive characteristic: Orientation away from codimension 1}
We continue with the notation of Section~\ref{sssec:delpezzo}. The following is Lemma~2.27 of~\cite{KLSW-relor}.
\begin{lm} \label{ev:etale_odp} If $k$ is a perfect field, then $M^\odp_{0,n}(S, D)$ is a smooth scheme and $\ev: M^\odp_{0,n}(S, D) \to S^n$ is \'etale.
\end{lm}

The following is a special case of Corollary~3.15 of~\cite{KLSW-relor}.
\begin{tm}\label{tm:eodpd}
Under Hypothesis~\ref{hyp:basic} when $\Char k = 0$ and Hypothesis~\ref{hyp:pc} when $\Char k > 0,$ we have 
\[
\codim \ev(\bar{M}_{0,d-1}(S, D) \setminus M^{\odp}_{0,d-1}(S,D)) \ge 1.
\]
\end{tm}

Let $\dpl_{M^\odp_{0,n}(S, D)}$ denote the double point locus over $M^\odp_{0,n}(S, D)$ as in Definition~\ref{df:dpl}. The following is Lemma 5.4 of~\cite{KLSW-relor}.
\begin{lm}\label{lm:dplodp}
If $k$ is a perfect field, the projection from the double point locus 
\[
\pi: \dpl_{M^\odp_{0,n}(S, D)} \to M_{0,n}^\odp(S,D)
\]
is finite \'etale.
\end{lm}
\begin{proof}
Lemma 5.4 of~\cite{KLSW-relor} asserts that $\pi$ is \'etale. It follows from the definition that $\pi$ is proper, so $\pi$ is finite. 
\end{proof}

\begin{tm}\label{thm:rel_or_charpsf}
Suppose $k$ and $(S,D)$ satisfy Hypothesis~\ref{hyp:pc}. Then, there is a closed subset $A \subset S^n$ with $\codim A \geq 1$ preserved by the action of $\mfS_n$ such that 
\[
V = \bar M_{0,n}(S,D) \setminus \ev^{-1}(A)
\]
satisfies the following.
\begin{enumerate}
\item \label{item:rel_or_charpsf:scheme}
$V$ is a smooth $k$-scheme.
\item \label{item:rel_or_charpsf:smooth} $\ev_{\sigma}: V_{\sigma} \to (S^n)_{\sigma}$ is a finite flat morphism of smooth $k$-schemes.
\item\label{item:rel_or_charpsf:Dsmooth}
$\pi_\sigma : \dpl_\sigma \to V_\sigma$ is a finite flat morphism of smooth $k$-schemes.
\item \label{item:rel_or_charpsf:detdev}
The canonical section $\det \Tev_{\sigma}:\calO_{V_{\sigma}}\to \omega_{\ev_{\sigma}}$ is an isomorphism.
\item \label{item:rel_or_charpsf:disc_pi}
The canonical section $\disc_{\pi_{\sigma}}:\calO_{V_{\sigma}}\to [\det(\pi_\sigma)_*\calO_{\dpl_{\sigma}}]^{\otimes -2}$ is an isomorphism.
\item \label{item:rel_or_charpsf:rel_or}
Letting $\calL_{\sigma}:=[\det(\pi_\sigma)_*\calO_{\dpl_{\sigma}}]^{-1}$, the composition
\[
\disc_{\pi_{\sigma}}\circ(\det \Tev_{\sigma})^{-1}:\omega_{\ev_{\sigma}}\to(\calL_{\sigma})^{\otimes 2}
\]
is an isomorphism.
\end{enumerate}
\end{tm}

\begin{proof}
Since $M_{0,n}(S,D) \subset \bar{M}_{0,n}(S,D)$ is open, it follows from Lemma~\ref{lm:openinc} that the subset $M_{0,n}^\odp(S,D) \subset \bar{M}_{0,n}(S,D)$ is open, and thus $\bar{M}_{0,n}(S,D) \setminus M_{0,n}^\odp(S,D)$ is closed. Let
\[
A = \ev(\bar{M}_{0,n}(S,D)\setminus M_{0,n}(S,D)^{\odp}) \subset S^n.
\]
Since the evaluation map $\ev: \bar{M}_{0,n}(S,D) \to S^n$ is proper, $A$ is closed. Theorem~\ref{tm:eodpd} asserts that $A$
has positive codimension. Since
\[
V = \bar M_{0,n}(S,D) \setminus \ev^{-1}(A) \subset M_{0,n}(S,D)^{\odp},
\]
Lemma~\ref{ev:etale_odp} implies that $V$ is smooth proving~\ref{item:rel_or_charpsf:scheme}. It also follows from Lemma~\ref{ev:etale_odp} that $\ev|_V$ is \'etale and thus flat and quasi-finite. As the base-change of a proper morphism, $\ev|_V$ is proper and thus finite. It follows that the Galois twist $\ev_\sigma : V_\sigma \to (S^n)_\sigma$ is also \'etale and finite proving~\ref{item:rel_or_char0:smooth} and~\ref{item:rel_or_charpsf:detdev}. The map $\pi : \dpl_V \to V$ is finite \'etale because it is the base change of $\pi: \dpl_{M^\odp_{0,n}(S, D)} \to M_{0,n}^\odp(S,D)$, which is finite \'etale by Lemma~\ref{lm:dplodp}. This proves~\ref{item:rel_or_charpsf:Dsmooth} and~\ref{item:rel_or_charpsf:disc_pi}. Finally,~\ref{item:rel_or_charpsf:rel_or} follows from~\ref{item:rel_or_charpsf:detdev} and~\ref{item:rel_or_charpsf:disc_pi}.
\end{proof}

\begin{df}\label{df:rel_or_charpsf}
Suppose $k$ and $(S,D)$ satisfy Hypothesis~\ref{hyp:pc}. The \textbf{double point orientation} away from codimension $1$ of the evaluation map 
\[
\ev_\sigma : \bar M_{0,n}(S,D)_\sigma \to (S^n)_\sigma
\]
is the orientation away from codimension $1$ given by the open set $U = S^n \setminus A \subset S^n,$ the line bundle  $\calL_\sigma \to V_\sigma = \ev_\sigma^{-1}(U)$ and the isomorphism $\disc_{\pi_{\sigma}}\circ\det \Tev_{\sigma}^{-1} : \omega_{\ev_\sigma} \to \calL_\sigma^{\otimes 2}$ as in Theorem~\ref{thm:rel_or_charpsf}.
\end{df}

\begin{rmk}\label{rmk:useHyp2}
    Hypothesis~\ref{hyp:pc} is used for example to show that the generic birational stable map is in $M^\odp_{0,n}(S, D)$. See \cite[Proposition 2.32, Assumption 2.30]{KLSW-relor}
\end{rmk}

\begin{rmk}\label{rmk:hyp2satisfied}
    The additional condition in Hypothesis~\ref{hyp:pc} that for every effective $D' \in Pic(S)$, there is a geometric point $f$ in each irreducible component of $M^\bir_0(S, D')$ with $f$ unramified is satisfied if $k$ is characteristic $0$ or if $k$ is characteristic $>3$ and $d_S \geq 3$. This follows by \cite{Testa-irreducibility} \cite{GoettschePand} (see \cite[Lemma 2.31]{KLSW-relor}) and \cite[Theorem A.1]{KLSW-relor}, respectively.
\end{rmk}

\subsubsection{Positive characteristic: Lifting data}\label{sssec:pclift}
Assume now that $k$ has positive characteristic. Let $\Lambda$ be a complete discrete valuation ring with residue field $k$ and quotient field $K$ of characteristic $0$. 

The following is a restatement of~\cite[Lemma~9.3]{KLSW-relor}.
\begin{lm}\label{lm:LiftingDelPezzo} Let $S$ be a del Pezzo surface over a field $k$, with effective Cartier divisor $D$.
\begin{enumerate}
\item\label{lm:LiftingDelPezzo1} There is a smooth proper $\Lambda$-scheme $\pi:\LiftS\to \Spec\Lambda$ with an isomorphism $\phi:\LiftS_k\xrightarrow{\sim} S$.
\item\label{lm:LiftingDelPezzo2}  For each lifting $(\LiftS, \phi)$ of $S$ over $\Lambda$ as in \ref{lm:LiftingDelPezzo1}, letting $i:S\to \LiftS$ be the closed immersion induced by $\phi$, the restriction map $i^*:\Pic(\LiftS)\to \Pic(S)$ is an isomorphism. 
\item\label{lm:LiftingDelPezzo3} For each lifting $(\LiftS, \phi)$ of $S$ over $\Lambda$ as in \ref{lm:LiftingDelPezzo1}, $\LiftS$ is a del Pezzo surface over $\Lambda$ and the generic fiber $\LiftS_K$ is a del Pezzo surface over $K$. Moreover, we have $d_{\LiftS_K}=d_S$. 
\item\label{lm:LiftingDelPezzo4} For each lifting $(\LiftS, \phi)$ of $S$ over $\Lambda$ as in \ref{lm:LiftingDelPezzo1}, there is an effective Cartier divisor $\LiftDeg$ on $\LiftS$ with $\phi(\LiftDeg_k)=D$.  Moreover, we have $\deg_K(-K_{\LiftS_K}\cdot \LiftDeg_K)=d$.
\end{enumerate}
\end{lm}

The reduction map defines an equivalence between the category of finite \'{e}tale extensions of $\Lambda$ and the analogous category over $k$ \cite[Expos\'e IX 1.10]{sga1}. Let $\Lambda_i$ be the finite \'etale extension of $\Lambda$ corresponding to the extension $L_i$ of $k.$ So,
\[
\sum_{i = 1}^r [\Lambda_i:\Lambda] = \sum_{i = 1}^r [L_i:k] = n. 
\]
Let $\Liftsig = (\Lambda_1,\ldots,\Lambda_r)$. Let $\LiftS$ be a lifting of $S$ over $\Lambda$ as in Lemma~\ref{lm:LiftingDelPezzo}.
Given a substack $\LiftV \subset \bar M_{0,n}(\LiftS,\LiftDeg)$ preserved by the action of $\mfS_n$, let $\Liftev_\sigma : \LiftV_\sigma \to (\LiftS^n)_\sigma$ denote the twisted evaluation map and let $\Liftpi_\sigma : \Liftdpl_\sigma \to \LiftV_\sigma$ denote the twisted double point locus associated to $\Liftsig$ as defined in Section~\ref{sssec:twists}. Let $D_{\cusp,\sigma}, D_{\tac,\sigma} \subset \LiftV_\sigma$ be the divisors defined in Section~\ref{sssec:types}. 

\begin{tm}\label{thm:rel_or_pchar}
Let $k$ and $(S,D)$ satisfy Hypothesis~\ref{hyp:pc}. Then, there is a closed subset $\LiftA \subset \LiftS^n$ with $\codim \LiftA \geq 2$ preserved by the action of $\mfS_n$ such that 
\[
\LiftV = \bar M_{0,n}(\LiftS,\LiftDeg) \setminus \ev^{-1}(\LiftA)
\]
satisfies the following.
\begin{enumerate}
\item\label{item:thm:rel_or_twist_ev_charp:Vsmooth}
$\LiftV$ is a smooth $\Lambda$-scheme.
\item \label{item:thm:rel_or_twist_ev_charp:smooth} $\Liftev_{\sigma}: \LiftV_\sigma \to (\LiftS^n)_{\sigma}$ is a finite flat map of smooth $\Lambda$-schemes.

\item \label{item:thm:rel_or_twist_ev_charp:Dsmooth}
$\Liftpi_\sigma : \Liftdpl_\sigma \to \LiftV_\sigma$ is a finite flat map of smooth $\Lambda$-schemes.
\item \label{item:thm:rel_or_twist_ev_charp:detdev}
The canonical section $\det T\Liftev_{\sigma}:\calO_{\LiftV_\sigma}\to \omega_{\Liftev_{\sigma}}$ induces an isomorphism
\[
\det T\Liftev_{\sigma}:\calO_{\LiftV_\sigma}(D_{\cusp, \sigma})\to \omega_{\Liftev_{\sigma}}.
\]
\item \label{item:thm:rel_or_twist_ev_charp:disc_pi}
The canonical section $\disc_{\Liftpi_{\sigma}}:\calO_{\LiftV_\sigma}\to [\det(\Liftpi_\sigma)_*\calO_{\Liftdpl_{\sigma}}]^{\otimes -2}$ induces an isomorphism
\[
\disc_{\Liftpi_{\sigma}}:\calO_{\LiftV_\sigma}(D_{\cusp,\sigma})\to [\det (\Liftpi_{\sigma})_*\calO_{\Liftdpl_{\sigma}}(-D_{\tac,\sigma})]^{\otimes -2}.
\]
\item \label{item:thm:rel_or_twist_ev_charp:rel_or}
Letting $\LiftL_{\sigma}:=[\det(\Liftpi_\sigma)_*\calO_{\Liftdpl_{\sigma}}(-D_{\tac,\sigma})]^{-1}$, the composition
\[
\disc_{\Liftpi_{\sigma}}\circ(\det T\Liftev_{\sigma})^{-1}:\omega_{\Liftev_{\sigma}}\to \LiftL_{\sigma}^{\otimes 2}
\]
is an isomorphism.
\end{enumerate}
\end{tm}

\begin{proof}
Construction~9.6 of~\cite{KLSW-relor} gives $\LiftV$. 
Claim~\ref{item:thm:rel_or_twist_ev_charp:Vsmooth} is Proposition~9.9 of~\cite{KLSW-relor}. The finiteness of $\Liftev: \LiftV \to (\LiftS^n)$ is proved in the beginning of Section~9.3 of~\cite{KLSW-relor}. So, since the domain and range are smooth, $\Liftev$ is flat. Claim~\ref{item:thm:rel_or_twist_ev_charp:smooth} follows since smoothness, finiteness and flatness are preserved by Galois twists. The projection $\pi: \Liftdpl \to \LiftV$ is a finite flat morphism of smooth $\Lambda$-schemes by Lemma~9.11 of~\cite{KLSW-relor}. Claim~\ref{item:thm:rel_or_twist_ev_charp:Dsmooth} follows since smoothness, finiteness and flatness are preserved by Galois twists. The remaining claims are Theorem~9.15 of~\cite{KLSW-relor}.
\end{proof}

Recall Definitions~\ref{df:lifting_data} and~\ref{df:pseudo-oriented} concerning lifting data and pseudo-orientations.
\begin{lm}\label{lm:psor}
The equivalence class of the double point orientation away from codimension~$1$ of the evaluation map
$
\ev_\sigma: \bar M_{0,n}(S,D)_\sigma \to (S^n)_\sigma
$
given in Definition~\ref{df:rel_or_charpsf} is a pseudo-orientation.
\end{lm}
\begin{proof}
We construct lifting data for the double point orientation away from codimension $1$. In the notation of Definition~\ref{df:lifting_data}, we have 
\[
X = \bar M_{0,n}(S,D)_\sigma, \qquad Y = (S^n)_\sigma, \qquad f = \ev_\sigma : \bar M_{0,n}(S,D)_\sigma \to (S^n)_\sigma.
\] 
The orientation away from codimension $1$ is given by
\[
U = S^n \setminus A \subset S^n, \qquad L = \calL_\sigma \to V_\sigma = \ev_\sigma^{-1}(U), \qquad \rho = \disc_{\pi_{\sigma}}\circ\det \Tev_{\sigma}^{-1} : \omega_{\ev_\sigma} \to \calL_\sigma^{\otimes 2}
\]
as in Definition~\ref{df:rel_or_charpsf}. Using Theorem~\ref{thm:rel_or_pchar}, we claim that lifting data is given by 
\begin{gather*}
\calY = (\LiftS^n)_\sigma \to \Spec \Lambda,  \qquad \calU = (\LiftS^n \setminus \LiftA)_\sigma, \qquad  \calX = \LiftV_\sigma  = \Liftev_\sigma^{-1}(\calU), \\
\mathfrak{f} = \Liftev_\sigma : \LiftV_\sigma \to  (\LiftS^n \setminus \LiftA)_\sigma, \qquad \calL = \LiftL_\sigma \to \LiftV_\sigma, \qquad \rhol  = \disc_{\Liftpi_{\sigma}}\circ\det T\Liftev_{\sigma}^{-1}: \omega_{\Liftev_{\sigma}} \to \LiftL_{\sigma}^{\otimes 2}.
\end{gather*}
Indeed, since $\codim(\LiftA_\sigma \subset (\LiftS^n)_\sigma) \geq 2,$ it follows that 
\[
\codim (\LiftA_\sigma \cap S^n \subset S^n) \geq 1,
\]
and thus $\calU \cap S^n = \calU \cap Y$ is dense in $U$ as required for lifting data. 
Since the moduli space of stable maps respects base change, the evaluation map $\mathfrak{f} = \Liftev_\sigma : \LiftV_\sigma \to \calU$ is a lifting of 
\[
f|_{f^{-1}(\calU \cap U)}  = \ev_\sigma|_{\ev_\sigma^{-1}(\calU \cap U)} : \ev_\sigma^{-1}(\calU \cap U) \to \calU \cap U
\]
to a map of $\Lambda$ schemes. Theorem~\ref{thm:rel_or_pchar}\ref{item:thm:rel_or_twist_ev_charp:smooth} asserts that $\mathfrak{f}$ is a finite flat map of smooth $\Lambda$-schemes and thus proper, generically finite and local complete intersection as required for lifting data. Since double point locus and pushforward by the associated projection morphism respect base change, the line bundle $\LiftL_\sigma \to \LiftV_\sigma$ is a lift of $\calL_\sigma|_{\ev_\sigma^{-1}(\calU \cap U)}.$ Since the discriminant respects base change, the isomorphism  $\rhol= \disc_{\Liftpi_{\sigma}}\circ\det T\Liftev_{\sigma}^{-1}:\omega_{\Liftev_{\sigma}}\to \LiftL_{\sigma}^{\otimes 2}$ lifts the isomorphism $\rho = \disc_{\pi_{\sigma}}\circ\det \Tev_{\sigma}^{-1} : \omega_{\ev_\sigma} \to \calL_\sigma^{\otimes 2}.$ 
Thus, we have constructed the desired lifting data. 
\end{proof}
In light of Lemma~\ref{lm:psor}, we make the following definition.
\begin{df}\label{df:rel_por}
The \textbf{double point pseudo-orientation} of $\ev_\sigma: \bar M_{0,n}(S,D)_\sigma \to (S^n)_\sigma$ is given by the equivalence class of the double point orientation away from codimension $1.$
\end{df}

\subsubsection{Symmetrization in characteristic zero}
Given an open substack $V \subset \bar M_{0,n}(S,D)^0$ preserved by the action of $\mfS_n,$ recall the symmetrized evaluation map $\ev_\mfS : V_\mfS \to \Sym^n_0 S$ and double point locus $\pi_\mfS : \dpl_\mfS \to V_\mfS$ defined in Section~\ref{sssec:sym} as well as the divisors $D_{\cusp,\mfS}, D_{\tac,\mfS} \subset V_\mfS$ defined in Section~\ref{sssec:types}.
\begin{tm}\label{thm:rel_or_char0_sym}
Assume $k$ is a field of characteristic zero and let $(S,D)$ satisfy Hypothesis~\ref{hyp:basic}. Then, there is a closed subset $A \subset S^n$ containing the pairwise diagonals, preserved by the action of $\mfS_n,$ with $\codim A \geq 2,$ such that 
\[
V = \bar M_{0,n}(S,D) \setminus \ev^{-1}(A)
\]
satisfies the following.
\begin{enumerate}
\item\label{it:rel_or_char0_sym:scheme}
$V$ is a smooth $k$-scheme.
\item \label{item:rel_or_char0_sym:smooth} $\ev_{\mfS}: V_{\mfS} \to \Sym^n_0 S$ is a finite flat morphism of smooth $k$-schemes.
\item\label{item:rel_or_char0_sym:Dsmooth}
$\pi_\mfS : \dpl_\mfS \to V_\mfS$ is a finite flat morphism of smooth $k$-schemes.
\item \label{item:rel_or_char0_sym:detdev}
The canonical section $\det \Tev_{\mfS}:\calO_{V_{\mfS}}\to \omega_{\ev_{\mfS}}$ induces an isomorphism
\[
\det \Tev_{\mfS}:\calO_{V_{\mfS}}(D_{\cusp, \mfS})\to \omega_{\ev_{\mfS}}.
\]
\item \label{item:rel_or_char0_sym:disc_pi}
The canonical section $\disc_{\pi_{\mfS}}:\calO_{V_{\mfS}}\to [\det(\pi_\mfS)_*\calO_{\dpl_{\mfS}}]^{\otimes -2}$ induces an isomorphism
\[
\disc_{\pi_{\mfS}}:\calO_{V_{\mfS}}(D_{\cusp,\mfS})\to [\det (\pi_{\mfS})_*\calO_{\dpl_{\mfS}}(-D_{\tac,\mfS})]^{\otimes -2}.
\]
\item \label{item:rel_or_char0_sym:rel_or}
Letting $\calL_{\mfS}:=[\det(\pi_\mfS)_*\calO_{\dpl_{\mfS}}(-D_{\tac,\mfS})]^{-1}$, the composition
\[
\disc_{\pi_{\mfS}}\circ(\det \Tev_{\mfS})^{-1}:\omega_{\ev_{\mfS}}\to \calL_{\mfS}^{\otimes 2}
\]
is an isomorphism
\end{enumerate}
\end{tm}
\begin{proof}
Take $A$ and $V$ as in Theorem~\ref{thm:rel_or_char0} with $\sigma = (k,\ldots,k)$. Since the pairwise diagonals of $S^n$ are codimension $2,$ we may, if necessary, add them to $A$ without altering any of the assertions of Theorem~\ref{thm:rel_or_char0}. So,~\ref{it:rel_or_char0:scheme} is Theorem~\ref{thm:rel_or_char0}\ref{it:rel_or_char0:scheme}. Since finiteness, flatness and smoothness are preserved under quotients by free actions of finite groups,~\ref{item:rel_or_char0_sym:smooth} and~\ref{item:rel_or_char0_sym:Dsmooth} follow from Theorem~\ref{thm:rel_or_char0}\ref{item:rel_or_char0:smooth} and~\ref{item:rel_or_char0:Dsmooth} respectively. The relative canonical sheaf of a morphism and the discriminant of a morphism are both compatible with \'etale base change. Moreover, the divisor of a section of an invertible sheaf is detectible after \'etale base change. Consequently,~\ref{item:rel_or_char0_sym:detdev} and~\ref{item:rel_or_char0_sym:disc_pi} follow from Theorem~\ref{thm:rel_or_char0}\ref{item:rel_or_char0:detdev} and~\ref{item:rel_or_char0:disc_pi} respectively. Finally,~\ref{item:rel_or_char0_sym:rel_or} follows from~\ref{item:rel_or_char0_sym:detdev} and~\ref{item:rel_or_char0_sym:disc_pi}.
\end{proof}

Recall Definition~\ref{df:or_away_codim_c} concerning orientations away from codimension $c.$
\begin{df}\label{df:rel_or_char0_sym}
Assume $k$ is a field of characteristic zero and let $(S,D)$ satisfy Hypothesis~\ref{hyp:basic}. The \textbf{double point orientation} of the symmetrized evaluation map 
\[
\ev_\mfS : \bar M_{0,n}(S,D)^0_\mfS \to \Sym^n_0 S
\]
is the orientation away from codimension $2$ given by the open set 
\[
U = (S^n \setminus A)/\mfS_n \subset \Sym^n_0 S,
\]
the line bundle  $\calL_\mfS \to V_\mfS = \ev_\mfS^{-1}(U)$ and the isomorphism $\disc_{\pi_{\mfS}}\circ\det \Tev_{\mfS}^{-1} : \omega_{\ev_\mfS} \to \calL_\mfS^{\otimes 2}$ as in Theorem~\ref{thm:rel_or_char0_sym}.
\end{df}

\subsubsection{Symmetrization in positive characteristic: Orientation away from codimension 1}
\begin{tm}\label{thm:rel_or_charpsf_sym}
Assume $k$ and $(S,D)$ satisfy Hypothesis~\ref{hyp:pc}. Then, there is a closed subset $A \subset S^n$ containing the pairwise diagonals, preserved by the action of $\mfS_n,$ with $\codim A \geq 1,$ such that 
\[
V = \bar M_{0,n}(S,D) \setminus \ev^{-1}(A)
\]
satisfies the following.
\begin{enumerate}
\item \label{item:rel_or_charpsf_sym:scheme}
$V$ is a smooth $k$-scheme.
\item \label{item:rel_or_charpsf_sym:smooth} $\ev_{\mfS}: V_{\mfS} \to \Sym^n_0 S$ is a finite flat morphism of smooth $k$-schemes.
\item\label{item:rel_or_charpsf_sym:Dsmooth}
$\pi_\mfS : \dpl_\mfS \to V_\mfS$ is a finite flat morphism of smooth $k$-schemes.
\item \label{item:rel_or_charpsf_sym:detdev}
The canonical section $\det \Tev_{\mfS}:\calO_{V_{\mfS}}\to \omega_{\ev_{\mfS}}$ is an isomorphism.
\item \label{item:rel_or_charpsf_sym:disc_pi}
The canonical section $\disc_{\pi_{\mfS}}:\calO_{V_{\mfS}}\to [\det(\pi_\mfS)_*\calO_{\dpl_{\mfS}}]^{\otimes -2}$ is an isomorphism.
\item \label{item:rel_or_charpsf_sym:rel_or}
Letting $\calL_{\mfS}:=[\det(\pi_\mfS)_*\calO_{\dpl_{\mfS}}]^{-1}$, the composition
\[
\disc_{\pi_{\mfS}}\circ(\det \Tev_{\mfS})^{-1}:\omega_{\ev_{\mfS}}\to(\calL_{\mfS})^{\otimes 2}
\]
is an isomorphism.
\end{enumerate}
\end{tm}
\begin{proof}
Similarly to the proof of Theorem~\ref{thm:rel_or_char0_sym} from Theorem~\ref{thm:rel_or_char0}, the present theorem follows from Theorem~\ref{thm:rel_or_charpsf}.
\end{proof}

\begin{df}\label{df:rel_or_charpsf_sym}
Assume $k$ and $(S,D)$ satisfy Hypothesis~\ref{hyp:pc}. The \textbf{double point orientation} away from codimension $1$ of the symmetrized evaluation map 
\[
\ev_\mfS : \bar M_{0,n}(S,D)^0_\mfS \to \Sym^n_0 S
\]
is the orientation away from codimension $1$ given by the open set 
\[
U = (S^n \setminus A)/\mfS_n \subset \Sym^n_0 S,
\] 
the line bundle  $\calL_\mfS \to V_\mfS = \ev_\mfS^{-1}(U)$ and the isomorphism $\disc_{\pi_{\mfS}}\circ\det \Tev_{\mfS}^{-1} : \omega_{\ev_\mfS} \to \calL_\mfS^{\otimes 2}$ as in Theorem~\ref{thm:rel_or_charpsf_sym}.
\end{df}

\subsubsection{Symmetrized lifting data}
In the following, we freely use the notation introduced in Section~\ref{sssec:pclift}.
\begin{tm}\label{thm:rel_or_pchar_sym}
Let $k$ and $(S,D)$ satisfy Hypothesis~\ref{hyp:pc}. Then, there is a closed subset $\LiftA \subset \LiftS^n$ with $\codim \LiftA \geq 2$ preserved by the action of $\mfS_n$ such that 
\[
\LiftV = \bar M_{0,n}(\LiftS,\LiftDeg) \setminus \ev^{-1}(\LiftA)
\]
satisfies the following.
\begin{enumerate}
\item\label{item:thm:rel_or_twist_ev_charp:Vsmooth_sym}
$\LiftV$ is a smooth $\Lambda$-scheme.
\item \label{item:thm:rel_or_twist_ev_charp:smooth_sym} $\Liftev_{\mfS}: \LiftV_\mfS \to \Sym^n_0\LiftS$ is a finite flat map of smooth $\Lambda$-schemes.
\item \label{item:thm:rel_or_twist_ev_charp:Dsmooth_sym}
$\Liftpi_\mfS : \Liftdpl_\mfS \to \LiftV_\mfS$ is a finite flat map of smooth $\Lambda$-schemes.
\item \label{item:thm:rel_or_twist_ev_charp:detdev_sym}
The canonical section $\det T\Liftev_{\mfS}:\calO_{\LiftV_\mfS}\to \omega_{\Liftev_{\mfS}}$ induces an isomorphism
\[
\det T\Liftev_{\mfS}:\calO_{\LiftV_\mfS}(D_{\cusp, \mfS})\to \omega_{\Liftev_{\mfS}}.
\]
\item \label{item:thm:rel_or_twist_ev_charp:disc_pi_sym}
The canonical section $\disc_{\Liftpi_{\mfS}}:\calO_{\LiftV_\mfS}\to [\det(\Liftpi_\mfS)_*\calO_{\Liftdpl_{\mfS}}]^{\otimes -2}$ induces an isomorphism
\[
\disc_{\Liftpi_{\mfS}}:\calO_{\LiftV_\mfS}(D_{\cusp,\mfS})\to [\det (\Liftpi_{\mfS})_*\calO_{\Liftdpl_{\mfS}}(-D_{\tac,\mfS})]^{\otimes -2}.
\]
\item \label{item:thm:rel_or_twist_ev_charp:rel_or_sym}
Letting $\LiftL_{\mfS}:=[\det(\Liftpi_\mfS)_*\calO_{\Liftdpl_{\mfS}}(-D_{\tac,\mfS})]^{-1}$, the composition
\[
\disc_{\Liftpi_{\mfS}}\circ(\det T\Liftev_{\mfS})^{-1}:\omega_{\Liftev_{\mfS}}\to \LiftL_{\mfS}^{\otimes 2}
\]
is an isomorphism.
\end{enumerate}
\end{tm}
\begin{proof}
Similarly to the proof of Theorem~\ref{thm:rel_or_char0_sym} from Theorem~\ref{thm:rel_or_char0}, the present theorem follows from Theorem~\ref{thm:rel_or_pchar}.
\end{proof}

Recall Definitions~\ref{df:lifting_data} and~\ref{df:pseudo-oriented} concerning lifting data and pseudo-orientations.
\begin{lm}\label{lm:psor_sym}
The equivalence class of the double point orientation away from codimension~$1$ of the evaluation map
$
\ev_\mfS: \bar M_{0,n}(S,D)^0_\mfS \to \Sym^n_0 S
$
given in Definition~\ref{df:rel_or_charpsf_sym} is a pseudo-orientation.
\end{lm}
\begin{proof}
Similarly to the proof of Lemma~\ref{lm:psor} from Theorem~\ref{thm:rel_or_pchar}, the present theorem follows from Theorem~\ref{thm:rel_or_pchar_sym}.
\end{proof}

In light of Lemma~\ref{lm:psor_sym}, we make the following definition.
\begin{df}\label{df:rel_por_sym}
The \textbf{double point pseudo-orientation} of the symmetrized evaluation map
\[
\ev_\mfS: \bar M_{0,n}(S,D)^0_\mfS \to \Sym^n_0 S
\]
is given by the equivalence class of the double point orientation away from codimension $1.$
\end{df}

\section{The degree of \texorpdfstring{$\ev$}{ev}}\label{sec:degev}

In light of Example~\ref{ex:or_away_codim_2_is_pseudo-orientation}, Definitions~\ref{df:rel_or_char0} and~\ref{df:rel_por} in characteristic $0$ and positive characteristic, respectively, give a pseudo-orientation of 
\[
\ev_{\sigma}:\bar{M}_{0,n}(S,D)_{\sigma} \to S^n_{\sigma},
\] 
under Hypotheses~\ref{hyp:basic} and \ref{hyp:pc}, respectively, on $(k,S,D)$. So, Theorem~\ref{tm:A1Degree_pseudo-oriented} gives
\[
\deg(\ev_\sigma) \in \sGW(S^n_\sigma).
\]
Proposition~\ref{pr:htpyinvtsheaf_iso_Xtopi0} allows us to make the following definition.
\begin{df}\label{df:qecc}
The quadratically enriched count of rational curves on $S$ of degree $D$ associated with the extension fields $\sigma = (L_1,\ldots,L_r),$
\[
\underline{N}_{S,D,\sigma} \in \sGW(\pi_0^\Aone(S^n_\sigma)),
\]
is the invariant of $\Aone$-connected components associated to
$\deg(\ev_\sigma) \in \sGW(S^n_\sigma).$
When $S$ is $\Aone$-connected, the corresponding invariant in $GW(k)$ is denoted by $N_{S,D,\sigma}.$
\end{df}
 
Similarly, Definitions~\ref{df:rel_or_char0_sym} and~\ref{df:rel_por_sym} in characteristic zero and positive characteristic, respectively, give a pseudo-orientation of 
\[
\ev_\mfS: \bar M_{0,n}(S,D)^0_\mfS \to \Sym^n_0 S,
\]
under Hypotheses~\ref{hyp:basic} and \ref{hyp:pc}, respectively. So, Theorem~\ref{tm:A1Degree_pseudo-oriented} gives
\[
\deg(\ev_\mfS) \in \sGW(\Sym^n_0 S).
\]
Proposition~\ref{pr:htpyinvtsheaf_iso_Xtopi0} allows us to make the following definition.
\begin{df}\label{df:qecc_sym}
The symmetrized quadratically enriched count of rational curves on $S$ of degree $D$
\[
\underline{N}_{S,D}^\mfS \in \sGW(\pi_0^\Aone(\Sym^n_0S)),
\]
is the invariant of $\Aone$-connected components associated to
$\deg(\ev_\mfS) \in \sGW(\Sym^n_0S).$
\end{df}

In the remainder of this section, we compare $\underline{N}_{S,D,\sigma}$ and $\underline{N}_{S,D}^\mfS$.

Let $q : S^n_0 \to \Sym^n_0 S$ denote the quotient map and abbreviate $S^n_{\sigma,0}$ for the Galois twist of $S^n_0.$ Since the symmetric group acts trivially on $\Sym^n_0 S,$ it follows that $(\Sym^n_0 S)_\sigma \simeq \Sym^n_0 S.$ Thus, the Galois twist of $q$ gives 
\[
q_\sigma : S^n_{\sigma,0} \to \Sym_0^n S.
\]  
Similarly, the Galois twist of the quotient map $Q: \bar{M}_{0,n}(S,D)^0 \to \bar M_{0,n}(S,D)^0_\mfS$ gives
\[
Q_\sigma : \bar{M}_{0,n}(S,D)^0_\sigma \to \bar M_{0,n}(S,D)^0_\mfS
\]
There is a pullback diagram
\begin{equation}\label{cd:twistsym}
\xymatrix{\bar{M}_{0,n}(S,D)^0_{\sigma} \ar[d]^{\ev_{\sigma}} \ar[rr]^{Q_\sigma}  && \ar[d]^{\ev_{\mfS_n}}\bar{M}_{0,n}(S,D)^0_{\mfS} \\
S^n_{\sigma,0} \ar[rr]^{q_\sigma} && \Sym^n_0 S.}
\end{equation}
Let
\[
i_{\sigma,0}: S^n_{\sigma,0} \to S^n_\sigma
\] 
denote the inclusion.

\begin{pr}\label{pr:pullback_deg_sym_gives_deg_twist}
Under Hypothesis~\ref{hyp:basic} in characteristic zero and Hypothesis~\ref{hyp:pc} in positive characteristic, we have the equality
\[
q_\sigma^* \underline{N}_{S,D}^\mfS = i_{\sigma,0}^*\underline{N}_{S,D,\sigma}.
\]
\end{pr}

\begin{proof}
    The double point locus and evaluation map are compatible with pullback. It follows that the pullback under $Q_\sigma^*$ of the orientation away from codimension $1$ determined by the double point pseudo-orientation of $\ev_{\mfS}$ is equivalent to the orientation away from codimension $1$ determined by the double point pseudo-orientation of $\ev_\sigma$. The proposition then follows from Proposition~\ref{pr:deg-pseudo-or-comp-pullback}.
\end{proof}
\begin{rmk}\label{rmk:SymnSurface_not_smooth}
$\Sym^n S$ is not smooth by purity of the branch locus \cite[0BMB]{stacks-project} \cite{Zariski-purity} applied to the quotient map $S^n \to \Sym^n S$. So even though the complement of $\Sym^n_0S\subset \Sym^nS$ is codimension $2$, purity results on Witt groups do not imply that the restriction map $\sGW(\Sym^n S) \to \sGW(\Sym^n_0S)$ is an isomorphism. In fact, the sections $\underline{N}_{S,D}^\mfS \in \sGW(\Sym_0^nS)$ do not extend in general: If they did, Proposition~\ref{pr:pullback_deg_sym_gives_deg_twist} would imply that $\underline N_{S,D,\sigma}$ was independent of $\sigma.$ However, Table~\ref{tab:valuesNSDsigma} and Example~\ref{ex:trksigma} below show that $\underline N_{S,D,\sigma},$ which in the $\Aone$-connected case under consideration is the pull-back of $N_{S,D,\sigma},$ in fact depends on $\sigma.$ For example, one can consider $k = \R$ and then the signature of the trace form $\Tr_{k(\sigma)/k}$ depends on $\sigma.$
\end{rmk}

We now show that Proposition~\ref{pr:pullback_deg_sym_gives_deg_twist} computes $\underline{N}_{S,D}^\mfS$ at every rational point of $\Sym^n_0 S$ in terms of the degrees of the twisted evaluation maps $\underline{N}_{S,D,\sigma}$.

\begin{lm}\label{lm:SsigmatoSsym_Nis_on_rational_points}
    For any rational point $p_0$ of $\Sym^n_0 S $, there exists a finite \'etale $k$-algebra $\sigma$ of degree $n$ and a rational point $p_1$ of $S^n_{\sigma,0}$ such that $$q_\sigma(p_1) = p_0.$$
\end{lm}

\begin{proof}
    We may choose a geometric point $\tilde{p}_0: \Spec \overline{k} \to S^n$ of $S^n$ which has image $p_0$ under the finite \'etale map $S^n_0 \to \Sym^n_0 S$. The geometric point $\tilde{p}_0$ corresponds to an ordered tuple $(x_1,\ldots,x_n)$ of geometric points of $S$. Since $S^n_0 \to \Sym^n_0 S$ is a Galois cover with Galois group $\mathfrak{S}_n$, we may identify the fiber over $\tilde{p}_0$ with the orbit of $(x_1,\ldots,x_n)$ under $\mathfrak{S}_n$. Thus we have a well-defined homomorphism $\Gal(\tilde{p}_0): \Gal(\overline{k}/k) \to \mathfrak{S}_n$ given by
    \[
    g\tilde{p}_0  = \Gal(\tilde{p}_0)(g) (x_1,\ldots,x_n).
    \] The homomomorphism $\Gal(\tilde{p}_0)$ defines an action of $\Gal(\overline{k}/k)$ on the set $\{x_1,\ldots,x_n\}$. Under Galois theory this action corresponds to a finite \'etale $k$-algebra $k \to k(\sigma)$ of degree $n$ and a twist $S^n_{\sigma}$ as in Section~\ref{sssec:twists} such that the the image of the geometric point $(x_1,\ldots,x_n)$ in $S^n_{\sigma}$ is a $k$-point $p_1$. We have $q_{\sigma}(p_1) = p_0$ by construction. 
\end{proof}

\begin{co}\label{co:p0degevmathfrakS}
Suppose $(k,S,D)$ satisfy Hypothesis~\ref{hyp:basic} in characteristic $0$ and Hypothesis~\ref{hyp:pc} in positive characteristic. Let $p_0$ be a rational point of $\Sym^n_0 S$. Choose $\sigma$ and a rational point $p_1$ of $S^n_{\sigma}$ as in Lemma~\ref{lm:SsigmatoSsym_Nis_on_rational_points}. Then
    \[
    p_0^* \underline{N}_{S,D}^\mfS = p_1^* \underline{N}_{S,D,\sigma}.
    \] 
In particular, if $S$ is $\A^1$-connected, we have the equality
    \[
    p_0^* \underline{N}_{S,D}^\mfS = N_{S,D,\sigma}
    \] 
in $\GW(k)$.
\end{co}

\begin{proof}
The first statement follows from combining Proposition~\ref{pr:pullback_deg_sym_gives_deg_twist} and Lemma~\ref{lm:SsigmatoSsym_Nis_on_rational_points}. When $S$ is additionally supposed to be $\A^1$-connected, so is $S^n_{\sigma}$ by Proposition~\ref{A1-connectivity_res_scalars}. Thus $\deg(\ev_{\sigma})$ is in $\GW(k)$ by Definition~\ref{df:GWk-degree-cases} and $p_1^*$ is identified with the identity map.
\end{proof}

In particular, the weighted count of rational curves given by $\underline{N}_{S,D}^\mfS$ through the points corresponding to $p_0$ can be computed from $\underline{N}_{S,D,\sigma}$. In fact, for any point $p_0$ of $\Sym_0^n S$, one can compute $\underline{N}_{S,D}^\mfS$ at $p_0$ using the twisted evaluation maps: any extension of odd degree induces an injection on Grothendieck--Witt groups. Since we are working away from characteristic~$2$, it follows that the perfect closure $k(p_0) \to k(p_0)^{p^{-\infty}}$ induces an injection $\GW(k(p_0)) \to \GW(k(p_0)^{p^{-\infty}})$. Thus we may compute $p_0^* \underline{N}_{S,D}^\mfS $ after base change to $k(p_0)^{p^{-\infty}}$, where we may apply Corollary~\ref{co:p0degevmathfrakS}.

By compatibility of the degree with base change, Proposition~\ref{pr:degree_finite_commutes_basechange} and \eqref{cd:twistsym}, the local degrees of $\ev_{\mfS}$ and $\ev_{\sigma}$ are compatible as well. For example, let $(u: \bbb{P} \to S_{k(u)}, p_1,\ldots, p_n)$ be a point in the intersection of the twisted locus of parametrized curves with only ordinary double points $\bar{M}_{0,n}(S,D)^{\odp}_{\sigma}$ with $\bar{M}_{0,n}(S,D)^{0}_{\sigma}$. By Propositions~\ref{local_deg_stable_base_change} and \ref{degxfrho=degfxix*rho}, the local degree $\deg_{Q_{\sigma}(u)} \ev_{\mfS_n}$ of $ \ev_{\mfS_n}$ at $Q_{\sigma}(u)$ and the local degree of $\ev_{\sigma}$ at $u$ are related by
\[
\deg_{Q_{\sigma}(u)} \ev_{\mfS_n} \otimes_{k(\mfS(u))} k(u) = \deg_{u} \ev_{\sigma}.
\] Note that this holds both in characteristic $p$ and in characteristic $0$.

\section{Geometric formula for local degree of \texorpdfstring{$\ev$}{ev}}\label{sec:geom_formula_local_deg_ev}

Recall the discussion of the discriminant from Section~\ref{sssec:disc} and the notation of Section~\ref{sssec:delpezzo}.
In characteristic zero, let $V$ be as in Theorem~\ref{thm:rel_or_char0}, and in positive characteristic, let $V$ be as in Theorem~\ref{thm:rel_or_charpsf}. Let $\calL_{\sigma}:=[\det(\pi_\sigma)_*\calO_{\dpl_{\sigma}}(-D_\tac)]^{-1}$ and let
\[
\rho: =  \disc_{\pi_{\sigma}}\circ(\det \Tev_{\sigma})^{-1}:\omega_{\ev_{\sigma}}\overset{\sim}\longrightarrow \calL_{\sigma}^{\otimes 2}
\]
be the isomorphism given by Theorem~\ref{thm:rel_or_char0}\ref{item:rel_or_char0:rel_or} in characteristic zero and by Theorem~\ref{thm:rel_or_charpsf}\ref{item:rel_or_charpsf:rel_or} in positive characteristic.

\begin{pr}\label{pr:deg_uev=<disc>}
Let $u$ be a point of $V_\sigma \cap M^\odp_{0,n}(S, D)_\sigma$ over the field $k(u)$. Then the local degree of $\ev_{\sigma}$ at $u$ with respect to the orientation $\rho$ is computed by 
\[
\deg_u \ev = \Tr_{k(u)/k(\ev(u))} \langle \disc_{\pi_{\sigma}}  \rangle.
\]
\end{pr}

\begin{proof}
By construction, $\rho$ is given by the composition
\begin{equation*}
\xymatrix{
\omega_{\ev_{\sigma}} \ar[rr]^{(\det \Tev_{\sigma})^{-1}} && \calO(D_{\cusp,\sigma})  \ar[rr]^(.35){\disc_{\pi_{\sigma}}} && (\det(\pi_\sigma)_*\calO_{\dpl_{\sigma}}(-D_{\tac,\sigma}))^{\otimes 2}.
}
\end{equation*}
Since $D_\cusp \cap M^\odp_{0,n}(S, D)_\sigma = \emptyset = D_\tac \cap M^\odp_{0,n}(S, D)_\sigma,$ the restriction $\rho|_{V_\sigma \cap M^\odp_{0,n}(S, D)_\sigma}$ is given by the composition
\begin{equation}\label{comp_for_rel_or_ev_discprop}
\xymatrix{
\omega_{\ev_{\sigma}} \ar[rr]^(.55){(\det \Tev_{\sigma})^{-1}} && \calO  \ar[rr]^(.35){\disc_{\pi_{\sigma}}} && (\det(\pi_\sigma)_*\calO_{\dpl_{\sigma}})^{\otimes 2}.
}
\end{equation}
Recall that $\Tev_{\sigma}$ denotes the section of $\Hom(\ev_{\sigma}^* T^* (S^n)_\sigma, T^* M^{\odp}_{0,n}(S,D)_\sigma)$ given by the differential of $\ev_{\sigma}$, so the determinant $\det \Tev_{\sigma}$ is a section of $\omega_{\ev_{\sigma}}$. The first isomorphism of \eqref{comp_for_rel_or_ev_discprop} takes $\det \Tev_{\sigma}$ to the constant function $1$, and the last isomorphism takes $1$ to $\disc_{\pi_{\sigma}}$. So, by Definition~\ref{df:Jacobian},
\[
J\ev_\sigma := \rho(\det T\ev_\sigma) = \disc_{\pi_\sigma}.
\]
Since $\ev_\sigma|_{M^\odp_{0,n}(S, D)_\sigma}$ is \'etale by Proposition~\ref{ev:etale_odp}, it follows that $\deg_u \ev_{\sigma} = \Tr_{k(u)/k(\ev(u))} \langle \disc_{\pi_{\sigma}} \rangle$ by Proposition~\ref{degxf=Tr<J>}.
\end{proof}

The discriminant $\langle \disc \pi_{\sigma} \rangle$ of Proposition~\ref{pr:deg_uev=<disc>} has a concrete geometric description. Let $k(u)^s$ denote the separable closure. The point $u$ corresponds to a map $u: C \to S$ from a smooth genus $0$ curve $C \to k(u)$ together with points $p_1,\ldots, p_n$ of $C_{k(u)^s}$ permuted appropriately under the Galois action. Over $k(u)^s$, the image curve $u_{k(u)^s}(C_{k(u)^s})$ has $\delta = \frac{1}{2}D\cdot(K_S + D) + 1$ nodes permuted under the action of $\Gal(k(u)^s/k(u))$. Over $k(u)$ these nodes $p$ consist of points of $S$ with various residue fields $k(p)$. At each node, there are exactly two tangent directions of $u(C)$ defining a field extension $k(p) \subseteq k(p)[\sqrt{D(p)}].$ According to Definition~\ref{df:massnode:intro}, the $\emph{mass}$ of the node $p$ is given by the quadratic form 
\[
\mass(p) = \langle N_{k(p)/k(u)} D(p)\rangle \in GW(k(u)).
\]

\begin{pr}\label{pr:loc_deg_ev_sigma} Let $u$ be a point of $V_\sigma \cap M^\odp_{0,n}(S, D)_\sigma$ over the field $k(u),$ and let $u: C \to S$ be the corresponding map. Let $\Nodes$ denote the set of nodes of $u(C)$. Then
\[
\deg_u \ev_{\sigma} = \Tr_{k(u)/k(\ev_{\sigma}(u))} \prod_{p \in \Nodes } \mass(p) 
\]
\end{pr}

\begin{proof}
We show $\disc \pi_{\sigma}(u) = \prod_{p\in \Nodes} N_{k(p)/k(u)} D(p) $ in $k(u)^*/(k(u)^*)^2$, which is sufficient by Proposition~\ref{pr:deg_uev=<disc>}. Let $\dpl = \dpl_V \overset{\pi}{\to} V$ denote the double-point locus as in Definition~\ref{df:dpl} and let $C_V \to V$ denote the universal curve.
The subscheme $\dpl \subset C_V \times_V C_V$ is invariant under the action of $\Z/2$ switching the factors of $C_V \times_V C_V$, so it inherits an action of $\Z/2$. Let $\calN$ denote the quotient, which will be called the universal node, and let $\calN_\sigma$ denote its Galois twist. We obtain a factorization of $\pi_{\sigma},$ 
\[
\dpl_\sigma \to \calN_\sigma \to V_\sigma.
\]
Pulling back over $\Spec k(u) \to V_\sigma \cap M_{0,n}^{\odp}(S,D)$, the universal node splits as the disjoint union of the nodes $p$ of $u(C).$ Indeed, since $u$ is in the ordinary double point locus, 
\[
\calN_\sigma \otimes k(u) \cong \prod_{p\in \Nodes} \Spec k(p).
\]
The pullback of the double point locus $\calD_\sigma$ over $\Spec k(p)$ is a degree $2$ \'etale cover 
\[
\calD_{k(p)} \to \Spec k(p),
\]
whence of the form $\Spec k(p)[\sqrt{D(p)}] \to \Spec k(p)$ or $\Spec k(p) \coprod \Spec k(p) \to \Spec k(p)$ (a split node) because the characteristic of $k$ is not $2$. Since the discriminant is multiplicative over products of rings, it follows that
\[
\disc \pi_{\sigma}(u) = \prod_{p\in \Nodes} \disc (\calD_{k(p)} \to k(u)).
\] By Lemma~\ref{disc_tower_field_extensions},  \begin{align*}\disc (\calD_{k(p)} \to k(u)) & = \disc(k(p)/k(u))^2 N_{k(p)/k(u)} \disc (\calD_{k(p)} \to \Spec k(p))\\
& = N_{k(p)/k(u)} D(p).
\end{align*}
\end{proof}

We include the following well-known lemma for completeness.

\begin{lm}\label{disc_tower_field_extensions}
Let $K \subset L \subset M$ be a tower of finite degree field extensions. Then,
\[
\disc(M/K) = \disc(L/K)^{[M:L]} \norm{L}{K}(\disc(M/L)).
\]
\end{lm}
\begin{proof}
Let $\{x_i\}_{i \in S}$ be a basis for $L$ over $K$ and let $\{y_j\}_{j \in T}$ be a basis for $M$ over $L$. Define matrices $A,B,$ by
\[
A_{i}^j := \Tr_{L/K}(x_ix_j), \qquad \qquad  B_{i}^j: =\Tr_{M/L}(y_iy_j).
\]
Observe that $\{x_iy_j\}_{i \in S, j \in T}$ is a basis for $M$ over $K.$ Define a matrix $C$ by
\[
C_{ij}^{kl} := \Tr_{M/K}(x_iy_j x_k y_l), \qquad i,k \in S, \quad j,l \in T.
\]
So, we have
\begin{gather*}
\disc(M/K) = [\det(C)] \in K^*/(K^*)^2 , \qquad \disc(M/L) = [\det(B)] \in L^*/(L^*)^2, \\
  \disc(L/K) = [\det(A)] \in K^*/(K^*)^2.
\end{gather*}
Calculate
\[
C_{ij}^{kl} = \Tr_{L/K}(x_ix_k \Tr_{M/L}(y_jy_l)) = \Tr_{L/K}(x_ix_k B_j^l).
\]
Write $B_{j}^{l} x_i = \sum_{m \in S} D_{ij}^{ml}x_m$ for $D_{ij}^{ml} \in K.$ Then
\[
\Tr_{L/K}(x_ix_k B_j^l) = \sum_{m \in S} D_{ij}^{ml} \Tr_{L/K}(x_mx_k) = \sum_{m \in S} D_{ij}^{ml} A_m^k.
\]
Writing the result of the preceding calculation in terms of matrix multiplication, we have
\[
C = D \circ (\id_{[M:L]} \otimes A).
\]
Taking determinants, we obtain
\[
\det(C) = \det(A)^{[M:L]}\det(D) = \det(A)^{[M:L]} N_{L/K}(\det(B)).
\]
A reference for the last equality is~\cite{Kovacs_Silver_Williams}. The lemma follows.
\end{proof}

\section{Enumerative interpretation}\label{sec:enumerative interpretation}

We have defined the degree of twisted evaluation maps on moduli stacks of rational stable maps to del Pezzo surfaces. Over a dense subset of the target, the degree is a sum of local degrees. We gave an interpretation of the local degree at points of the moduli stack corresponding to unramified rational curves whose image has only ordinary double points. Combining, we now give enumerative interpretations of Theorems~\ref{the:intro:deg} and~\ref{the:intro:degpc}.

We continue using the notation of Section~\ref{sssec:delpezzo}. In particular, $S$ is a del Pezzo surface over a perfect field $k,$ $D \in \Pic(S),$ and we
fix a list $\sigma = (L_1, L_2, \ldots, L_r)$ of field extensions $k \subseteq L_i \subseteq \kbar$ such that $\sum_{i=1}^r [L_i: k] = n:=\deg(-D\cdot K_S) -1$.
We include the following well-known lemma, which allows us to view $\ev_{\sigma}$ as a map with codomain $\prod_{i=1}^r \Res_{L_i/k} S$, for completeness. 
\begin{lm} \label{lm:restriction_scalars_is_twist}\label{lm:Srationalkinfinitehaspointsgp}
There is a natural isomorphism 
\begin{equation*}
\phi_\sigma:(S^n)_\sigma \cong \prod_{i=1}^r \Res_{L_i/k} S
\end{equation*} If moreover, $S$ is $k$-rational and $k$ is an infinite field, then any non-empty open subset of $(S^n)_{\sigma}$ contains a rational point.
\end{lm}
The proof is deferred to the end of the section.
 
\begin{df}\label{df:general_position_relative_sigma}
   \label{rmk:generally_chosen_points} 
We say points $p_1,\ldots,p_j \in S$ are a \textbf{$\sigma$-tuple} if there is a field extension $E$ of $k$ such that 
\[
\prod_{i = 1}^{r}L_i \otimes_k E \cong \prod_{i=1}^j k(p_i).
\]
The \textbf{corresponding point} $p_* \in \prod_{i=1}^r \Res_{L_i/k} S$ is the one to which $(p_1,\ldots,p_j)$  maps under the bijection
\[
\prod_{i = 1}^r S(L_i \otimes_k E) \simeq \left(\prod_{i=1}^r \Res_{L_i/k} S\right)(E).
\]
We say a statement holds for $\sigma$-tuples $p_1,\ldots,p_j$ in \textbf{general position} if there exists a dense open subset $U  \subset \prod_{i=1}^r \Res_{L_i/k} S$ independent of $E$ such that the statement holds whenever the corresponding point $p_*$ belongs to $U.$
\end{df}

\begin{rmk}
\label{rmk:rational_points_res} 
    Let $p_i \in S(L_i)$ for $i = 1,\ldots,r$. Then $p_1,\ldots,p_r$ are a $\sigma$-tuple with $E = k.$ 
    A statement holding for such $\sigma$-tuples in general position is the same as the statement holding for $p_1,\ldots,p_r$ in general position as in Section~\ref{ssec:enummean}.
\end{rmk}

Recall the quadratically enriched count of rational curves $\underline N_{S,D,\sigma}$ from Definition~\ref{df:qecc}.

\begin{tm}\label{thmintro:sGWRes_curve_count}
Under Hypothesis~\ref{hyp:basic} when $\Char k = 0$ and Hypothesis~\ref{hyp:pc} when $\Char k > 0,$ for $\sigma$-tuples of points $p_1,\ldots,p_j \in S$ in general position corresponding to $p_* \in \prod_{i=1}^r \Res_{L_i/k} S,$ we have the equality in $\GW(k(p_*)),$ 
\[
\underline{N}_{S,D, \sigma}(p_*) =  \sum_{\substack{u \text{ rational curve} \\ \text{ on } S\\  \text{ in class } D\\\text{ through the points} \\ p_1, \ldots, p_j}}  \Tr_{k(u)/k(p_*)} \prod_{p \text{ node}\text{ of }u(\bbb{P}^1)}\operatorname{mass}(p).
\]
Moreover, if $S$ is also $k$-rational and $k$ is an infinite field, there exists such $p_1,\ldots,p_j$ with $j=r$ and $k(p_i) = L_i$ for $i=1,\ldots,r$.
\end{tm}

\begin{proof}
By Theorems~\ref{tm:eodpd},\ref{thm:rel_or_char0} and~\ref{thm:rel_or_charpsf}, there is a dense open subset $U\subset \prod_{i=1}^r \Res_{L_i/k} S$ such that for all $p_* \in U,$ we have $\ev_{\sigma}^{-1}(p_*) \subset V_\sigma \cap M_{0,n}(S,D)^{\odp}_{\sigma}$. 
By Proposition~\ref{pr:deg=sumlocal} and Definition~\ref{df:rel_or_char0} when $\Char k = 0$, or Definitions~\ref{df:rel_or_charpsf} and~\ref{df:rel_por} when $\Char k > 0,$ we have 
\[
\underline{N}_{S,D, \sigma}(p_*) = \sum_{u \in \ev_{\sigma}^{-1} (p_*)} \deg_u \ev_{\sigma}, \qquad p_* \in U.
\]
By construction, $\ev_{\sigma}^{-1} (p_*)$ is the set of rational curves $u$ on $S$ in class $D$ passing through the points $p_1,\ldots,p_j$. So, by Proposition~\ref{pr:loc_deg_ev_sigma}, the local degree is given by 
\[
\deg_u \ev_{\sigma} = \Tr_{k(u)/k(p_*)} \prod_{p \in \Nodes } \mass(p),
\]
as claimed. 

Suppose that $S$ is also $k$-rational and $k$ is an infinite field. Then by Lemma~\ref{lm:restriction_scalars_is_twist}, $U$ contains a rational point $p$, completing the proof.
\end{proof}

\begin{co}\label{cor:intro:countA1DelPezzo}
In the notation of Theorem~\ref{thmintro:sGWRes_curve_count}, suppose that $S$ is additionally $\A^1$-connected. There is $N_{S,D, \sigma}$ in $\GW(k)$ such that 
for $\sigma$-tuples of points $p_1,\ldots,p_j \in S$ in general position corresponding to $p_* \in \prod_{i=1}^r \Res_{L_i/k} S,$ we have the equality in $\GW(k(p_*)),$ 
\[
N_{S,D, \sigma}\otimes_k k(p_*)=  \sum_{\substack{u \text{ rational curve} \\ \text{ on } S\\  \text{ in class } D\\\text{ through the points} \\ p_1, \ldots, p_j}}  \Tr_{k(u)/k(p_*)}\prod_{p \text{ node}\text{ of }u(\bbb{P}^1)}\operatorname{mass}(p)
\] 
where $N_{S,D, \sigma} \otimes_k k(p_*)$ denotes the image of $N_{S,D, \sigma}$ under the pullback map $\GW(k) \to \GW(k(p_*))$. Moreover, if $S$ is also $k$-rational and $k$ is an infinite field, we may take $p_1,\ldots,p_j$ such that $k(p_*)=k$.
\end{co}

\begin{proof}
 When $S$ is $\A^1$-connected, the twisted product $\prod_{i=1}^r \Res_{L_i/k} S$ is as well by Proposition~\ref{A1-connectivity_res_scalars}. By Corollary~\ref{pr:A1connected_implies_GW=GW(k)}, the section $\underline{N}_{S,D, \sigma}$ of Definition~\ref{df:qecc} is pulled back from a unique element $N_{S,D, \sigma}$ in $\GW(k)$, which has the claimed property by Theorem~\ref{thmintro:sGWRes_curve_count}.

\end{proof}

\begin{rmk}
By construction, the invariants $\deg (\ev_{\sigma})$ in $\GW(k)$ only depend on the list of field extensions $\sigma=(L_1,\ldots,L_r)$. Thus the count of the degree $D$ rational curves through generally chosen points with residue fields $L_1,\ldots,L_r$ is independent of the chosen points. This strengthens \cite[Example 3.9]{Levine-Welschinger} where this statement is proven for $[L_i:k] \leq 2$.
\end{rmk}

\begin{proof}[Proof of Lemma \ref{lm:restriction_scalars_is_twist}] 

By the universal property of the restriction of scalars, it suffices to define for each finitely generated $k$-algebra $A$, a natural isomorphism
\[
\phi_A: \prod_{i=1}^rS(L_i\otimes_kA)\to (S^n)_\sigma(A) .
\] Choose a finite Galois extension $E$ of $k$ with $k \subseteq L_i \subseteq E \subseteq \bar{k}$ for each $i$. Let $G= \Gal(E/k)$. It follows from the theory of descent that $(S^n)_\sigma(A)$ is naturally in bijection with the $G$-invariant subset of  $S^n(E\otimes_k A)$ for the $G$-action associated to the cocycle of Section~\ref{sssec:twists}. Recall that in Section~\ref{sssec:twists}, we identify the set of ring homomorphisms $\prod_i L_i \to \bar{k}$ with the set $\{1,\ldots,n\}$. The $G$-action on the set of such ring homomorphisms then defines 
\[
\Gal(\sigma): G \to \mathfrak{S}_n
\] where $\mathfrak{S}_n$ denotes the symmetric group.

Since $S$ is defined over $k$, $G$ acts on $S(E \otimes_k A)$ by sending $x : \Spec E \otimes_k A \to S$ to $gx$ defined
\[
\Spec E \otimes_k A \xrightarrow{g^{-1} \otimes 1}\Spec E \otimes_k A \xrightarrow{x} S.
\] Let $x_i$ be elements of $S(E \otimes_k A)$ for $i=1,\ldots,n$. The $G$-action on $S^n(E\otimes_k A)$ is given
\[
g(x_1,\ldots, x_n) = (g x_{\Gal(\sigma)(g)^{-1}(1)},\ldots,g x_{\Gal(\sigma)(g)^{-1}(n)}).
\]

Note that any ring homomorphism $\prod_{i=1}^r L_{i} \to \bar{k}$ factors through a unique $L_i$. Since we have identified the set of such ring homomorphisms with $\{1,\ldots,n\}$, we may define a map
\[
e: \{1,\ldots,n\} \to \{1,\ldots,r\}.
\] associating to each $\prod_{i=1}^r L_{i} \to \bar{k}$ this unique $L_{i}$. Suppose $y_i \in S(L_i \otimes_k A)$ for $i=1,\ldots,r$. For each $j \in e^{-1}(i)$, there is an associated map $$ \prod_{i=1}^r L_i \to L_i \xrightarrow{f_j} E \to \bar{k},$$ because $E$ is normal. Define $(x_1,\ldots,x_n)$ in $S^n(E \otimes_k A)$ by $x_i = (f_i \otimes 1_A)^* y_{e(i)}$. For any $g \in G$, we have by definition
\[
g f_i = f_{\Gal(\sigma)(g)(i)}.
\] Thus 
\[
g^{-1} x_{\Gal(\sigma)(g)(i)} = g^{-1}(f_{\Gal(\sigma)(g)(i)} \otimes 1_A)^* y_{e({\Gal(\sigma)(g)(i)})} = (f_i \otimes 1_A)^* y_{e(i)} = x_i
\] where the second equality follows because $e({\Gal(\sigma)(g)(i)}) = i$ by construction. Thus $(x_1,\ldots,x_n)$ is $G$-invariant, and we may define $\phi_A (y_1,\ldots,y_r)$ to be the point of $(S^n)_{\sigma} (A)$ corresponding to $(x_1,\ldots,x_n)$.

Since $(f_i \otimes 1_A)^* y_{e(i)}$ determines $y_{e_i}$, $\phi_A$ is injective. To show surjectivity, let $(x_1,\ldots,x_n)$ be a $G$-invariant element of $S^n(E \otimes_k A)$. For each $i=1,\ldots,r$, let $j(i)$ be the element of $\{1,\ldots,n\}$ corresponding to the ring homomorphism
\[
\prod_{i=1}^r L_i \to L_i \to E \to \overline{k}
\] where the map $L_i \to E$ corresponds to the original inclusion $L_i \subseteq E$. Then $
\phi_A(x_{j(1)},\ldots,x_{j(r)})$ gives the point of $(S^n)_{\sigma}(A)$ corresponding to $(x_1,\ldots,x_n)$ by unwinding definitions and $\phi_A$ is surjective as desired.

When $S$ is $k$-rational and $k$ is infinite, any non-empty open subset of $\Res_{L_i/k} S$ contains a rational point (see for example \cite[Proposition 2.5]{BrugalleWickelgren-AB}) completing the proof. 
\end{proof}

\section{Examples}\label{section:examples}

\subsection{\texorpdfstring{$N_{S,-K_S,\sigma}$}{Invariant of anti-canonical class} for \texorpdfstring{$d_S \geq 3$}{del Pezzo degree greater than three}}

Let $S$ be a del Pezzo surface over a perfect field $k$ of degree $d_S \geq 3$. We compute the enriched counts of rational curves on $S$ in anticanonical degree $D= -K_S$. Let $n=-K_S \cdot D -1 = d_S-1$.  Let $\sigma = (L_1,\ldots, L_r)$ be an $r$-tuple of field extensions with $\sum_{i=1}^r [L_i:k]=n$. Let $k(\sigma)$ denote the corresponding product of fields \begin{equation*}
k(\sigma) = \prod_{i=1}^r L_i,
\end{equation*} which gives a finite \'etale extension $k \to k(\sigma)$. Recall that $\chi^{\A^1}(S)$ denotes the $\A^1$-Euler characteristic of $S$. The main result of the present section is the following proposition, the proof of which is given below.
\begin{pr}\label{pr:NS-KSsigma}\label{ex:NS-KSsigma}
Let $S$,$k$, and $k(\sigma)$ be as above with $S$ additionally assumed $\A^1$-connected. Then in $\GW(k)$ we have the equality

\[
N_{S,-K_S,\sigma} =   \langle -1 \rangle\chi^{\A^1}(S)  + \langle 1 \rangle +  \Tr_{k(\sigma)/k} \langle 1 \rangle.
\] 
\end{pr}

Since $d_S \geq 3$, a choice of basis for $H^0(S, -K_S)$ determines an embedding $S \hookrightarrow \bbb{P}_k^{d_S}$. A point $p$ of $\Res_{k(\sigma)/k}S$ corresponds to some points $(p_1,\ldots,p_{r'})$ of $S$ such that $k(p) \otimes_k k(\sigma) \cong \prod_{i=1}^{r'} k(p_i)$ by the definition of the restriction of scalars. For a general point of   $\Res_{k(\sigma)/k}S$, the linear conditions on an element of $H^0(\bbb{P}^{d_S}, \calO(1))$ corresponding to the element vanishing at $p_i$ for $i=1,\ldots,r'$ are independent. Choose such a $p=(p_1,\ldots,p_{r'})$ in $\Res_{k(\sigma)/k}S$.

Thus $$\{ f \in H^0(\bbb{P}^{d_S}, \calO(1)): f(p_i) = 0 \text{ for }i=1,\ldots r'\}$$ is a $2$-dimensional vector space over $k(p)$. Choose a basis $\{f,g\}$ and let
\[
X   = \{ [s,t] \times x : tf(x) + sg(x) = 0  \} \subset \bbb{P}^1 \times S
\] be the corresponding pencil over $k(p)$. The base locus $B=\{ f = g = 0\} \hookrightarrow S$ of the pencil has degree $d_S$ by B\'ezout's theorem. By construction, the points $p_i$ lie in $B$, whence $B = \{ p_1,\ldots,p_{r'}\} \cup \{p_0\}$ where $p_0$ is a $k(p)$-rational point of $S_{k(p)}$.

Let $\pi: X \to \bbb{P}^1_{k(p)}$ denote the projection. By construction, the fibers of $\pi$ are precisely the curves in class $-K_S$ passing through $\{ p_1,\ldots,p_i\}$. (These all then also pass through $p_0$.)

Let $C \hookrightarrow S$ be a fiber of $\pi$.  By adjunction, $C$ has canonical class $K_C = K_S \otimes \calO(C)$. Since $C$ is in class $-K_S$, we have $\calO(C) \cong -K_S$, whence $K_C = \calO$ and $C$ has arithmetic genus $1$. It follows that for a general choice of point in $\Res_{k(\sigma)/k} S$, the fibers of $\pi$ are either smooth or rational with a single node. Thus
\begin{equation}\label{eq:NSDsigma_in_anticanonical_case} N_{S,D,\sigma} \otimes k(p) = \sum_{\substack{u \text{ rational curve} \\ \text{ on } S_{k(p)}\\  \text{ in class } -K_S\\\text{ through the points} \\ p_1, \ldots, p_{r'}}} \Tr_{k(u)/k(p)}\operatorname{mass}(q(u)),
\end{equation} where $q(u)$ denotes the node of $u(\bbb{P}^1)$. We will compute the right hand side directly using the $\A^1$-Euler characteristic.

The projection $X \to S$ realizes $X$ as the blow-up
\[
X \cong \Bl_B S.
\] It follows from \cite[Prop 1.4]{Levine-EC} that the $\A^1$-Euler characteristic $\chi^{\A^1}(X)$ is computed\begin{align*}\chi^{\A^1}(X) & = \chi^{\A^1}(S) + (\chi^{\A^1}(\bbb{P}^1) - \langle 1 \rangle ) \chi^{\A^1} (B )\\
& =  \chi^{\A^1}(S) + \langle -1 \rangle \chi^{\A^1} (\{ p_0,\ldots,p_i\} ) \\
& = \chi^{\A^1}(S)  + \langle -1 \rangle + \langle -1 \rangle \chi^{\A^1} (\{ p_1,\ldots,p_i\} )  \\
& = \chi^{\A^1}(S)  + \langle -1 \rangle +  \sum_{i=1}^i\langle -1 \rangle \Tr_{k(p_i)/k(p)} \langle 1 \rangle.
 \end{align*} Since $\sum_{i=1}^i\langle -1 \rangle \Tr_{k(p_i)/k(p)} \langle 1 \rangle$ is the pullback to $\GW(k(p))$ of $ \langle -1 \rangle \Tr_{k(\sigma)/k} \langle 1 \rangle$, it follows that \begin{equation}\label{chiX_from_blowupS}
\chi^{\A^1}(X) =  \chi^{\A^1}(S)  + \langle -1 \rangle +  \langle -1 \rangle \Tr_{k(\sigma)/k} \langle 1 \rangle
\end{equation} in $\GW(k(p))$.

We can also calculate $\chi^{\A^1}(X) $ using $\pi$ and the work of the second named author \cite[Section 10]{Levine-EC}. Comparing the two will compute $N_{S,-K_S,\sigma} \otimes k(p)$. An isomorphism $T \bbb{P}^1 \cong \calO(1)^{\otimes 2}$ defines a relative orientation of $\Hom(\pi^* T^* \bbb{P}^1, T^* X)$, where $T^*X$ denotes the cotangent bundle, or K\"ahler differentials. Thus there is a well-defined Euler number $e(\Hom(\pi^* T^* \bbb{P}^1, T^* X))$ in $\GW(k(p))$. The morphism $\pi$ determines a section $T^* \pi$ of the vector bundle $\Hom(\pi^* T^* \bbb{P}^1, T^* X)$ and the Euler number can be computed as a sum \[e(\Hom(\pi^* T^* \bbb{P}^1, T^* X)) = \sum_{x: T^* \pi(x) = 0} \ind_x T^* \pi.\] See \cite[Section 1]{Levine-EC} or \cite[Section 4]{cubicsurface} and \cite{bachmann_wickelgren} for compatibility checks. In the Witt group $W(k(p)):= \GW(k(p))/\bbb{Z}(\langle 1 \rangle + \langle -1 \rangle)$, we have equalities \begin{align}\label{chiXfrompencilcubicsS}
\chi^{\A^1}(X) & = e( T^* X ) = e(\Hom(\pi^* T^* \bbb{P}^1, T^* X)) = \sum_{x: T^* \pi(x) = 0} \ind_x T^* \pi,
\end{align} where the first equality is \cite[Theorem 3.1, Theorem 7.1]{Levine-EC} and the second is \cite[Theorem 9.1]{Levine-EC}. Comparison with the classical computation (where $\chi$ is multiplicative and the general fiber has Euler characteristic $0$), shows \eqref{chiXfrompencilcubicsS} is also valid in $\GW(k(p))$.

\begin{lm}\label{inddpi_S}
For general $(p_1,\ldots,p_r)$, the zeros of $T^* \pi$ are the nodes in the fibers of $\pi$, and for a node $q$, the local index $\ind_q T^* \pi$ is computed $ \ind_q T^* \pi= \Tr_{k(q)/k(p)} (\langle -1 \rangle \mass(q))$.
\end{lm}

\begin{proof}
Let $U\subset X$ denote the open subset of the pencil given by $U = \{s \neq 0\} \subseteq \bbb{A}^1 \times S$ and let $t$ be the coordinate on $\bbb{A}^1$. Choose $(t,q)$ in $U$, and local analytic coordinates $(x,y)$ on $S$ for the completion of $\calO_S$ at $q$. In these coordinates, $U$ is given by $$\{ t \times (x,y) : t F(x,y) + G(x,y) = 0\}.$$ The point $(t,q)$ is a zero of $T^* \pi$ if and only if $\pi^* dt (t,q)= 0$. Since $t F(x,y) + G(x,y) = 0$, we have that \begin{equation}\label{dtfromdxdy}dt F + t \partial_x F dx + t \partial_y F dy + \partial_x G dx + \partial_y G dy = 0\end{equation} at $q$. Since we assume the pencil is smooth, we can not have $F = t \partial_x F + \partial_x G =  t \partial_y F + \partial_y G = 0$. It follows that $\pi^* dt (t,q)= 0$ if and only if  $$(t \partial_x F + \partial_x G) dx + (  t \partial_y F + \partial_y G ) dy = 0,$$ which occurs if and only if $t \partial_x F + \partial_x G =  t \partial_y F + \partial_y G = 0$. This latter condition occurs if and only if $q$ is a node of $tF + G = 0$. Without loss of generality, we may assume that a zero of $T^*\pi$ is in $U$. Thus the zeros of $T^* \pi$ are the nodes in the fibers of $\pi$ as claimed. Note that we have also shown that if $q$ is a node in the fiber at $t$, then $F(q) \neq 0$.

Consider a zero $(t,q)$ in $U$ of $T^*\pi$. Consider the local trivialization of $\Hom(\pi^* T^* \bbb{P}^1, T^* X )$ corresponding to the basis $\{dt \mapsto dx, dt \mapsto  dy \}$. This local trivialization is compatible with the local coordinates and the canonical relative orientation of $\Hom(\pi^* T^* \bbb{P}^1, T^* X )$ (coming from the orientability of $\bbb{P}^1$). Using these local coordinates and local trivialization, the section $T^*\pi$ corresponds to $$\left( \frac{t \partial_x F + \partial_x G}{F} , \frac{  t \partial_y F + \partial_y G }{F} \right)$$ because $T^*\pi (dt) = (t \partial_x F + \partial_x G)/F dx + (  t \partial_y F + \partial_y G )/F dy$ by \eqref{dtfromdxdy}. Since $t = -G/F$, the section $T^*\pi$ likewise corresponds to the function $$\left( \frac{-G \partial_x F + F \partial_x G}{F^2} , \frac{  -G \partial_y F +F \partial_y G }{F^2} \right) = \left(\partial_x \frac{G}{F}, \partial_y \frac{G}{F}\right).$$ Note that the Jacobian determinant of $(\partial_x \frac{G}{F}, \partial_y \frac{G}{F})$ equals the Hessian $\Hess(G/F)$. We will show this Jacobian determinant does not vanish and $k(p) \subseteq k(q)$ is separable, which then implies that $\ind_q T^* \pi = \Tr_{k(q)/k(p)} \langle \Hess(G/F)(q) \rangle$ by \cite[Proposition 34]{cubicsurface}. 

The Hessian $\Hess(G/F)$ as a function of $x$ and $y$ equals the Hessian $\Hess(t + G/F)$ because $t$ is a fixed scalar. Moreover, since $t+ G/F$ and its partials vanish at $q$, there is an equality between $\Hess(t F + G)$ and $\Hess(t + G/F)$ evaluated at $q$ by the chain rule. By genericity, the fibers $\{(x,y): t F + G = 0 \}$ of $\pi$ have only nodes, whence $\Hess(t F + G) (q) \neq 0$. Thus $\Hess(G/F)(q)\neq 0$. Since $q$ is a node, the extension $k(p) \subseteq k(q)$ is separable \cite[Expos\'e XV, Th\'eor\`eme 1.2.6]{SGA7_10_22} and we have that $\ind_q T^* \pi = \Tr_{k(q)/k(p)} \langle \Hess(G/F)(q) \rangle$, whence $\ind_q T^* \pi = \Tr_{k(q)/k(p)} \langle \Hess(t F + G) (q)\rangle$, proving the claim.
\end{proof}

\begin{proof}[Proof of Proposition~\ref{pr:NS-KSsigma}]
Recall that a point $p$ of $\Res_{k(\sigma)/k} S$ with residue field $k(p)$ corresponds to points $(p_1,\ldots,p_j)$ of $S$ with residue fields $E_1,\ldots,E_j$ such that $k(\sigma) \otimes_k k(p) \cong \prod_{i=1}^j E_j$. The condition that a section of $H^0(\bbb{P}^{d_S}, \cO(1))$ vanish at some $p_i$ consists of $[k(p_i):k]$ linear conditions over $k$. Let $U$ be the dense open subset of $\Res_{k(\sigma)/k} S$ such that for all points $p$ of $U$ the $n$ linear conditions on $H^0(\bbb{P}^{d_S}, \cO(1))$ corresponding to the vanishing of a section at $p = (p_1,\ldots, p_j)$ consists of $n$ linearly independent conditions. Passing to a potentially smaller dense open subset $k(U')$ of $\Res_{k(\sigma)/k} S$, we may assume that the corresponding pencil $X \to \bbb{P}^1$ has only nodal singularities. Combining \eqref{chiX_from_blowupS}, \eqref{chiXfrompencilcubicsS} and Lemma~\ref{inddpi_S} we have that the proposition holds after pullback to $k(p)$ for any $p$ in $U'$. Since $S$ is $\A^1$-connected, so is $\Res_{k(\sigma)/k} S$ by Proposition~\ref{A1-connectivity_res_scalars}. Thus $\GW(k) \to \sGW(\Res_{\sigma/k} S)$ admits a splitting and is thus injective. Since $U'$ is dense in $\sGW(\Res_{k(\sigma)/k} S)$, we have that
\[
\sGW(U') \to \sGW(\Res_{k(\sigma)/k} S)
\] is likewise injective. It follows that $\GW(k) \to \sGW(U') $ is injective, so we must have the claimed equality.

\end{proof}

\begin{rmk}\label{Nanticanonical_SA1connected}
The hypothessis that $S$ is $\A^1$-connected is not necessary in this computation of $N_{S,-K_S,\sigma}$. If $S$ is $\A^1$-connected, then we have defined $N_{S,-K_S,\sigma}$ in $\GW(k)$ by Theorem \ref{the:intro:deg} or \ref{the:intro:degpc}. In particular, the statement of Proposition~\ref{pr:NS-KSsigma} is clearly well-defined. More generally, however, we have defined $\underline{N}_{S,D,\sigma}$ in $\sGW(\pi_0^{\A^1}(\Res_{k(\sigma)/k} S))$ with Theorem~\ref{the:intro:degnc} or ~\ref{the:intro:degpcnc} . We see from the proof of Proposition~\ref{pr:NS-KSsigma} that the element of $\GW(k)$ given by $  \langle -1 \rangle\chi^{\A^1}(S)  + \langle 1 \rangle +  \Tr_{k(\sigma)/k} \langle 1 \rangle$ has the property that its pullback to $\sGW(\pi_0^{\A^1}(\Res_{k(\sigma)/k} S))$ is $\underline{N}_{S,D,\sigma}$ .
\end{rmk}

\begin{rmk}
    The hypothesis that $k$ is perfect is also not necessary for the proof of Proposition~\ref{pr:NS-KSsigma}. As in the case of Remark~\ref{Nanticanonical_SA1connected}, it is necessary to assume Hypothesis~\ref{hyp:pc} in order to define $N_{S, -K_S, \sigma}$. Alternatively, however, one could use the right hand side of Theorem~\ref{the:intro:curve_count}, or more precisely the right hand side of \eqref{eq:NSDsigma_in_anticanonical_case}, to define $N_{S, -K_S, \sigma} \otimes k(p)$. Apriori, this is not well-defined because it might depend on the points $p =(p_1,\ldots,p_j)$. However, the proof of Proposition~\ref{pr:NS-KSsigma} shows that when we drop the assumption of perfectness from Hypothesis~\ref{hyp:pc}, there is nonetheless a dense open subset of $\Res_{k(\sigma)/k} S$ for which this definition is independent of the choice and is equal to $\langle -1 \rangle\chi^{\A^1}(S)  + \langle 1 \rangle +  \Tr_{k(\sigma)/k} \langle 1 \rangle.$
\end{rmk}

\begin{ex}\label{ex:SnotA1connected}
    Let $S$ be a smooth scheme over $k \subseteq \bbb{R}$. If $S$ is $\A^1$-connected, then $S(\bbb{R})$ is connected. This gives many examples of del Pezzo surfaces which are not $\A^1$-connected. For example, the cubic surface $S_{\lambda}$ in $\bbb{P}^4$ determined by the equations
    \[
    \lambda x^3+y^3+z^3+w^3+t^3 = 0 \quad \quad x+y+z+w+t = 0
    \] has $S(\bbb{R})$ disconnected for $ 1/16< \lambda <1/4$ (\cite[Example 3.1]{Polo-Blanco-Jaap}) and is therefore not $\A^1$-connected. See also \cite{Degtyarev-Kharlamov-quasi-simple}.
\end{ex}

\subsection{Enriched numbers}

We can give entirely explicit calculations of some $\A^1$ Gromov--Witten invariants. See Table~\ref{tab:valuesNSDsigma}. Powerful tropical methods have produced many more calculations \cite{PuentesPauli-Correspondence} \cite{PuentesMarkwigPauliRorhle-quad_extensions} after this work originally appeared. See also \cite{BrugalleWickelgren-AB}. We justify the calculations in the table: The $\A^1$-Euler characteristic of $\bbb{P}^2$ is $\langle 1\rangle+\langle -1\rangle+\langle 1\rangle$ by \cite[Example 1.7]{Hoyois_lef}. Combining with Proposition~\ref{ex:NS-KSsigma} gives that $N_{\bbb{P}^2, \cO(3),\sigma} = 2(\langle 1\rangle+\langle -1\rangle)+\Tr_{k(\sigma)/k}\langle 1\rangle $. One similarly computes that $\chi^{\A^1}(\bbb{P}^1 \times \bbb{P}^1) = 2(\langle 1\rangle+\langle -1\rangle)$. The $\A^1$-Euler characteristic of hypersurfaces such as the Fermat cubic surface and $xy^2+y^2z+z^2w+w^2x = 0$ can be computed explicitly using \cite{LLV-Eulerchar} and the resulting $\A^1$-Gromov--Witten invariants are given in the table. The remaining entries of Table~\ref{tab:valuesNSDsigma} are all of the following form: the classical Gromov--Witten invariant over $\mathbb{C}$ is $1$ because there is precisely $1$ curve on the surface $S$  through $-K_S \cdot D -1$ generally chosen points in these cases and, moreover, the image of this curve is smooth. It follows that the corresponding $\A^1$-Gromov--Witten invariant is $\langle 1 \rangle$. For example, the complete linear system associated to $\calO(1) \boxtimes \calO(d)$ on $\bbb{P}^1 \times \bbb{P}^1$ defines an embedding
\[
\bbb{P}^1 \times \bbb{P}^1 \to \bbb{P}^{2d+1}
\] not contained in any hypersurface. For $[\sigma:k]=-K_S \cdot (\calO(1) \boxtimes \calO(d)) -1 = 2d+1$, there is an open dense subset of $\Res_{\sigma/k}(\bbb{P}^1 \times \bbb{P}^1)$ consisting of points $p = (p_1\ldots, p_r)$ so that over the algebraic closure the $p_i$ debtermine a unique hypersurface $H$ in $\bbb{P}^{2d+1}$. The intersection of $H$ with the image of $\bbb{P}^1 \times \bbb{P}^1 $ is isomorphic to $\bbb{P}^1$ by projection on the first factor in $\bbb{P}^1 \times \bbb{P}^1 $. Since this curve is smooth, the discriminant of the double point locus is trivial. Note that $\bbb{P}^1 \times \bbb{P}^1$ and $\bbb{P}^2$ are $k$-rational, so we have rational points of $\Res_{\sigma/k} S$ in general position, at least after an extension of fields of odd degree (which may be necessary in $k$ is finite). Since odd degree field extension induce an injection on Grothendieck--Witt groups, it follows that the resulting $\A^1$-Gromov--Witten invariants are $\langle 1\rangle$ by Theorem~\ref{the:intro:curve_countnc}. The table entries $N_{\bbb{P}^2, \cO(1), \sigma}$ and $N_{\bbb{P}^2, \cO(2), \sigma}$ are  $\langle 1\rangle$ by a very similar argument.

Note that $\Tr_{k(\sigma)/k} \langle 1 \rangle = \sum_{i=1}^r \Tr_{L_i/k} \langle 1 \rangle$ is the sum of the trace forms of the field extensions $k \subseteq L_i$. The trace forms  $\Tr_{k(\sigma)/k} \langle 1 \rangle$ can be computed explicitly as well. 
\begin{ex}\label{ex:trksigma}
For $k$ of characteristic not $2$, $d \in k^*$, and $\sigma = (k,k,\ldots,k,k[\sqrt{d}])$, we have $$\Tr_{k(\sigma)/k} \langle 1 \rangle = (n-2)\langle 1 \rangle + \langle 2 \rangle + \langle 2d \rangle.$$ 
\end{ex}
More information on trace forms is available. See for example \cite{Conner-Perlis}. 

\begin{rmk}
Note that for appropriate $\sigma$ and $S$, many of the $N_{S,D,\sigma}$ in Table~\ref{tab:valuesNSDsigma} are not only sums of $\langle \pm 1\rangle$'s. For example, it is impossible to express $\langle -3 \rangle+4(\langle 1\rangle+\langle -1\rangle)+\langle 1 \rangle+\Tr_{k(\sigma)/k}\langle 1\rangle $ for $\sigma=(k,k,\ldots,k)$ in the form $a\langle 1\rangle + b \langle -1\rangle$ for $a$ and $b$ integers because the discriminant is divisible by $3$. A lot more is happening here than over $\bbb{R}$ and $\bbb{C}$!
\end{rmk}

\begin{table}
\label{tab:valuesNSDsigma}       
\begin{tabular}{llll}
\hline\noalign{\smallskip}
 $S$&$D$& $\sigma$ & $N_{S,D,\sigma}=$ count of rational curves \\
\noalign{\smallskip}\hline\noalign{\smallskip}
 $\bbb{P}^2$ & $\calO(1)$ & all $\sigma$ & $\langle 1 \rangle$  \\
$\bbb{P}^2$ & $\calO(2)$ & all $\sigma$ & $\langle 1 \rangle$  \\
$\bbb{P}^2$ & $\calO(3)$ & all $\sigma$ &  $2(\langle 1\rangle+\langle -1\rangle)+\Tr_{k(\sigma)/k}\langle 1\rangle $  \\
 $\bbb{P}^1 \times \bbb{P}^1$ & $\calO(1) \boxtimes \calO(d)$ & all $\sigma$ & $\langle 1 \rangle$  \\
  $\bbb{P}^1 \times \bbb{P}^1$ & $\calO(2) \boxtimes \calO(2)$ & all $\sigma$ & $2(\langle 1\rangle+\langle -1\rangle)+\langle 1 \rangle+\Tr_{k(\sigma)/k}\langle 1\rangle $  \\
  Fermat Cubic Surface & $\calO_{\bbb{P}^3}(1)$&  all $\sigma$ & $\langle -3 \rangle+4(\langle 1\rangle+\langle -1\rangle)+\langle 1 \rangle+\Tr_{k(\sigma)/k}\langle 1\rangle $ \\
  $x^3+y^3+z^3+w^3 = 0$& & & \\
  the cubic surface & $\calO_{\bbb{P}^3}(1)$&  all $\sigma$ & $\langle 5 \rangle+4(\langle 1\rangle+\langle -1\rangle)+\langle 1 \rangle+\Tr_{k(\sigma)/k}\langle 1\rangle $ \\
  $xy^2+y^2z+z^2w+w^2x = 0$& & & \\
  $S$&$-K_S$&all $\sigma$& $\langle -1 \rangle\chi^{\A^1}(S)  + \langle 1 \rangle +  \Tr_{k(\sigma)/k} \langle 1 \rangle$\\
\noalign{\smallskip}\hline
\end{tabular}
\newline
\caption{$\GW(k)$-enriched counts of rational curves}
\end{table}

\bibliographystyle{alpha}

\bibliography{Bibli}

\end{document}